\documentclass{amsart}

\usepackage{amssymb}

\usepackage{amsmath} %Command added by article author
\usepackage{enumitem} %Command added by article author
\usepackage{hyperref} %Command added by article author

\usepackage{txfonts} %Command added by article author
\usepackage{tikz} %Command added by article author
\usepackage{stmaryrd} %Command added by article author

\RequirePackage{latexsym} %Command added by article author
\newtheorem{Theorem}{Theorem}  %Command added by article author
\newtheorem*{Theorem*}{Theorem}  %Command added by article author
\newtheorem*{Lemma}{Lemma} %Command added by article author
\newtheorem*{Definition}{Definition} %Command added by article author
\newtheorem{theorem}{Theorem}[section]
\newtheorem{lemma}[theorem]{Lemma}
\newtheorem{proposition}[theorem]{Proposition} %Command added by article author
\newtheorem{corollary}[theorem]{Corollary} %Command added by article author
\newtheorem{fact}[theorem]{Fact} %Command added by article author

\newtheorem*{claim}{Claim}%Command added by article author
\newtheorem*{subclaim}{Subclaim} %Command added by article author

\theoremstyle{definition}
\newtheorem{definition}[theorem]{Definition}

\theoremstyle{remark}
\newtheorem{remark}[theorem]{Remark}

\newtheorem*{claimproof}{Proof of Claim} %Command added by article author
\newtheorem*{subclaimproof}{Proof of Subclaim} %Command added by article author

\numberwithin{equation}{section} %Command added by article author

\newcommand{\Ord}{\ensuremath{\text{{\rm Ord}}}}

\newcommand{\TR}{\ensuremath{\text{{\rm TR}}}}
\newcommand{\ZFC}{\ensuremath{\text{{\sf ZFC}}}}

\newcommand{\CH}{\ensuremath{\text{{\sf CH}}}}

\DeclareMathOperator{\dom}{dom}

\DeclareMathOperator{\coker}{coker}
\DeclareMathOperator{\val}{val}

\DeclareMathOperator{\crit}{crit}
\DeclareMathOperator{\rng}{rng}

\DeclareMathOperator{\otp}{otp}

\DeclareMathOperator{\cf}{cf}
\DeclareMathOperator{\cof}{cof}

\DeclareMathOperator{\rank}{rank}

\DeclareMathOperator{\SPT}{SPT}

\newcommand{\Th}{\ensuremath{\text{{\it Th}}}}

\newcommand{\FA}{\ensuremath{\text{{\sf FA}}}}

\newcommand{\PFA}{\ensuremath{\text{{\sf PFA}}}}
\newcommand{\SP}{\ensuremath{\text{{\sf SP}}}}

\newcommand{\MM}{\ensuremath{\text{{\sf MM}}}}

\newcommand{\SSP}{\ensuremath{\text{{\sf SSP}}}}

\newcommand{\FF}{\ensuremath{\text{{\rm F}}}}
\newcommand{\NS}{\ensuremath{\text{{\sf NS}}}}
\newcommand{\UU}{\mathsf{U}}
\newcommand{\RR}{\mathsf{R}}
\newcommand{\QQ}{\mathsf{Q}}
\newcommand{\BB}{\mathsf{B}}
\newcommand{\CC}{\mathsf{C}}
\newcommand{\DD}{\mathsf{D}}

\newcommand{\TT}{\mathbb{T}}
\newcommand{\BBB}{\mathsf{RO}}
\newcommand{\RCS}{\mathcal{RCS}}

\newcommand{\II}{\mathcal{I}}
\newcommand{\FFF}{\mathcal{F}}

\let\mathbb=\varmathbb

\begin{document}

\title[Category forcings and generic absoluteness]
{Category forcings, $\MM^{+++}$, and generic absoluteness for the theory of 
strong forcing axioms}


%    author one information
\author{Matteo Viale}
\address{Department of Mathematics ``Giuseppe Peano''\\
University of Torino\\
via Carlo Alberto 10\\
10125\\
Torino\\
ITALY}
%\curraddr{}
\email{matteo.viale@unito.it}
\thanks{
The author acknowledges support from: 
PRIN grant 2009: Modelli e insiemi, 
PRIN grant 2012: Logica, modelli e insiemi, 
Kurt G\"odel Research Prize Fellowship 2010,
the Fields Institute in Mathematical Sciences,
San Paolo Junior PI grant 2012.
Thanks to David Asper\'o and Menachem Magidor for the careful 
examination of the results of this paper. Moreover they both contributed to 
outline what limitations a generic absoluteness result for 
Martin's maximum should face. 
Asper\'o also gave substantial inputs for the proof
of the freezeability property for
$\SSP$-forcings.
Paul Larson 
found several other weak spots in previous drafts of this manuscript.
I should also thank among others:
Daisuke Ikegami,
Sean Cox, Boban Veli\v{c}kovi\'c,
Ralph Schindler, James Cummings, Ilijas Farah, 
Giorgio Audrito, Stevo Todor\v{c}evi\'c, Joan Bagaria, Sy David Friedman.
Many thanks to the referee for the patient and careful examination of this article and for the
useful suggestions.
}

%    \subjclass is required.

\date{}

%\dedicatory{To Chiara and to our kids, Pietro and Adele.}

%    Abstract is required.
\begin{abstract}
We analyze certain subfamilies of the category of
complete boolean algebras
with complete homomorphisms, families which are of particular interest in set theory. 
In particular we study the category
whose objects are stationary set preserving, atomless complete boolean algebras
and whose arrows are complete homomorphisms with a stationary
set preserving quotient. 
We introduce a maximal forcing axiom $\MM^{+++}$ as a combinatorial property of this category.
This forcing axiom strengthens
Martin's maximum and can be seen at the same time 
as a strenghtening of Baire's category theorem and of the axiom of choice. 
Our main results show that
$\MM^{+++}$ is consistent relative to large cardinal axioms and that $\MM^{+++}$
makes 
the theory of the Chang model $L([\Ord]^{\leq\aleph_1})$
with parameters in $P(\omega_1)$ generically invariant for stationary set preserving forcings that
preserve this axiom. 
We also show that our results give 
a close to optimal extension to the Chang model $L([\Ord]^{\leq\aleph_1})$
of Woodin's generic absoluteness results for the Chang model $L([\Ord]^{\aleph_0})$
and give an a posteriori explanation of the success 
forcing axioms have met in set theory.
\end{abstract}

\maketitle

The main objective of this paper is to show that there is
a natural recursive extension $T$ of
$\ZFC+$\emph{ large cardinals}
which gives a complete theory of the Chang model
$L([\Ord]^{\leq\aleph_1})$ with respect to the unique efficient method to produce consistency results 
for this structure, i.e stationary set preserving forcings. In particular we will show that a closed
formula
$\phi$ relativized to this Chang model is first order derivable in $T$ if and only if it 
is provable in $T$
that $T+\phi$ is forceable by a stationary set preserving forcing.
In our eyes the results of this paper 
give a solid a posteriori explanation of the success forcing axioms have met in
providing at least one consistent solution to many $\ZFC$-provably undecidable problems.
The paper can be divided in six sections: 

\begin{itemize}
\item
An introduction (section~\ref{sec:intro}) where it is 
shown how the above results stem out of
Woodin's work on $\Omega$-logic and of Woodin's absoluteness results for
$L(\mathbb{R})$ and for the Chang model $L([\Ord]^{\aleph_0})$. We also try to explain 
in what terms we can assert that the main result of this paper is an optimal extension
to the Chang model $L([\Ord]^{\leq\aleph_1})$ of Woodin's
absoluteness results for the Chang model $L([\Ord]^{\aleph_0})$.
We shall also try to argue that the results of this paper 
give an a posteriori explanation of the success
that forcing axioms have met 
in solving a variety of problems showing up in set theory 
as well as in many other fields of pure mathematics.
\item
Section~\ref{sec:background} presents some background material 
on stationary sets (subsection~\ref{subsec:statsets}), large cardinals (subsection~\ref{subsec:lc}),
posets and boolean completions (subsection~\ref{subsec:boolalgforc}),
stationary set preserving forcings (subsection~\ref{subsec:SSPforc}), 
forcing axioms (subsection~\ref{subsec:forcax}), iterated forcing (subsection~\ref{subsec:itfor})
which will be needed in the remainder of the paper. 
\item
Section~\ref{sec:univpartord}
introduces the notion of category forcings. 
We shall look at subcategories of the category of complete boolean algebra with
complete homomorphisms.
Given a category $(\Gamma,\Theta)$ (where $\Gamma$ is the class of objects and 
$\Theta$ the class of arrows) we associate to it the partial order 
$(\UU^{\Gamma,\Theta},\leq_\Theta)$ whose elements are the objects in $\Gamma$ ordered by
$\BB\geq_\Theta\CC$ iff there is an $i:\BB\to\CC$ in $\Theta$.
We shall also feel free to confuse a set sized partial order with its uniquely defined boolean completion.
In this paper we shall focus on the analysis of the category
$(\SSP,\SSP)$  whose objects are
the stationary set preserving ($\SSP$) complete boolean algebras 
and whose arrows (still denoted by $\SSP$)
are the complete homomorphisms with a stationary set preserving quotient. The reasons are twofolds:
\begin{itemize}
\item
We aim at a generic absoluteness result for a strengthening of Martin's maximum. This 
naturally leads to an analysis of the category of forcings which are relevant for this forcing axiom, i.e.
the $\SSP$-forcings. 
\item
We are able to predicate all the  nice features we shall isolate for a category forcing just for the 
forcing $(\UU^{\SSP,\SSP},\leq_{\SSP})$ (which we shall denote from now on just as $\UU^{\SSP}$).
\end{itemize}
The following list sums up the main concepts and results we isolate 
on the combinatorial properties of these category forcings:
\begin{enumerate}
\item
We introduce the key concept (at least for our aims) of totally rigid element of a category 
$(\Gamma,\Theta)$.
 
$\BB\in\Gamma$ is $\Theta$-totally rigid if it is fixed by any automorphism 
of some complete boolean algebra in $\Gamma$ which absorbs $\BB$ 
using an arrow in $\Theta$. We can formulate this property in purely categorical terms as follows:
\begin{quote}
$\BB$ object of $\Gamma$ is $\Theta$-totally rigid
if for all $\QQ\in\Gamma$ there is at most one arrow $i:\BB\to\QQ$ in $\Theta$.
\end{quote}
We show that in the presence of class many supercompact cardinals, 
the class of $\SSP$-totally rigid partial orders is dense in $\UU^{\SSP}$ (Theorem~\ref{thm:distrSSP*}).

\item
We show that the
cut off $\UU^{\SSP}_\delta=\UU^{\SSP}\cap V_\delta$
of this category forcing at a rank initial segment $\delta$ which is 
an inaccessible limit of $<\delta$-supercompact cardinals is an $\SSP$-totally rigid partial order
which belongs to $\SSP$ and
which absorbs all forcings in $\SSP\cap V_\delta$ (Theorem~\ref{thm:univ1}.\ref{thm:univ1-1}). 
\item
We also show that
$\UU^{\SSP}_\delta$ forces $\MM^{++}$ in case $\delta$ is a 
supercompact limit of $<\delta$-supercompact cardinals (Theorem~\ref{thm:univ1}.\ref{thm:univ1-2}).
\item
We show that the quotient of the category forcing $(\UU^{\SSP})^V$
with respect to a generic filter $G$ for any of its elements $\BB\in\SSP^V$ is 
the category forcing $(\UU^{\SSP})^{V[G]}$ as computed in the generic extension $V[G]$
(see Theorem~\ref{lem:univ4}, where it is also given a precise definition of this statement).
\end{enumerate}

\item
In section~\ref{subsec:nortow} we sum up the relevant facts about
towers of normal ideals we shall need in order to formulate the results of section~\ref{sec:MM+++}.

\item
Section~\ref{sec:MM+++} introduces and analyzes the forcing axiom $\MM^{+++}$.
First we observe that in the presence of class many Woodin cardinals
the forcing axiom
$\MM^{++}$ can be formulated as the assertion that the class
of \emph{presaturated towers} is dense in the category forcing $\UU^{\SSP}$.
What can be said about the intersection of the class of totally rigid posets and the class of 
presaturated towers? Can this intersection be still a dense class in $\UU^{\SSP}$? 
Can $\UU^{\SSP}_\delta$ belong to this
intersection for some $\delta$?
The forcing axiom $\MM^{+++}$ arises as a positive answer to these questions and
is a slight strengthening of the assertion that the class of 
presaturated tower forcings which are also
totally rigid is dense in the category forcing $\UU^{\SSP}$.

We first prove our main generic 
absoluteness result, i.e. that over any model of $\MM^{+++}+$\emph{large cardinals}
any stationary set preserving forcing which preserves $\MM^{+++}$ does not change the theory of
$L([\Ord]^{\leq\aleph_1})$ with parameters in $P(\omega_1)$ (Theorem~\ref{cor:absres}). 
Then we turn to the proof of the consistency of $\MM^{+++}$ 
showing that any of the standard forcing 
methods to produce a model of 
$\MM^{++}$ collapsing an inaccessible $\delta$ to become $\omega_2$ actually produces
a model of $\MM^{+++}$ provided that $\delta$ is a super huge cardinal (Corollary~\ref{cor:MMM-easy},
and Theorem~\ref{thm:CONMM+++}).
In particular Theorem~\ref{thm:CONMM+++} 
provides a rather extensive sample of notions which can be used to apply
the generic absoluteness result.

\item
We end our paper in section~\ref{sec:com} 
with some comments regarding our results.
In particular: 
\begin{itemize}
\item
We outline that the usual forcing axioms can be seen as topological formulations
of strenghtenings of the axiom of choice and of Baire's category theorem.
On the other hand the category theoretic framework allows to present
$\MM^{+++}$ (and other types of forcing axioms) as a 
formulation in the language of categories of suitable strenghtenings of
many of the usual forcing axioms.
\item
We outline the modularity of our results conjecturing that they 
ultimately can be predicated for many category forcings 
$(\UU^{\Gamma},\leq_\Gamma)$ given by 
classes of forcings $\Gamma$ 
satisfying certain natural requirements. 
\item
We give some heuristic argument suggesting that 
$\MM^{++}$ is an axiom really weaker than $\MM^{+++}$.
\item
We show that our results are optimal outlining
that no axiom strictly weaker than $\MM^{++}$ can produce a generic absoluteness result
for the theory of the Chang model $L([\Ord]^{\leq\aleph_1})$ with respect to $\SSP$-forcings
which preserve this axiom.
\item
We also give a list of some possible lines of further investigations and link our results to other recent
researches in this area by Hamkins-Johnstone~\cite{HAMJOH13}, Tsaprounis~\cite{TSA13}, and the
author \cite{VIAMMREV}, \cite{VIAAUD14}(the latter with Audrito).
\end{itemize}
\end{itemize}
Section~\ref{sec:MM+++} depends on the results of section~\ref{sec:univpartord}.
Sections~\ref{sec:background}
and~\ref{subsec:nortow}
present background material and the reader acquainted with it can just skim through them.
We tried to maintain this paper as much self contained as possible, nonetheless
with the exception of the introduction,
the reader is expected to have a strong 
background in set theory and familiarity with forcing axioms,
tower forcings, large cardinals. 
Detailed references for the material presented in 
this paper are mentioned throughout the text, basic sources are Jech's 
set theory handbook~\cite{JEC03}, Larson's book on
the stationary tower forcing~\cite{LAR04}, Foreman's handbook chapter on generic elementary 
embeddings~\cite{FORHST}, our notes on semiproper iterations~\cite{VIAAUDSTE13}.

\tableofcontents

\section{Introduction}\label{sec:intro}

Since
its introduction by Cohen in 1963 forcing has been the key and most effective 
tool to obtain independence results in set theory.
This method has found applications in set theory and
in virtually all fields of pure mathematics: 
in the last forty years natural problems of group theory, functional analysis, 
operator algebras, general topology, and many other subjects were shown to 
be undecidable by means of forcing.
However already in the early seventies and with more evidence since the eighties
it became apparent that many consistency results could all be derived by a short list of 
set theoretic principles which are known in the literature as forcing axioms.
These axioms gave set theorists and mathematicians a very powerful tool to 
obtain independence results:
for any given mathematical problem we are most likely able to compute its (possibly different)
solutions 
in the constructible universe $L$ and in models of strong forcing axioms.
These axioms settle basic problems in 
cardinal arithmetic like the size of the continuum and the singular 
cardinal problem (see among others the works of Foreman, 
Magidor, Shelah~\cite{foreman_magidor_shelah}, 
Veli\v{c}kovi\'c~\cite{VEL92}, Todor\v{c}evi\'c~\cite{tod02}, 
Moore~\cite{moore.MRP}, Caicedo and Veli\v{c}kovi\'c~\cite{caivel06}, 
and the author~\cite{VIA08}), as well as combinatorially 
complicated ones like the basis problem for uncountable linear 
orders (see Moore's result~\cite{moore.basis} which extends 
previous work of Baumgartner~\cite{bau73}, Shelah~\cite{she76}, 
Todor\v{c}evi\'c~\cite{tod98}, and others).
Interesting problems originating from other fields of mathematics and apparently 
unrelated to set theory have also been settled appealing to forcing axioms, as it is the case 
(to cite two of the most prominent examples)
for Shelah's
results~\cite{SHE74} on Whitehead's problem in group theory and Farah's result~\cite{FAR11}
on the non-existence of outer automorphisms of the Calkin algebra in operator algebra.
Forcing axioms assert that for a large class of compact topological spaces $X$
Baire's category theorem can be strengthened to the statement that any family of
$\aleph_1$-many dense open subsets of $X$ has non empty intersection.
In light of the success these axioms have met in solving problems
 a convinced platonist may start to argue that these principles
may actually give a ``complete'' theory of a suitable fragment of the universe. 
However it is not clear how one could formulate such a completeness result.
The aim of this introduction is to explain in which sense we can show that
these strong forcing axioms can give such a 
``complete'' theory.
Our argument will find its roots in the work
of Woodin in $\Omega$-logic. 
The basic observation is that the 
working tools of a set theorist are either 
first order calculus, by which he/she can justify his/her proofs over $\ZFC$, or forcing,
by which he/she can obtain his/her independence results over $\ZFC$.
However it appears that there is still a gap between what we can achieve by ordinary proofs in 
some axiom system which extends $\ZFC$ and the independence 
results that we can obtain over this theory by means of forcing. 
More specifically to close the gap it appears that
we are lacking two desirable feature we would like to have for a ``complete'' first order theory $T$
that axiomatizes set theory with respect to its semantics given by the class of boolean 
valued models of $T$:
\begin{itemize}
\item
$T$ is complete with respect to its intended semantics, i.e for all statements $\phi$
only one among $T+\phi$ and $T+\neg\phi$ is forceable.
\item
Forceability over $T$ should correspond to a  notion of derivability with respect to some proof system, 
eventually derivability with respect to a standard first order calculus for $T$.
\end{itemize}

Both statements appear to be rather bold and have to be handled with care:
Consider for example the statement $\omega=\omega_1$ in a theory
$T$ extending $\ZFC$ with the statements \emph{$\omega$ is the first infinite cardinal} and 
\emph{$\omega_1$ is the first uncountable cardinal}.
Then clearly
$T$ proves $|\omega|\neq|\omega_1|$, while if one forces with $Coll(\omega,\omega_1)$
one produce a model of set theory where this equality holds (however the formula
\emph{$\omega_1$ is the first uncountable cardinal} is now false in this model).
On a first glance this suggests that as we expand the language for $T$, forcing starts to act randomly
on the formulae of $T$ switching the truth value of its formulae with parameters in ways which
it does not seem simple to describe\footnote{
This is not anymore the case for the closed formulae of $T$.
Hamkins and L\"owe 
represent the forceability of a closed formula $\phi$ as an interpretation of the 
modal statement $\diamond\phi$ and
have shown that the
closed sentences of $\ZFC$ and the family of all generic multiverses (i.e the generic extensions
of a given model $V$ of $\ZFC$) can be used to define
a family of correct and complete frames for the propositional modal logic
$S4.2$~\cite{HAMLOE08}.}.
 However the above difficulties are raised essentially by our lack of attention to define 
the type of formulae for which we aim to have the completeness of $T$ with respect to 
forceability. 
We shall show that when the formulae are prescribed to talk only about a suitable initial segment
of the set theoretic universe and we consider only forcings that preserves the intended
meaning of the parameters by which we enriched the language of $T$, this random
behaviour of forcing does not show up anymore.

We shall assume a platonistic stance towards set theory, 
thus we have one canonical model $V$ of $\ZFC$ of which we try to
uncover the truths.
To do this we may allow ourselves to use all model theoretic techniques that produce new models of
the truths of $Th(V)$ on which we are confident, which (if we are platonists)
certainly include $\ZFC$ and all the axioms of large cardinals.
We may start our quest for uncovering the truth in $V$ by first settling the theory of
$H_{\omega_1}^V$ (the hereditarily countable sets), then the theory of $H_{\omega_2}^V$
(the sets of hereditarily cardinality $\aleph_1$) and so on so forth covering step by step
all infinite cardinals.
To proceed we need some definitions:

 \begin{Definition}
Given a theory $T\supseteq \ZFC$ and a family 
$\Gamma$ of partial orders definable in $T$, we say that
$\phi$ is $\Gamma$-consistent for $T$ if 
$T$ proves that there exists a partial order in $\Gamma$ which forces $\phi$.

Given a model $V$ of $\ZFC$ we say that $V$ models that $\phi$ is
$\Gamma$-consistent if $\phi$ is $\Gamma$-consistent for $\Th(V)$.

\end{Definition}

\begin{Definition}
Given a partial order $P$ and a cardinal $\lambda$,
$\FA_\lambda(P)$ holds if for all $\{D_\alpha:\alpha<\lambda\}$ family of dense subsets of
$P$, there is $G\subset P$ filter which has non empty intersection with all
$D_\alpha$.
\end{Definition}

\begin{Definition}
Let 
\[
T\supseteq\ZFC+\{\lambda \text{ is an infinite cardinal}\}
\]
$\Omega_\lambda$ is the definable (in $T$) class of partial orders $P$ which satisfy
$\FA_\lambda(P)$.
\end{Definition}
In particular Baire's category theorem amounts to say that
$\Omega_{\aleph_0}$ is the class of all partial orders (denoted by Woodin
as the class $\Omega$).
The following is a basic observation whose proof 
can be found in~\cite[Lemma 1.2]{VIAMMREV}
\begin{Lemma}[Cohen's absoluteness Lemma]
Assume $T\supseteq\ZFC+\{p\subseteq\omega\}$ and
$\phi(x,p)$ is a $\Sigma_0$-formula.
Then the following are equivalent:
\begin{itemize}
\item
$T\vdash \exists x\phi(x,p)$,
\item
$T\vdash \exists x\phi(x,p)$ is $\Omega$-consistent.
\end{itemize}
\end{Lemma}
This shows that 
for $\Sigma_1$-formulas with real parameters
the desired overlap between the ordinary 
notion of provability and the semantic notion of forceability is a provable fact in $\ZFC$.
Now it is natural to ask if we can expand the above in at least two directions:
\begin{enumerate}
\item
Increase the complexity of the formula,
\item
Increase the language allowing parameters also for other infinite cardinals.
\end{enumerate}
The second direction requires almost no effort once one notices that in order to prove that
$\exists x\phi(x,p)$ is provable if it is $\Omega$-consistent, we used the fact that for all forcing $P$
$\FA_{\aleph_0}(P)$ is  provable in $\ZFC$, and that 
the latter is just the most general formulation of Baire's 
category theorem.
We can thus reformulate the above equivalence in a modular form as follows 
(see for a proof~\cite[Lemma 1.3]{VIAMMREV}):
\begin{Lemma}[Generalized Cohen's absoluteness Lemma]
Assume 
\[
T\supseteq\ZFC+\{p\subset\lambda\}+\{\lambda\text{ is an infinite cardinal}\}
\] 
and
$\phi(x,p)$ is a $\Sigma_0$-formula.
Then the following are equivalent:
\begin{itemize}
\item
$T\vdash \exists x\phi(x,p)$,
\item
$T\vdash \exists x\phi(x,p)$ is $\Omega_{\lambda}$-consistent.
\end{itemize}
\end{Lemma}
The extent by which we can increase the complexity of the formula requires once again
some attention to the semantical interpretation of its parameters and its quantifiers.
We have already observed that the formula $\omega=\omega_1$
is inconsistent but $\Omega$-consistent
in a language with parameters for $\omega$ and $\omega_1$.
One of 
Woodin's main achievements\footnote{We follow Larson's presentation as in~\cite{LAR04}.} 
in $\Omega$-logic show that if we restrict the semantic interpretation of
$\phi$ to range over the structure $L([\Ord]^{\aleph_0})$ and 
we assume large cardinal axioms, we can get a full correctness and completeness 
result~\cite[Corollary 3.1.7]{LAR04}:

\begin{Theorem*}[Woodin]
Assume 
\[
T\supseteq\ZFC+\{p\subset\omega\}+\text{there are class many Woodin cardinals which
are limits of Woodin cardinals}
\] 
and $\phi(x)$ is any formula in one free variable.
Then the following are equivalent:
\begin{itemize}
\item
$T\vdash [L([\Ord]^{\aleph_0})\models\phi(p)]$,
\item
$T\vdash [L([\Ord]^{\aleph_0})\models\phi(p)]$ is $\Omega$-consistent.
\end{itemize}
\end{Theorem*}
The natural question to address now is whether we can step up this result also for uncountable 
$\lambda$.
If so in which form?
Woodin~\cite[Theorem 3.2.1]{LARHST} has proved another remarkable absoluteness result for $\CH$ i.e.:
\begin{Theorem*}[Woodin]
Let $T$ extend $\ZFC+$\emph{ there are class many measurable Woodin cardinals}.

A $\Sigma^2_1$-statement $\phi(p)$ with real parameter $p$
is $\Omega$-consistent for $T$ if and only if 
\[
T+\CH\vdash\phi(p).
\]
\end{Theorem*}

However there are two distinct results that show that we cannot hope to obtain a complete
(and unique) theory with respect to forceability which extends $\ZFC+\CH$:
\begin{itemize}
\item
Asper\'o, Larson and Moore~\cite{ASPLARMOO12} have shown that there are 
two distinct $\Pi_2$-statements $\psi_0,\psi_1$
over the theory of $H_{\aleph_2}$ such that $\psi_0+\psi_1$ denies $\CH$ and
$\psi_i+\CH$ is forceable by means of a proper forcing for $i=0,1$ over a model of 
$\ZFC+$\emph{ large cardinal axioms}. This shows that
any completion of $\ZFC+\CH+$\emph{ large cardinal axioms} 
cannot simultaneously realize all $\Pi_2$-statements over the theory of
$H_{\aleph_2}$ each of which is known to be consistent with $\CH$ 
(even consistent by means of a \emph{proper} forcing). 
\item
Woodin and Koellner~\cite{WOOKOE09} have shown that if there is 
an $\Omega$-complete theory $T$ for $\Sigma^2_3$-statements with real parameters which implies 
$\CH$, then there is another $\Omega$-complete theory 
$T'$ for $\Sigma^2_3$-statements with real parameters which denies $\CH$.
\end{itemize}
In particular the first result shows that we cannot hope to extend $\ZFC+\CH$ to a natural maximal
completion which settle the $\Pi_2$-theory of the structure $H_{\aleph_2}$ at least with respect to
the semantics given by forcing.
Finally Woodin has proved a remarkable absoluteness result for a close relative of forcing axioms,
Woodin's axiom\footnote{See~\cite{LARHST} or~\cite{woodin} for a detailed 
presentation of the models of this axiom.} $(*)$. For this axiom Woodin can prove the consistency of
a completeness and correctness result for $\ZFC+(*)$
with respect to a natural but non constructive proof system and to $\Omega$-consistency.
This completeness result is very powerful for it applies to the largest possible 
class of models produced by forcing, but it has two features which need to be clarified:
\begin{itemize}
\item
It is not known if $(*)$ is $\Omega$-consistent i.e. if its consistency can be proved by forcing over an 
ordinary model of $\ZFC$.
\item
The correctness and completeness result for $(*)$ are with respect to a natural but non-constructive
proof system and moreover the completeness theorem is known to hold only 
under certain assumptions on the set theoretic properties of $V$.
\end{itemize}

Let us now focus on the first order theory with parameters in $P(\omega_1)$
of the structure $H_{\omega_2}$ or more generally
of the Chang model $L([\Ord]^{\leq\aleph_1})$.
A natural approach to the study of this Chang model is to expand the language 
of $\ZFC$ to enclude constants for all elements of $H_{\omega_2}$
and the basic relations between these elements:
\begin{Definition}
Let $V$ be a model of $\ZFC$ and $\lambda\in V$ be a
cardinal.
The $\Sigma_0$-diagram of $H_\lambda^V$ is given by
the theory
\[
\{\phi(p):p\in H_\lambda^V,\,\phi(p)\text{ a $\Sigma_0$-formula true in $V$}\}.
\]
\end{Definition}
Following our approach, the natural theory of $V$ we should look at  is the theory:
\[
T=\ZFC+\text{ large cardinal axioms }+\text{$\Sigma_0$-diagram of $H_{\omega_2}$},
\]
since we already know that $\ZFC+$\emph{large cardinal axioms} settles the theory of
$L([\Ord]^{\aleph_0})$ with parameters in
$H_{\omega_1}$.
Now consider any model $M$ of $T$ we may obtain using model-theoretic techniques.
In particular we may assume that $M$ is  a ``monster model" which contains $V$
such that for some completion $\bar{T}$ of $T$,
$M$ is a model of $\bar{T}$ which amalgamates ``all'' possible models of $\bar{T}$ and 
realizes all consistent types of $\bar{T}$ with parameters in $H_{\omega_2}^V$.
If we have any hope that $\bar{T}$ is really the theory of $V$ we are aiming for, 
we should at least require that $V\prec_{\Sigma_1} M$.
Once we make this requirement we notice the following:
\[
V\cap\NS_{\omega_1}^M=\NS_{\omega_1}^V.
\]
If this were not the case, then for some $S$ stationary and costationary
in $V$, $M$ models that $S$ is not stationary, i.e. that there is a club of 
$\omega_1$ disjoint from $S$. Since $V\prec_{\Sigma_1} M$ such a club can be found
in $V$. This means that $V$ already models that $S$ is non stationary.
Now the formula $S\cap C=\emptyset$\emph{ and $C$ is a club} is $\Sigma_0$ and thus it is part of
the $\Sigma_0$-diagram of $H_{\omega_2}^V$.
However this contradicts the assumption that $S$ is stationary and costationary in $V$ 
which is expressed by the fact that the above formula is not part of the 
$\Sigma_0$-elementary diagram of $V$.
This shows that $V\not\prec_{\Sigma_1} M$.
Thus any ``monster'' model $M$ as above should be correct about the non-stationary ideal,
so we better add this ideal as a predicate to the $\Sigma_0$-diagram of
$H_{\omega_2}^V$, to rule out models of the completions of $T$ which cannot even be
$\Sigma_1$-superstructures of $V$.
Remark that on the forcing side, this is immediately leading to the 
notion of stationary set preserving forcing: if we want to use forcing to 
produce such approximations of ``monster'' models while preserving 
the fact of being a $\Sigma_1$-elementary superstructure of $V$ with respect to $T$,
we are forced to restrict our attention to stationary set preserving forcings. 
Let us denote by $\SSP$ the class of stationary set preserving posets and recall that
Martin's maximum asserts that $\FA_{\aleph_1}(P)$ holds for all $\SSP$-partial orders $P$.
Subject to the limitations we have outlined the best possible result we can hope for
is to find a theory 
\[
T_1\supset T
\]
such that:
\begin{enumerate}
\item \label{req1}
$T_1$ proves the strongest possible forcing axiom i.e. some natural strengthening of 
Martin's maximum,
\item \label{req2}
For any $T_2\supseteq T_1$ and 
any formula $\phi(S)$ 
relativized to the Chang model $L([\Ord]^{\leq\aleph_1})$ 
with parameter $S\subset\omega_1$ and a predicate for the non-stationary ideal $\NS_{\omega_1}$,
$\phi(S)$ is provable in
$T_2$ if and only if the theory $T_1+\phi(S)$ is $\Omega_{\aleph_1}$-consistent for $T_2$.
\end{enumerate}
We shall separately give arguments to justify these two requirements.
 
The first requirement above (\ref{req1}) is a natural maximality principle, since
it can be argued that Martin's maximum $\MM$ is a natural strengthening of the axiom of 
choice and of Baire category theorem:
\begin{itemize}
\item
On the one hand Todor\v{c}evi\'c~\cite{TODACFA14} has noticed that the axiom of choice
 is equivalent (over $\mathsf{ZF}$) to the
global forcing axiom asserting that $\FA_\lambda(P)$ holds for each regular cardinal 
$\lambda$ and for any 
$<\lambda$-directed closed poset\footnote{
Roughly the argument goes as follows: the axiom of choice
is equivalent to the assertion that the the axiom of dependent choice $\mathsf{DC}_\lambda$
holds for all infinite cardinals $\lambda$. It is almost immediate to check that 
$\mathsf{DC}_\lambda$ is equivalent to the assertion that $\FA_\lambda(P)$
holds for any 
$<\lambda$-directed closed poset $P$ and that ---assuming the amount of axiom of choice needed to
implement Stone duality---
$\mathsf{DC}_\omega$ is an equivalent formulation
of Baire's category theorem (see for more 
details~\href{http://www.personalweb.unito.it/matteo.viale/LUMINY2014viale.pdf}
{http://www.personalweb.unito.it/matteo.viale/LUMINY2014viale.pdf}).} 
$P$.
\item
On the other hand Shelah has shown that $\FA_{\aleph_1}(P)$ provably 
fails if $P$ is not stationary set 
preserving below some of its conditions. 
\item
It is also well known by means of Stone duality that for any compact Hausdorff 
topological space
$(X,\tau)$, the intersection of 
a family of $\lambda$-many dense open sets is non-empty if and only if the partial order
$P=(\tau\setminus\{\emptyset\},\supseteq)$ is such that $\FA_\lambda(P)$ holds.
\end{itemize}
In particular: the first and the third item
show that Martin's maximum can be seen as a natural topological formulation of a
strengthening of the axiom of choice, since countably closed forcings are stationary set preserving.
The second and the third items show
that Martin's maximum is a maximal topological strenghtening of Baire's category theorem.
It has been shown by the work of Foreman, Magidor and Shelah~\cite{foreman_magidor_shelah} 
that Martin's maximum  is 
$\SSP$-consistent with respect to $T=\ZFC+$\emph{large cardinals}
and thus it is consistent. Denying it is not required by the known constraints
we have to impose on $T$ in order to get a complete extension of $T$.

The second requirement (\ref{req2}) above is the best possible form of completeness theorem
we can currently formulate: there may be interesting model theoretic tools to produce
models of $T$ which are not encompassed by forcing,
however we haven't as yet developed powerful techniques to exploit them in the study of models
of $\ZFC$. Moreover the second requirement (\ref{req2})
shows that forcing
becomes a powerful proof tool in the presence of strong forcing axioms, since it transforms a 
validity problem in a consistency problem.

$\MM^{++}$ is a well known ``natural'' strengthening of Martin's maximum (i.e. of the equality 
$\SSP=\Omega_{\aleph_1}$) shown to be consistent relative to the existence of a 
supercompact cardinal by the work of~\cite{foreman_magidor_shelah}.
In the present paper we show that a further ``natural'' strengthening of $\MM^{++}$, which we call
$\MM^{+++}$ (see Def.~\ref{def:MM+++})
 enriched by suitable large cardinals axioms (see Def.~\ref{def:sahlc} for the relevant notions)
is already such a theory $T_1$ as we can prove the following:
\begin{Theorem}\label{thm:mainthm}
Let $\ZFC^*$ stands for\footnote{$\Sigma_2$-reflecting cardinals are defined in 
Def.~\ref{def:reflcar}.} 
\[
\ZFC+\text{ there are class many 
$\Sigma_2$-reflecting cardinals }
\]
and $T^*$ be any theory extending 
\[
\ZFC^*+\MM^{+++}+\omega_1\text{ is the first uncountable
cardinal }+S\subset\omega_1.
\]
Then for any formula\footnote{If we allow formulae of arbitrary
complexity, we do not need to enrich the language with 
a predicate for the non-stationary ideal, since this ideal is a definable predicate over
$H_{\omega_2}$ (though defined by a $\Sigma_1$-property) and thus can be incorporated 
as  a part of the formula.} 
$\phi(S)$ the following are equivalent:
\begin{enumerate}
\item
$T^*\vdash [L([\Ord]^{\leq\aleph_1})\models\phi(S)]$,
\item
$T^*\vdash \MM^{+++}$ and $[L([\Ord]^{\leq\aleph_1})\models\phi(S)]$ are
jointly $\Omega_{\aleph_1}$-consistent.
\end{enumerate}
\end{Theorem}
We shall also see that the result is sharp in the sense that
the work of Asper\'o~\cite{ASPPFA++} and Larson~\cite{larson-dwo} 
shows that we cannot obtain the above completeness and 
correctness result relative to forcing axioms which are just slightly weaker than $\MM^{++}$.
It remains open whether our axiom $\MM^{+++}$ is really stronger than $\MM^{++}$ in the presence
of large cardinals.

Finally I think that the present results show that we have all reasons to expect that
$\MM^{+++}$ (and most likely already $\MM^{++}$)
can decide Woodin's axiom $(*)$: any proof of the consistency of $\MM^{+++}$ with $(*)$
or with its negation obtained by an $\SSP$-forcing would convert by the results of this paper into
a proof of the provability of $(*)$ (or of its negation) from $\MM^{+++}$.

\section{Background material on large cardinals, 
generalized stationarity, forcing axioms}\label{sec:background}

\subsection{Stationary sets and normal ideals}\label{subsec:statsets}

We follow standard set theoretic notation as in~\cite{JEC03}.
In particular for an arbitrary set or class $X$ and a cardinal $\lambda$
\[
[X]^{\lambda}=\{Y\in P(X): |Y|=\lambda\}.
\]
$[X]^{\leq\lambda}$ and $[X]^{<\lambda}$ are defined accordingly.
For $X\supseteq\lambda$
\[
P_\lambda(X)=\{Y\in P(X): |Y|<\lambda\text{ and }Y\cap\lambda\in\lambda\}.
\]
We let
\[
P_\lambda=\{Y\in V: |Y|<\lambda\text{ and }Y\cap\lambda\in\lambda\}.
\]
Thus $P_\lambda(X)=P(X)\cap P_\lambda$.

For any $f:P_\omega(X)\to X$ we let 
$C_f\subset P(X)$ be the set of its closure points (i.e. the set of $y\subset X$ such that
$f[P_{\omega}(y)]\subset y$).
\begin{definition}
$S$ is stationary in $X$ 
if $S\cap C_f$ is non-empty for all $f:P_\omega(X)\to X$.
$S$ is stationary if it is stationary in $\cup S$.
\end{definition}

\begin{definition}
$I\subset P(P(X))$ is an ideal on $X$ if it is closed under subsets and 
finite unions. The dual filter of an ideal
$I$ is denoted by $\breve{I}$.
$I^+=P(P(X))\setminus I$.

$I$ is normal if
for all $S\in I^+$ and regressive
$f:S\to X$ there is $T\in I^+$ on which $f$ is constant.

$I$ is $<\kappa$-complete if for all $J\subset I$ of size less than $\kappa$,
$\cup J\in I$.

The completeness of $I$ is the largest $\kappa$ such that $I$ is $<\kappa$-complete.

If $S\in I^+$, $I\restriction S$ is the ideal generated by $I\cup\{P(X)\setminus S\}$. 
\end{definition}

\begin{definition}
The non stationary ideal $\NS_X$ on $X$ is the ideal generated
by the complements of sets of the form 
\[
C_f=\{Z\in P(X): f[Z^{<\omega}]\subset Z\}
\]
for some $f:X^{<\omega}\to X$.
Its dual filter is the club filter on $X$.
\end{definition}
An ideal $I$ on $X$ concentrates on $S\subset P(X)$, if $P(X)\setminus S\in I$.

\begin{remark}
$P_{\aleph_0}(X)=[X]^{<\aleph_0}$ for all $X\supseteq\omega$ and 
$P_{\aleph_1}(X)$ is a club subset of $[X]^{<\aleph_1}$ for all $X\supseteq\omega_1$.
For other cardinals $\lambda>\aleph_1$, 
$[X]^{<\lambda}\setminus P_\lambda(X)$ can be stationary.
\end{remark}

\begin{lemma}[Pressing down Lemma]
Assume $S$ is stationary and $f:S\to \cup S$ is such that 
$f(X)\in X$ for all $X\in S$. Then there is $T\subset S$ on which 
$f$ is constant. In particular $\NS_X$ is a normal ideal for all $X$.
\end{lemma}

We call functions $f$ as in the Lemma regressive.

For a stationary set $S$ and a set $X$, if $\cup S\subseteq X$
we let $S^{X}=\{M\in P(X):M\cap\cup S\in S\}$,
if $\cup S\supseteq X$
we let $S\restriction X=\{M\cap X:M\in S\}$.
We define
an order on stationary sets given by
$T\leq S$ 
if letting $X=(\cup T)\cup (\cup S)$, there is $f :P_\omega(X)\to X$ 
such that $T^X\cap C_f\subseteq S^X$.
We let $S\equiv T$ if $S\leq T$ and 
$T\leq S$.

In general if $\{S_i:i<\xi\}$ is a family of stationary sets we let
$\eta$ be the least such that $\{S_i:i<\xi\}\in V_\eta$ and
\[
\bigwedge\{S_i:i<\xi\}=\{M\prec V_\eta: 
\{S_i:i<\xi\}\in M\text{ and }\forall S_i\in M: M\cap \cup S_i\in S_i\}
\]
\[
\bigvee\{S_i:i<\xi\}=\{M\prec V_\eta: 
\{S_i:i<\xi\}\in M\text{ and }\exists S_i\in M: M\cap \cup S_i\in S_i\}
\]
It can be seen that these definitions are independent of 
the choice of the ordinal $\eta$.

We say that $S$ and $T$ are compatible if
$S\wedge T$ is stationary.
Moreover it can be checked that $\bigwedge$ and $\bigvee$ are 
exact lower and upper bounds for $\leq$.

\subsection{Large cardinals}\label{subsec:lc}
We shall repeatedly use supercompact cardinals which for us are defined as follows:
\begin{definition}
$\delta$ is $\gamma$-supercompact if for all $S\in V_\gamma$ there is 
an elementary 
\[
j:V_\eta\to V_\gamma
\]
with $j(\crit(j))=\delta$ and $S\in j[V_\eta]$.
$\delta$ is supercompact if it is $\gamma$-supercompact for all $\gamma\geq\delta$.
\end{definition}

We shall also use this equivalent characterization of supercompactness:
\begin{proposition}
The following are equivalent for a limit ordinal $\gamma>\delta$:
\begin{itemize}
\item
$\delta$ is $\xi$-supercompact for all $\xi<\gamma$,
\item
For all $\xi<\gamma$ the set 
\[
\{M\prec V_\xi: M\cap\delta\in \delta\text{ and } \pi_M\text{ is an isomorphism of $M$ with some }V_\alpha\}
\]
is stationary.
\end{itemize}
\end{proposition}

We shall also repeatedly mention Woodin cardinals, however we shall never actually 
need to employ them, so we dispense with their definition and we remark that if there is an elementary
$j:V_{\alpha+1}\to V_{\delta+1}$ than both $\alpha$ and $\delta$ 
are Woodin cardinal
and any normal measure on
$\delta$ concentrates on Woodin cardinals below $\delta$.

\begin{definition}\label{def:reflcar}
$\delta$ is a $\Sigma_2$-reflecting cardinal if it is inaccessible and for all formulae
$\phi(p)$ with $p\in V_\delta$ 
\[
\exists \gamma V_\gamma\models\phi(p)
\] 
if and only if there exists $\alpha<\delta$ such that
\[
V_\alpha\models\phi(p).
\] 
\end{definition}
Finally we will need the notion of super almost huge cardinals
and of Laver function for almost huge embeddings 
which for us are defined as follows:
\begin{definition}\label{def:sahlc}
An elementary 
$j:V\to M$ with $\crit(j)=\delta$ is almost huge 
if $M^{<j(\delta)}\subset M\subset V$.

$\delta$ is super almost huge if for all $\lambda>\delta$ there is an almost huge
$j:V\to M$ with $\crit(j)=\delta$ and $j(\delta)>\lambda$.

$f:\delta\to V_\delta$ is a Laver function for almost huge embeddings if for all
$X$ there is an almost huge $j:V\to M$ with $\crit(j)=\delta$ such that $j(f)(\delta)=X$.
\end{definition}

\begin{fact}~\cite[Theorem 12, Fact 13]{COX12}
Assume $j:V\to M$ is elementary and such that $M^{j^2(\delta)}\subset M\subset V$.
Then $V_{j(\delta)}$ models that $\delta$ carries a Laver function for almost huge embeddings.  
\end{fact}

\subsection{Posets and their boolean completions}
\label{subsec:boolalgforc}

We refer the reader to~\cite{VIAAUDSTE13} for a detailed account
on the results presented in this and the next subsections.

Given a partial order $(P,\leq_P)$ we let $\BBB(P)$ denote its boolean completion given by the
regular open set in the
order topology on $P$ (the topological space whose points are the elements of $P$ and 
whose open sets are the downward closed subsets of $P$ with respect to $\leq_P$).

Given partial orders $(P,\leq_P)$, $(Q,\leq_Q)$,
$i:P\to Q$ is a regular embedding if
it is order and incompatibility preserving and maps maximal antichain to maximal antichains.

Remark that if $i:P\to Q$ is a regular embedding $i$ gives rise to a complete injective homomorphism
$\bar{i}:\BBB(P)\to\BBB(Q)$ mapping a regular open set $A\subseteq P$ (in the
order topology on $P$) to the regular open set\footnote{For a given topological space
$(X,\tau)$ and $B\subseteq X$ $\mathring{\overline{B}}$ denote the
interior of the closure of $B$ in the topology $\tau$.}
$\mathring{\overline{i[A]}}$ (in the
order topology on $Q$).

Let $G$ be $V$-generic for $P$.
The quotient forcing $(Q/i[G],\leq_{Q/_{i[G]}})$ is the partial order defined in $V[G]$ 
as follows:
\begin{itemize}
\item
$q\in Q/i[G]$ 
if $q\in Q$ is compatible with all elements in $i[G]$
\item 
$q\leq_{Q/_{i[G]}} r$ if for all $s\in Q/i[G]$ such that $s\leq_Q q$ there is
$t\leq_Q s, r$ such that $t\in Q/i[G]$.
\end{itemize}
We identify $Q/i[G]$ with its separative quotient given
by equivalence classes of the form $[q]_{i[G]}$ induced by the order relation 
$\leq_{Q/_{i[G]}}$. 

In almost all cases (with the notable exception of the discussion around Theorem~\ref{lem:univ4})
it will be more convenient for us to deal with these 
concepts in the language of boolean algebras rather than in the language of posets, thus 
we introduce also the following notation:

Let $\BB$ and $\QQ$ be complete boolean algebras.
We say that
a complete homomorphism (or complete embedding) $i:\BB\to\QQ$ is \emph{regular} if it is injective.
We remark that if $i:\BB\to\QQ$ is a complete homomorphism and
$\BB$ is atomless $i[\BB]$ is a complete and atomless subalgebra of $\QQ$ isomorphic via $i$
to $\BB\restriction \coker(i)$ where
\[
\coker(i)=\neg_{\BB}\bigvee_{\BB}\{b:i(b)=0_{\QQ}\}.
\]
We also remark that if $i:\BB\to\QQ$ and $J$ is an ideal on $\BB$, letting 
$K=\downarrow i[J]$ (where $\downarrow X$ is the downward closure of $X$),
we can define
$i/J:\BB/J\to \QQ/K$ by $[b]_J\mapsto [i(b)]_K$. 

The above notion of quotients introduced for posets and boolean algebras are correlated as follows:
In the case that $G$ is $V$-generic for $\BBB(P)$, $J$ is its dual ideal 
and $i:P\to Q$ is a regular embedding,
we get that in $V[G]$ $\BBB(Q/i[G])^{V[G]}$ is isomorphic to $\BBB(Q)^V/i[J]$.

Also, assuming $G$ is a filter on $\BB$ and $J$ is its dual ideal, 
we shall feel free to denote by $i/G$ and $\BB/G$
the boolean algebra $\BB/J$ and the homomorphism $i/J$.
We define
$V^\BB$ as the class of $\tau\in V$ such that $\tau:V^\BB\to\BB$.
The canonical name in $V^{\BB}$ for the $V$-generic filter is denoted by $\dot{G}_{\BB}$,
we will let $\val_G(\tau)=\{\val_G(\sigma):\tau(\sigma)\in G\}$ denote the evalution map induced by
a $V$-generic filter 
$G$ for $\BB$  on $\BB$-names.

\subsection{$\SSP$-forcings and $\SSP$-correct embeddings}\label{subsec:SSPforc}

\begin{proposition}\cite[Proposition 2.11]{VIAAUDSTE13}
Assume $i:\BB\to\QQ$ is a complete homomorphism.
Let $\hat{i}:V^{\BB}\to V^{\QQ}$ be defined by the requirement that
\[
\hat{i}(\tau)(\hat{i}(\sigma))=i\circ \tau(\sigma)
\]
for all $\sigma\in\dom(\tau)\in V^{\BB}$.
Then for all provably $\Delta_1$-properties $\phi(x_0,\dots,x_n)$
\[
i(\llbracket\phi(\tau_0,\dots,\tau_n)\rrbracket_{\BB})=
\llbracket\phi(\hat{i}(\tau_0),\dots,\hat{i}(\tau_n))\rrbracket_{\QQ}.
\]
\end{proposition}

We denote by $\SSP$ the class of stationary set preserving partial orders and
by $\SP$ the class of semiproper partial orders 
(see~\cite[Def. 6.1, Def. 6.4]{VIAAUDSTE13} for  a definition of semiproperness).
We recall that semiproper posets are stationary set preserving and that
$P$ is stationary set preserving if 
\[
P\Vdash S\text{ is stationary}
\]
for all $S$ in $V$ stationary subset of $\omega_1$.

Given $i:\BB\to\QQ$ we let $\QQ/i[\dot{G}_{\BB}]$ be a $\BB$-name for the quotient
(living in $V[G]$ with $G$ $V$-generic filter for $\BB$)
of the boolean algebra $\QQ$ by the ideal generated by the dual of $i[G]$
(which we shall denote by $\QQ/i[G]$).

The following type of complete embeddings will be of central interest for us:
\begin{definition}
Let $\Gamma$ be a definable class of partial orders.
A complete homomorphism $i:\BB\to\QQ$ is $\Gamma$-correct 
if $\QQ\in\Gamma$ and
\[
\llbracket \QQ/i[\dot{G}_{\BB}]\in\Gamma\rrbracket_{\BB}=1_{\BB}.
\]
\end{definition}

We shall repeatedly use the following facts:
\begin{fact}\label{fac:charSSP}
Assume $P\in\SSP$ and $\dot{Q}\in V^{\BBB(P)}$ is a $P$-name for a partial order.
Let $\BBB(P)=\BB$, $\BBB(P*\dot{Q})=\QQ$, and
$i:\BB\to\QQ$ be the associated 
canonical embedding between the respective boolean completions.
Then $P\Vdash\dot{Q}\in\SSP$ (or equivalently
$i:\BBB(P)\to\BBB(P*\dot{Q})$ is $\SSP$-correct)
iff for all
$\dot{S}\in V^{\BBB(P)}$
\[
i(\llbracket \dot{S}\text{ is a stationary subset of }\omega_1\rrbracket_{\BB})=
\llbracket \hat{i}(\dot{S})\text{ is a stationary subset of }\omega_1\rrbracket_{\QQ}.
\]

\end{fact}

\begin{definition}
Let $H$ be $V$-generic for $\QQ$ and $G\in V[H]$ be $V$-generic for $\BB$.
We say that $G$ is $\SSP$-correct if\footnote{$\NS_{\omega_1}$ stands for the 
nonstationary ideal on $\omega_1$.} $\NS_{\omega_1}^{V[G]}=\NS_{\omega_1}^{V[H]}\cap V[G]$. 
\end{definition}

\begin{fact}
Let $\QQ\in\SSP$ and $i:\BB\to\QQ$ be a complete homomorphism.
Then $i$ is $\SSP$-correct iff for all $H$ $V$-generic filters for $\QQ$ we have that
$i^{-1}[H]$ is an $\SSP$-correct $V$-generic filter for $\BB$.
\end{fact}

We feel free (except in some specific cases arising in the next section)
to write $i:\BB\to\QQ$ is correct and $G\in V[H]$ is correct 
to abbreviate $i$, $G$ are $\SSP$-correct.

We shall also need the following:
\begin{proposition}\label{lem:firstfctlem-1}
	Let $\Gamma$ be either the class $\SP$ or the class $\SSP$.
	Let $\QQ$, $\QQ_0$, $\QQ_1$ be complete boolean algebras, and let $G$ be a $V$-generic 
	filter for $\QQ$. Let $i_0$, $i_1$, $j$ form a commutative diagram of complete homomorphisms 
	as in the following picture:
	\[
	\begin{tikzpicture}[xscale=1.5,yscale=-1.2]
		\node (B) at (0, 0) {$\QQ$};
		\node (Q0) at (1, 0) {$\QQ_0$};
		\node (Q1) at (1, 1) {$\QQ_1$};
		\path (B) edge [->]node [auto] {$\scriptstyle{i_0}$} (Q0);
		\path (B) edge [->]node [auto,swap] {$\scriptstyle{i_1}$} (Q1);
		\path (Q0) edge [->]node [auto] {$\scriptstyle{j}$} (Q1);
	\end{tikzpicture}
	\]
	Then $j, i_0,i_1$ are $\Gamma$-correct homomorphisms in $V$
	iff in $V[G]$ $\QQ_0/i_0[G],\QQ_1/i_1[G]$ are both in $\Gamma$
	and $j/_G:\QQ_0/i_0[G]\to\QQ_1/i_1[G]$ is 
	a $\Gamma$-correct homomorphism.
	\end{proposition}
	For the case
	$\Gamma=\SP$ see	the proof of~\cite[ Proposition 7.4]{VIAAUDSTE13}.
	The case $\Gamma=\SSP$ can be proved along the same lines.
	
	We shall repeatedly apply the above Proposition in the following context:
	
	\begin{proposition}\label{lem:firstfctlem-2}
	Assume $G$ is $V$-generic for $\BB\in\SSP$.
	\begin{enumerate}
	\item\label{lem:firstfctlem-2-1}
	Assume 
	$k_j:\BB\to\QQ_j$ and
	$l_j:\QQ_j\to \RR$ are correct homomorphism in $V$ such that
	$l_0\circ k_0=l_1\circ k_1=l$.
	Then in $V[G]$
	both $l_j/_G:\QQ_j/_{k_j[G]}\to \RR/_{l[G]}$ are correct homomorphisms.
	\item\label{lem:firstfctlem-2-2}
	Assume $k_j:\BB\to\QQ_j$ are correct homomorphisms in $V$,
	$i_j:\QQ_j/_{k_j[G]}\to \QQ$ are correct homomorphisms in $V[G]$
	for $j=0,1$.
	Then there are in $V$:
	\begin{itemize} 
	\item
	$\RR\in \SSP$, 
	\item
	a correct homomorphism 
	$l:\BB\to\RR$,
	\item
	correct homomorphisms
	$l_j:\QQ_j\to\RR$ 
	\end{itemize}
	such that:
	\begin{itemize}
	\item
	$\QQ$ is isomorphic to $\RR/_{l[G]}$ in $V[G]$.
	\item
	$l_j/_G=i_j$ for $j=0,1$ (modulo the isomorphism identifying $\RR/_{l[G]}$ and $\QQ$),
	\item
	$l_j\circ k_j=l$ for $j=0,1$,
	\item
	$0_{\RR}\not\in l[G]$.
	\end{itemize}
	\end{enumerate}
	\end{proposition}
\begin{proof}
We sketch a proof of the second item.
Let $\dot{\QQ},\dot{i}_j\in V^{\BB}$ be such that for both $j$
$\val_G(\dot{i}_j)=i_j$ and
\[
b_j=\llbracket i_j:\QQ_j/_{k_j[\dot{G}_{\BB}]}\to \dot{\QQ} 
\text{ is a correct homomorphism}\rrbracket_{\BB}\in G.
\]	
Let $b=b_0\wedge b_1$ and define $l_j:\QQ_j\to	\RR=\BB\restriction b*\dot{\QQ}$
by 
\[
l_j(q)=\langle b\wedge\bigwedge_{\BB}\{a:k_j(a)\geq q\},\dot{i}_j([q]_{k_j[\dot{G}_{\BB}]})\rangle
\] 
and
$l:\BB\to\BB\restriction b*\dot{\QQ}$ by $l(r)=\langle b\wedge r,1_{\dot{\QQ}}\rangle$.
We leave to the reader to check that these definitions work.
\end{proof}

\begin{definition}
Assume $\lambda,\nu$ are regular cardinals.

$\BB$ is $<\lambda$-CC if all antichains in $\BB$ have size less than $\lambda$.

$\BB$ is $(<\nu,<\lambda)$-presaturated if for all $\gamma<\lambda$, all families
$\{A_\alpha:\alpha<\gamma\}$ of maximal antichains of $\BB$, and all 
$b\in\BB^+=\BB\setminus\{0_{\BB}\}$, there 
is $q\leq b$ in $\BB^+$ such that
\[
|\{a\in A_\alpha: a\wedge q>0_{\QQ}\}|<\nu
\] 
for all $\alpha<\gamma$.
\end{definition}

\begin{fact}
For all regular cardinals $\nu\geq\lambda$,
$\BB$ is $(<\nu,<\lambda)$-presaturated iff for all $\gamma<\lambda$ and all
$\dot{f}:\gamma\to\nu$ in $V^{\BB}$, $1_{\BB}$ forces that $\rng(\dot{f})$ is bounded below
$\nu$.
\end{fact}

\begin{definition}
Let $A=\{\BB_j:j\in J\}$ be a family of complete boolean algebras
and $i_j:\BB\to\BB_j$ be a complete homomorphism for all $j\in J$.
We let:
\begin{itemize}
\item $\BB_A=\bigvee_{j\in J}\BB_j$ (the lottery sum of $A$) be defined as the boolean algebra
given by the set of functions 
$f:\gamma\to\bigcup\{\BB_i:i<\gamma\}$ such that $f(i)\in\BB_i$ for all $i<\gamma$
with boolean operations given componentwise
by 
\begin{itemize}
\item
$f\wedge_{\BB_A} g=(f(i)\wedge_{\BB_i}g(i):i<\gamma)$,
\item
$\bigvee_{\BB_A} \{f_\nu:\nu<\theta\}=(\bigvee_{\BB_i} \{f_\nu(i):\nu<\theta\}:i<\gamma)$,
\item
$\neg_{\BB_A} f=(\neg_{\BB_i} f(i):i<\gamma)$.
\end{itemize}
\item
$i_A:\BB\to \BB_A$ be defined by 
$i_A(b)=\langle i_j(b):j\in J\rangle$.
\end{itemize}
\end{definition}

We leave to the reader to check the following propositions:
\begin{proposition}\label{prp:lotterysum}
Let $A=\{\BB_j:j\in J\}$ be a family of complete boolean algebras
and $i_j:\BB\to\BB_j$ be an $\SSP$ (or $\SP$) correct homomorphism for all $j\in J$.
Then $i_A$ is also an $\SSP$ ($\SP$) correct homomorpshism.
\end{proposition}

\begin{proposition}
Assume $A$ is a maximal antichain of $\BB$.
Then 
$\BB$ is isomorphic to $\bigvee_{a\in A}\BB\restriction a$.
\end{proposition}

\subsection{$\MM^{++}$}\label{subsec:forcax}
\begin{definition}
Let $\BB\in V$ be  a complete boolean algebra and $M\prec H_{|\BB|^+}$.

$H\subset \BB\cap M$ is $M$-generic for $\BB$ if $H$ is a filter on the boolean algebra
$\BB\cap M$
and $H\cap D\neq\emptyset$
for all $D\in M$ predense subsets of $\BB$. 

Given an $M\prec H_{|\BB|^+}$ such that $\omega_1\subset M$,
let $\pi_M:M\to N$ be the transitive collapse map.

$H\subset \BB$ is a \emph{correct} $M$-generic filter for $\BB$ 
if it is $M$-generic for $\BB$ and 
letting $G=\pi_M[H]$ we have that $N[G]$ is stationarily correct i.e.:
\[
\NS_{\omega_1}^{N[G]}=\NS_{\omega_1}^V\cap N[G],
\]
where $N[G]=\{\val_G(\tau):\tau\in N^{\pi_M(\BB)}\}$.
\end{definition}

\[
T_{\BB}=\{M\prec H_{|\BB|^+}: M\in P_{\omega_2} \text{ and there is a
correct $M$-generic filter for $\BB$}\}.
\]

We formulate $\MM^{++}$ as the assertion 
that $T_{\BBB(P)}$ is stationary for all 
$P\in\SSP$.
For an equivalence of this formulation with other 
more common formulations see~\cite[Lemma 3]{COX12}.

We also say that $\MM^{++}$ up to $\alpha$ holds if 
$T_{\BBB(P)}$ is stationary for all $P\in\SSP$ with a dense 
subset of size less than $\alpha$.

We shall need the following local version of the celebrated proof of the consistency of 
$\MM$ by Foreman Magidor and Shelah:
\begin{theorem}\label{thm:FMSvar}
Assume $\delta$ is $\delta+\omega+1$-supercompact and $\BB\in\SSP\cap V_\delta$.
Then there is $\QQ$ of size $\delta$ and
$i:\BB\to \QQ$ such that
\begin{itemize}
\item
$\llbracket \QQ/_{k[\dot{G}_{\BB}]}\text{ is semiproper}\rrbracket_{\BB}=1_{\BB}$,

\item
$\QQ$ forces $\MM^{++}$ up to $\beth(\omega)$ while collapsing $\delta$ to become
$\omega_2$.
\end{itemize}
\end{theorem}
\begin{proof}
Sketch:
adapt the proof of the original result as presented in~\cite[Theorem 37.9]{JEC03} or
in~\cite[Theorem 8.5]{VIAAUDSTE13} to the different assumptions we are making on $\delta$.
\end{proof}

\subsection{Iterated forcing}\label{subsec:itfor}
We feel free to view iterations following the 
approach presented in~\cite{VIAAUDSTE13} which expands on the work of Donder and Fuchs on 
revised countable support iterations~\cite{DONFUC92}. 
We refer the reader to~\cite[Section 3]{VIAAUDSTE13} for the relevant definitions
and to sections 6 and 7 of the same paper for the relevant proofs.
Here we recall the minimal amount of information we need to make sense of our use of 
these results in this paper.

\begin{definition} \label{def:retr}
	Let $i: \BB \to \QQ$ be a regular embedding, the \emph{retraction} associated to $i$ is the map
	\[
	\begin{array}{llll}
		\pi_i :& \QQ &\to& \BB \\
		&c &\mapsto& \bigwedge \{b \in \BB: ~ i(b) \geq c\}
	\end{array}
	\]
\end{definition}

\begin{proposition}\cite[Proposition 2.6]{VIAAUDSTE13} \label{eRetrProp}
	Let $i:\BB\to\QQ$ be a regular embedding, $b \in \BB$, $c,d \in \QQ$ be arbitrary. Then,
	\begin{enumerate}
		\item $\pi_i\circ i(b)= b$ hence $\pi_i$ is surjective;
		\item \label{eRPComp} $i\circ\pi_i(c)\geq c$ hence $\pi_i$ maps 
		$\QQ^+=\QQ\setminus\{0_{\QQ}\}$ to $\BB^+=\BB\setminus\{0_{\BB}\}$;
		\item \label{eRPJoins} $\pi_i$ preserves joins, i.e. 
		$\pi_i (\bigvee X)= \bigvee \pi_i[X]$ for all $X \subseteq \QQ$;
		\item $i(b)= \bigvee \{e : \pi_i (e) \leq b \}$.
		\item \label{eRPHomo} 
		$\pi_i (c \wedge i(b)) = \pi_i(c)\wedge b = \bigvee \{ \pi_i(e): e \leq c, \pi_i(e) \leq b\}$;
		\item \label{eRPMeets} $\pi_i$ does not preserve neither meets 
		nor complements whenever $i$ is not surjective, 
		but $\pi_i(d \wedge c) \leq \pi_i(d) \wedge \pi_i(c)$ and $\pi_i(\neg c) \geq \neg \pi_i(c)$.
	\end{enumerate}
\end{proposition}

For a definition of semiproperness and of semiproper embedding 
see~\cite[Def 6.1, Def. 6.4]{VIAAUDSTE13}.
\begin{definition}
\[
\mathcal{F}=\{i_{\alpha,\beta}:\BB_\alpha\to \BB_\beta:\alpha\leq \beta<\gamma\}
\] 
is an iteration system 
if each $\BB_\alpha$ is a complete boolean algebra and 
each $i_{\alpha,\beta}$ is a regular embedding such that:
\begin{itemize} 
\item
whenever
 $\alpha\leq \beta\leq \gamma$, $i_{\beta,\gamma}\circ i_{\alpha,\beta}=i_{\alpha,\gamma}$,
 \item
 $i_{\alpha,\alpha}$ is the identity mapping for all $\alpha<\gamma$.
\end{itemize}

\begin{itemize}
\item
The full limit $T(\FFF)$ of $\FFF$ is given by threads $f:\gamma\to V$ such that
$\pi_{\alpha,\eta}\circ f(\eta)=f(\alpha)$ for all $\alpha\leq\eta<\gamma$, where
$\pi_{\alpha,\eta}$ is the retraction associated to $i_{\alpha,\eta}$.
For $f,g\in T(\FFF)$, 
$f\leq_{T(\FFF)} g$ iff for all $\alpha<\gamma$, $f(\alpha)\leq_{\BB_\alpha}g(\alpha)$.
\item 
The direct limit $C(\FFF)$ of an iteration system 
\[
\mathcal{F}=\{i_{\alpha,\beta}:\BB_\alpha\to \BB_\beta:\alpha\leq \beta<\gamma\}
\]  
is the partial order 
whose elements are the eventually constant threads
$f\in T(\FFF)$, i.e. threads $f$ such that
for some $\alpha<\gamma$ and all $\beta>\alpha$,
$i_{\alpha\beta}\circ f(\alpha)=f(\beta)$. 
For $f\in C(\FFF)$ the least such $\alpha$ 
is called the support of $f$.
\item
The revised countable support (RCS) limit is
	\[
	\RCS(\FFF) = \{ f \in T(\FFF) : ~ 
	f \in C(\FFF) \vee \exists \alpha ~ f(\alpha) \Vdash_{\BB_\alpha} 
	\cf(\check{\lambda}) = \check{\omega} \}
	\]
\end{itemize}
\[
\mathcal{F}=\{i_{\alpha,\beta}:\BB_\alpha\to \BB_\beta:\alpha\leq \beta<\gamma\}
\] 
is a semiproper iteration system if for all $\alpha\leq\beta<\gamma$:
\begin{itemize}
\item
$\llbracket \BB_\beta/i_{\alpha\beta}[\dot{G}_\alpha]\text{ is semiproper}\rrbracket_{\BB_\alpha}=
1_{\BB_\alpha}$,

\item
$\BB_{\alpha+1}$ forces that $|\BB_\alpha|=\aleph_1$,
\item
$\BB_\alpha$ is the boolean completion of $\RCS(\FFF\restriction\alpha)$ if $\alpha$ is limit.
\end{itemize}
\end{definition}

We shall need the following two results of Shelah:
\begin{enumerate}
\item\label{thmx:she1}
(Shelah~\cite[Theorem 7.11]{VIAAUDSTE13}).
Let
\[
\mathcal{F}=\{i_{\alpha,\beta}:\BB_\alpha\to \BB_\beta:\alpha\leq \beta<\gamma\}
\]  
be a semiproper iteration system.
Then its RCS limit is a semiproper partial order.
\item \label{thmx:she2}
(Shelah~\cite[Theorem 39.10]{JEC03})
Assume $\alpha$ is a limit ordinal and $T_{\BB}$ is stationary for all $\BB\in\SSP\cap V_\alpha$.
Then any $\BB\in \SSP\cap V_\alpha$ is semiproper.
\end{enumerate}

\section{Category forcings}\label{sec:univpartord}

Assume $\Gamma$ is a class of partial orders and
$\Theta$ is a family of complete homomorphisms between the 
boolean completions of elements of $\Gamma$
closed under composition and which contains all identity maps.

We let $(\Gamma,\Theta)$ denote the category whose objects
are the complete boolean algebras in $\Gamma$ 
and whose arrows are given by complete homomorphisms
$i:\BB\to\QQ$ in $\Theta$.
We shall call embeddings in $\Theta$, $\Theta$-correct embeddings.
Notice that these categories immediately give rise to natural class partial orders
associated with them, 
partial orders whose elements are the complete boolean algebras in $\Gamma$
and whose order relation is given by the arrows in $\Theta$.
We denote these class partial orders by $\UU^{\Gamma,\Theta}$.
Depending on the choice of $\Gamma$ and $\Theta$ these partial orders can be trivial,
for example:
\begin{remark}
Assume $\Gamma$ is the class of all complete boolean algebras and 
$\Theta$ is the class of all complete embeddings, then any
two conditions in $\UU^{\Gamma,\Theta}$ are compatible, i.e.
$\UU^{\Gamma,\Theta}$ is forcing equivalent to the trivial partial order.
This is the case since for any pair of partial orders $P,Q$ and $X$ of size larger than
$2^{|P|+|Q|}$ there are 
regular embeddings of $\BBB(P)$ and $\BBB(Q)$ into the boolean completion of
$Coll(\omega,X)$. These embeddings witness the compatibility of $\BBB(P)$ with $\BBB(Q)$.
\end{remark}

Since we want to allow ourselves more freedom in the handling of our class forcings
$\UU^{\Gamma,\Theta}$ 
we allow elements of the category $\Gamma$ to be arbitrary partial orders\footnote{
Specifically our main aim is to show that for certain categories 
$(\Gamma,\Theta)$
$\Gamma\cap V_\delta\in \UU^{\Gamma,\Theta}$. In general
$(\Gamma\cap V_\delta,\leq_\Theta\cap V_\delta)$ is a non-separative partial order.} 
in $\Gamma$ and we 
identify the arrows in $\Theta$ between the objects 
$P$ and $Q$ in $\Gamma$ to be the $\Theta$-correct homomorphisms
between the boolean completions of $P$ and $Q$.
In this paper we shall actually focus on the class forcing $\UU^{\SSP}$ given by the class
of stationary set preserving partial orders and
the family of all $\SSP$-correct
homomorphisms between their boolean completions.
The main reason being that it is just for this category that we can predicate all the 
 properties of category forcings in which we are interested.
 We shall just write that $i$ is a correct embedding whenever 
$i$ is an $\SSP$-correct embedding.

\subsection{Basic properties of $\UU^{\SSP}$}

We start to outline some basic properties of this category forcing.

\subsubsection*{Incompatibility in $\UU^{\SSP}$}

First of all we show that $\UU^{\SSP}$ is non-trivial:
\begin{remark}
 $\UU^{\SSP}$ seen as a class partial order is not a trivial partial order.
 For example observe that if $P$ is Namba forcing on $\aleph_2$ and $Q$ is
 $Coll(\omega_1,\omega_2)$, then $\BBB(P),\BBB(Q)$ are incompatible conditions in 
 $\UU^{\SSP}$:
 If $\RR\leq_{\SSP}\BBB(P),\BBB(Q)$, we would have that if $H$ is $V$-generic for 
 $\RR$, $\omega_1^{V[H]}=\omega_1$ (since $\RR\in\SSP$)
 and there are $G,K\in V[H]$ $V$-generic filters for 
 $P$ and $Q$ respectively (since $\RR\leq_{\SSP}\BBB(P),\BBB(Q)$).
 $G$ allows to define in $V[H]$ a sequence cofinal in $\omega_2^V$ of type $\omega$ while $K$
allows to define in $V[H]$ a sequence cofinal in $\omega_2^V$ of type $(\omega_1)^V$.
These two facts entail that $V[H]$ models that $\cof(\omega_1^V)=\omega$ 
contradicting the assumption that $\omega_1^{V[H]}=\omega_1$.
\end{remark}

\subsubsection*{Suprema in $\UU^{\SSP}$}

The lottery sum defines a natural $\bigvee$ operation of suprema on subsets of $\UU^{\SSP}$.

\begin{proposition}
Let $A=\{\BB_i:i<\gamma\}$ be a family of stationary set preserving complete boolean algebras.
Then $\BB_A$ is the exact upper bound of $A$:
\end{proposition}
\begin{proof}
Left to the reader.
\end{proof}

\subsubsection*{Why this ordering on $\SSP$ partial orders?}

Given a pair $(\Gamma,\Theta)$ as above,
we can define two natural order relations $\leq_\Theta$ and $\leq^*_\Theta$ on $\Gamma$.
The first one is given by 
complete homomorphisms $i:\BB\to\QQ$ in $\Theta$
(which is the one we described before)
and the other given by \emph{regular} (i.e. complete and injective homomorphisms) 
embeddings $i:\BB\to\QQ$ in $\Theta$. Both notion of orders 
are interesting and as set theorists we are used to focus on this second stricter notion of order since it 
is the one suitable to develop a theory of iterated forcing. However in the present paper we 
shall focus mostly on complete (but possibly non-injective) homomorphisms
because this notion of ordering will grant us that whenever $\BB$ is put into a $V$-generic filter
for $\UU^{\SSP}\cap V_\delta=\UU_\delta$, 
then this $V$-generic filter for $\UU_\delta$ 
will also add a $V$-generic filter for $\BB$. 
If we decided to order the family $\SSP\cap V_\delta$ using 
\emph{regular} embeddings we would get that a generic filter for this other category forcing
defined 
according to this stricter notion of order will just give a directed 
system of $\SSP$-partial orders with regular embeddings between them, without
actually giving $V$-generic filters for the partial orders in this directed system.
On the other hand if we use the iteration theorems for the class of semiproper 
forcings we actually get the following:

\begin{proposition}
Let 
\[
\mathcal{F}=\{i_{\alpha,\beta}:\BB_\alpha\to\BB_\beta: \alpha\leq\beta<\gamma\}
\]
be an iteration system such that $i_{\alpha,\beta}$ is $\SP$-correct for all $\alpha\leq\beta<\gamma$
with $\BB_0\in\SSP$.
Then $\RCS(\mathcal{F})\in\SSP$ as well.
\end{proposition}
This gives quite easily that the class forcing $(\SSP,\leq^*_{\SP})$ is closed under 
set sized descending sequences.
This observation is useful to prove nice weak distributivity properties of the class forcing 
$(\SSP,\leq_{\SP})$.
On the other hand there is a key combinatorial feature of the class forcing $(\SSP,\leq_{\SSP})$
(freezeability, see Def. \ref{def:freecat}) 
which we are not able to predicate for the class forcing $(\SSP,\leq_{\SP})$, unless we assume 
large cardinals. 
In the presence of supercompact cardinals
the equality $\SP=\SSP$ can be forced by a semiproper forcing.
In particular we aim to use large cardinals to expand the above equality to the extent to be able to
identify the class forcings $(\SSP,\leq_{\SSP})$ and $(\SSP,\leq_{\SP})$ on a dense subset
$\TR$  (the class of totally rigid elements of $\SSP$ which force $\MM^{++}$, 
see Definition~\ref{def:totrig}). In this way
on $(\TR,\leq_{\SSP})=(\TR,\leq_{\SP})$ we can have at the same time 
the nice closure properties of $\leq^*_{\SP}$
with the nice combinatorial features of $\leq_{\SSP}$.

\subsubsection*{Rank initial segments of $\UU^{\SSP}$ are stationary set preserving posets}
The first main result of this section is the following theorem:

\begin{theorem}\label{thm:univ1}
Assume $\delta$ is an inaccessible limit of $<\delta$-supercompact cardinals.
Then: 
\begin{enumerate}
\item\label{thm:univ1-1}
$\UU_\delta=\UU^{\SSP}\cap V_\delta\in\SSP$ is totally rigid 
and collapses $\delta$ to become $\aleph_2$.
\item\label{thm:univ1-2}
If  $\delta$ is supercompact $\UU_\delta$ forces $\MM^{++}$.
\end{enumerate}
\end{theorem}

\subsection{Totally rigid partial orders for $\UU^{\SSP}$} \label{sec:trpo}

Total rigidity is the key property which we would like to be able to predicate for a category forcing.

\begin{definition}\label{def:totrig}
Assume $P$ is a partial order in $\Gamma$.
$P$ is $\Theta$-totally rigid if 
for no complete boolean algebra $\CC\in\Gamma$ there are
distinct complete homomorphisms $i_0:\BBB(P)\to\CC$ and 
$i_1:\BBB(P)\to\CC$
in $\Theta$.
\end{definition}

We shall see that the class of $\SSP$-totally rigid partial orders is dense in $\UU^{\SSP}$.
This result will be the cornerstone on which
we will elaborate to get the desired generic absoluteness theorem.
As it will become transparent to the reader at the end of this paper, 
we should be able to prove the appropriate form of generic absoluteness for any ``reasonable''
class of forcings $\Gamma$
for which we can predicate the existence of a dense class of totally rigid partial orders in $\UU^\Gamma$
and for which we have an iteration theorem.
We shall from now on
just say totally rigid to abbreviate $\SSP$-totally rigid.

These properties give equivalent characterizations of totally rigid boolean algebras:
\begin{lemma}\label{lem:eqtr}
The following are equivalent:
\begin{enumerate}
\item\label{lem:eqtr1}
for all $b_0,b_1\in \BB$ such that
$b_0\wedge_{\BB}b_1=0_{\BB}$ we have that
$\BB\restriction b_0$ is incompatible with $\BB\restriction b_1$ in $\UU^{\SSP}$, 
\item\label{lem:eqtr2}
For all $\QQ\leq_{\SSP}\BB$ and all $H$, $V$-generic filter for $\QQ$, there is just one $G\in V[H]$ 
correct $V$-generic filter for $\BB$.
\item\label{lem:eqtr4}
For all $\QQ\leq\BB$ in $\UU^{\SSP}$ there is only one complete homomorphism
$i:\BB\to\QQ$ such that
\[
\BB\Vdash \QQ/i[\dot{G}_{\BB}]\text{ is stationary set preserving.}
\]
\end{enumerate}
\end{lemma}

\begin{proof} 
We prove these equivalences as follows:
\begin{itemize}
\item
We first prove~\ref{lem:eqtr4} implies~\ref{lem:eqtr2} by contraposition.
Assume~\ref{lem:eqtr2} fails for $\BB$ as witnessed by some $\QQ\leq_{\SSP}\BB$, $H$ 
$V$-generic filter for $\QQ$, and $G_1\neq G_2\in V[H]$ 
correct $V$-generic filters for $\BB$.

Let $\dot{G}_1,\dot{G}_2\in V^{\QQ}$ and $q\in H$ be such that
$q$ is the boolean value of the statement
\begin{quote}\label{eqn:SSP}
$\dot{G}_1\neq \dot{G}_2$
are $V$-generic filters for $\BB$ and 
$V[\dot{H}]$ is a $V[\dot{G}_j]$-generic extension by an
$\SSP$ forcing in $V[\dot{G}_j]$ for $j=1,2$.
\end{quote}
Notice that the above statement is expressible by a forcing formula in the parameters
\[
H_{\omega_2}^{V[\dot{G}_j]},
\omega_1,\dot{G}_1,\dot{G}_2,\dot{G}_{\QQ},\check{\BB},\check{H_{|\BB|^+}}
\] 
stating that: 
\begin{quote}
for both $j=1,2$,
$\dot{G}_j$ are distinct filters on $\check{\BB}$ meeting all dense subsets of 
$\check{\BB}$ in $\check{H_{|\BB|^+}}$, and for all 
$\dot{S}$ in $H_{\omega_2}^{V[\dot{G}_j]}$
$\dot{S}$ is a stationary subset of $\check{\omega_1}$ in $H_{\omega_2}^{V[\dot{G}_j]}$ iff
it is such in $V[\dot{G}_{\QQ}]$.
\end{quote}

Let $r\leq_{\QQ} q$ and $b\in \BB$ be such that $r\Vdash_{\QQ}b\in\dot{G}_1\setminus\dot{G}_2$.

Define $i_j:\BB\to\QQ\restriction r$ by
$a\mapsto \llbracket\check{a}\in\dot{G}_j\rrbracket\wedge r$.

Then we get that $i_0(b)=r=i_1(\neg b)$, thus $i_0\neq i_1$.
We can check that $i_0,i_1$ are correct embeddings as follows:
First of all
\[
r\leq_{\QQ}\llbracket i_j^{-1}[\dot{G}_{\QQ}]=\dot{G}_j\rrbracket_{\QQ}
\]
and 
\[
r\leq_{\QQ}\llbracket b\in \dot{G}_0\setminus\dot{G}_1\rrbracket_{\QQ}.
\]
Now observe that if $\dot{S}\in V^{\BB}$ is such that
\[
\llbracket \dot{S}\text{ is a stationary subset of $\omega_1$}\rrbracket_{\BB}=1_{\BB}
\]
We get that 
\[
\llbracket \hat{i}_j(\dot{S})\in V[\dot{G}_j]\rrbracket_{\QQ}\geq r
\]
and
\[
\llbracket \hat{i}_j(\dot{S})\text{ is a stationary subset of $\omega_1$ in } V[\dot{G}_j]\rrbracket_{\QQ}
\geq r.
\]
Thus since $r$ forces that $V[\dot{G}_{\QQ}]$ is a generic extension of $V[\dot{G}_j]$
preserving stationary subsets of $\omega_1$, we get that 
\[
\llbracket \hat{i}_j(\dot{S})\text{ is a stationary subset of $\omega_1$ }\rrbracket_{\QQ}
\geq r.
\]
This shows that 
$i_1,i_2$ are distinct correct embeddings witnessing that~\ref{lem:eqtr4} fails for $\BB$.

\item
Now we prove that~\ref{lem:eqtr1} implies~\ref{lem:eqtr4} again by contraposition.
So assume~\ref{lem:eqtr4} fails for $\BB$ as witnessed by $i_0\neq i_1:\BB\to\QQ$.
Let $b$ be such that $i_0(b)\neq i_1(b)$. 

W.l.o.g. we can suppose that $r=i_0(b)\wedge i_1(\neg b)>0_{\QQ}$
Then $j_0:\BB\restriction b\to \QQ\restriction r$ and
$j_1:\BB\restriction \neg b\to \QQ\restriction r$ given by $j_k(a)=i_k(a)\wedge r$ for $k=0,1$ and
$a$ in the appropriate domain witness that $\BB\restriction \neg b$ and $\BB\restriction b$
are compatible in $\UU^{\SSP}$ i.e.~\ref{lem:eqtr1} fails.

\item
Finally we prove that~\ref{lem:eqtr2} implies~\ref{lem:eqtr1} again by contraposition.

So assume~\ref{lem:eqtr1} fails as witnessed by $i_j:\BB\restriction b_j\to\QQ$ for $j=0,1$ with
$b_0$ incompatible with $b_1$ in $\BB$.
Pick $H$ $V$-generic for $\QQ$. Then $G_j=i_j^{-1}[H]\in V[H]$ are distinct  and correct
$V$-generic filters for $\BB$
since $b_j\in G_j\setminus G_{1-j}$.  
\end{itemize}

\end{proof}

The next subsections have as objective the proof of the following theorems
which are the cornerstones on which we shall develop our analysis of $\UU^{\SSP}$ and
are the other two main results of this section:

\begin{theorem}
\label{thm:distrSSP*}
Assume there are class many supercompact cardinals. 
Let $\TR$ be the class of partial orders $Q$ such that:
\begin{itemize}
\item $Q$ forces $\MM^{++}$ (and thus the equality $\SP=\SSP$),
\item $Q$ is totally rigid.
\end{itemize}
Then $\TR$ is dense in $\UU^{\SSP,\SP}$ and\footnote{
We remark that there is no typo in the statement of the above theorem, i.e for any
$\BB$ there is a \emph{regular embedding} $i:\BB\to \QQ$ such that
\[
\llbracket \QQ/_{i[\dot{G}_{\BB}]}\text{ is semiproper }\rrbracket_{\BB}=1_{\BB}
\]
and $\QQ$ forces $\MM^{++}$.
Notice that on the class $\TR$ we have that $\QQ\leq_{\SSP}\BB$ iff
$\QQ\leq_{\SP}\BB$ iff there is a unique $\SP$-correct $i:\BB\to\QQ$.
} 
in $\UU^{\SSP}$.
\end{theorem}

\begin{theorem}\label{lem:univ4}
Assume there are class many supercompact cardinals
and $\BB$ is a stationary set preserving
forcing.
Then there are in $V$
\begin{itemize} 
\item
a stationary set preserving complete boolean
algebra  $\QQ$,
\item
a regular embedding
\[
i_0:\BB\to\QQ
\]
\item
a regular embedding
\[
i_1:\BB\to \UU^{\SSP}\restriction\QQ
\]
\end{itemize}
such that 
whenever $H$ is $V$-generic for $\BB$,
then
$V[H]$ models that the class forcing
\[
(\UU^{\SSP}_\delta\restriction\QQ)^V/i_1[H]
\]
is identified with\footnote{ I.e. we will show that there is an order and incompatibility
preserving map with a dense image between these two class forcings.} 
the class forcing 
\[
(\UU^{\SSP})^{V[H]}\restriction(\QQ/i_0[H])
\]
as computed in $V[H]$.

Moreover the above factorization property reflects down to 
$\UU^\SSP\cap V_\delta=\UU_\delta$ whenever the latter is stationary set preserving,
$\TR\cap\UU_\delta$ is dense in $\UU_\delta$ and
$\BB$ in\footnote{I.e. with the same requirements for $\BB,\QQ,i_0,i_1$ as in 
the first conclusion of the theorem,
$\UU_\delta^{V[G]}\restriction(\QQ/i_0[G])$ 
is forcing equivalent in $V[G]$ to $(\UU_\delta^V\restriction\QQ)/i_1[G]$ 
whenever 
$G$ is $V$-generic for
$\BB\in V_\delta$, $\QQ\in\TR\cap V_\delta$, $i_0,i_1\in V_\delta$.}
$\UU_\delta$.

\end{theorem}

This factorization property of $\UU_\delta$ is not shared by
the other forcings which are used to prove the consistency of $\MM^{++}$.
It is due to this property of $\UU_\delta$ that we can prove the generic absoluteness results
given in Theorem~\ref{thm:mainthm}.

The following subsections show the proofs of Theorems~\ref{thm:distrSSP*},~\ref{lem:univ4} 
and of Theorem~\ref{thm:univ1}.

\subsubsection{Freezeability}\label{sec:denprptrpo}

We shall introduce the concept of freezeability and
use it to prove Theorem~\ref{thm:distrSSP*}.

\begin{definition}\label{def:freecat}
Let $(\Gamma,\Theta)$ be a category of  cbas and complete homomorphisms.
$k:\BB\to \QQ$ $\Theta$-freezes $\BB$ if for all $\RR\leq_\Theta\QQ$
and $i_j:\QQ\to\RR$ in $\Theta$ for $j=0,1$ we have that $i_0\circ k=i_1\circ k$.
$\QQ\leq_\Theta\BB$ $\Theta$-freezes $\BB$ if
there is a $k:\BB\to \QQ$ which $\Theta$-freezes $\BB$.

$\BB$ is $\Theta$-freezeable if there is some $\QQ\leq_\Theta\BB$ which $\Theta$-freezes $\BB$.
\end{definition}
We just say freezeable when meaning $\SSP$-freezeable.

We shall show that all stationary set preserving posets are freezeable.
We need an analogue of Lemma~\ref{lem:eqtr} to characterize freezeability.

\begin{lemma}\label{lem:eqfr}
Let $k:\BB\to\QQ$ be a correct homomorphism.
The following are equivalent:
\begin{enumerate}
\item\label{lem:eqfr1}
For all $b_0,b_1\in \BB$ such that
$b_0\wedge_{\BB}b_1=0_{\BB}$ we have that
$\QQ\restriction k(b_0)$ is incompatible with $\QQ\restriction k(b_1)$ in $\UU^{\SSP}$. 
\item\label{lem:eqfr2}
For all $\RR\leq_{\SSP}\QQ$, all $H$ $V$-generic filter for $\RR$, there is just one $G\in V[H]$ 
correct $V$-generic filter for $\BB$ such that $G=k^{-1}[K]$ for all $K\in V[H]$ correct $V$-generic 
filters for $\QQ$.
\item\label{lem:eqfr4}
For all $\RR\leq_\SSP\QQ$ in $\UU^{\SSP}$ 
and $i_0,i_1:\QQ\to\RR$ witnessing that $\RR\leq_{\SSP}\QQ$ we have that
$i_0\circ k=i_1\circ k$.
\end{enumerate}
\end{lemma}

\begin{proof} 
We prove these equivalences mimicking the proof of Lemma~\ref{lem:eqtr}:
\begin{itemize}
\item
We first prove~\ref{lem:eqfr4} implies~\ref{lem:eqfr2} by contraposition.
Assume~\ref{lem:eqfr2} fails for $k:\BB\to\QQ$ as witnessed by some $\RR\leq_{\SSP}\QQ$, $H$ 
$V$-generic filter for $\RR$, and $K_0\neq K_1\in V[H]$ correct $V$-generic filters for $\QQ$ with
$b\in k^{-1}[K_0]\setminus k^{-1}[K_1]$. 
Let $\dot{K}_i\in V^{\RR}$ be $\RR$-names for $K_i$ and
$r\in H$ force that:

\emph{$\dot{K}_i$ are correct $V$-generic filters for $\QQ$ and 
$b\in k^{-1}[\dot{K}_0]\setminus k^{-1}[\dot{K}_1]$.}

Now define for $j=0,1$,
$i_j:\QQ\to\RR\restriction r$ by 
\[
c\mapsto\llbracket c\in\dot{K}_j\rrbracket_{\RR}\wedge r.
\]
Now observe that $i_0\circ k(b)=1_{\RR\restriction r}$ while
$i_1\circ k(b)=0_{\RR\restriction r}$. In particular $i_0$ and $i_1$ are correct embeddings
of $\QQ$ into $\RR$ witnessing that~\ref{lem:eqfr4} fails (we leave to the reader to check the 
correctness of $i_0,i_1$ along the same lines of what was done in Lemma~\ref{lem:eqtr}).

\item
Now we prove that~\ref{lem:eqfr1} implies~\ref{lem:eqfr4} again by contraposition.
So assume~\ref{lem:eqfr4} fails for $k:\BB\to \QQ$ as witnessed by $i_0\neq i_1:\QQ\to\RR$.
Let $b$ be such that $i_0\circ k(b)\neq i_1\circ k(b)$. 

W.l.o.g. we can suppose that $r=i_0\circ k(b)\wedge i_1\circ k(\neg b)>0_{\QQ}$
Then $j_0:\BB\restriction b\to \RR\restriction r$ and
$j_1:\BB\restriction \neg b\to \RR\restriction r$ given by $j_l(a)=i_l\circ k(a)\wedge r$ for $l=0,1$ and
$a$ in the appropriate domain witness that $\BB\restriction \neg b$ and $\BB\restriction b$
are compatible in $\UU^{\SSP}$ i.e.~\ref{lem:eqfr1} fails.

\item
Finally we prove that~\ref{lem:eqfr2} implies~\ref{lem:eqfr1} again by contraposition.

So assume~\ref{lem:eqfr1} fails as witnessed by $i_j:\QQ\restriction k(b_j)\to\RR$ for $j=0,1$ with
$b_0$ incompatible with $b_1$ in $\BB$.
Pick $H$ $V$-generic for $\RR$. Then $G_j=i_j^{-1}[H]\in V[H]$ are distinct correct 
$V$-generic filters for $\QQ$ and $b_j\in k^{-1}[G_j]$ for $j=0,1$. Since $k$ is correct we also have
that $k^{-1}[G_j]$ are distinct correct $V$-generic filters for $\BB$ in $V[H]$ witnessing that
~\ref{lem:eqfr2} fails
(since $b_j\in k^{-1}[G_j]\setminus k^{-1}[G_{1-j}]$).  
\end{itemize}

\end{proof}

Freezeable posets can be embedded in $\UU^\SSP$
as follows:
\begin{lemma}
Assume $\QQ$ freezes $\BB$.
Let $k:\BB\to\QQ$ be a correct and complete embedding of
$\BB$ into $\QQ$ which witnesses it.
Then the map
$i:\BB\to\UU^\SSP\restriction\QQ$ which maps $b\mapsto \QQ\restriction k(b)$
is a complete embedding of partial orders.
\end{lemma}
\begin{proof}
It is immediate to check that $i$ preserve the order relation on
$\BB$ and $\UU^\SSP$. Moreover $\QQ$ freezes $\BB$
if and only if $i$ preserve the incompatibility relation. 
Thus we only have to check that $i[A]$ is a maximal antichain
below $\QQ$ in $\UU^{\SSP}$ whenever $A$ is a
maximal antichain of $\BB$.
If not there is $\RR\leq_{\SSP}\QQ$ such that $\RR$ is incompatible
with $\QQ\restriction k(b)$ for all $b\in A$.
This means that $\RR\leq_{\SSP}\QQ$ is incompatible with
\[
\QQ=\bigvee\{\QQ\restriction k(b):b\in A\},
\]
a contradiction.
\end{proof}

We refer the reader to Subsection~\ref{subsec:itfor} for the relevant definitions and results on iterations
and to~\cite{VIAAUDSTE13} for a detailed account.

\begin{lemma}\label{lem:freezelem*}
 Assume 
 \[
 \{i_{\alpha\beta}:\BB_\alpha\to\BB_\beta:\alpha<\beta\leq\delta\}
 \]
 is a complete iteration system such that
 for each $\alpha$ there is $\beta>\alpha$ such that
 \begin{itemize}
 \item
 $\BB_\beta$ freezes $\BB_\alpha$ as witnessed by the correct regular embedding
 $i_{\alpha,\beta}$.
 \item
 $\BB_\delta$ is the direct limit  of the iteration system and is stationary set preserving.
\end{itemize}
 Then
 $\BB_\delta$ is totally rigid.
 \end{lemma}

 \begin{proof}
 Assume the Lemma fails.
 Then there are $f_0,f_1$ incompatible thread in $\BB_\delta$ such that
 $\BB_\delta\restriction f_0$ is compatible with  
 $\BB_\delta\restriction f_1$ in $\UU^{\SSP}$. 
 Now $\BB_\delta$ is a direct limit, so
 $f_0,f_1$ have support in some  
 $\alpha<\delta$. Thus 
 $f_0(\beta)$, $f_1(\beta)$ are incompatible in 
$\BB_\beta$ for all $\alpha<\beta\leq\delta$.
Now for eventually all $\beta>\alpha$ $\BB_\beta$ freezes $\BB_\alpha$
as witnessed by $i_{\alpha,\beta}$.
In particular, since $f_i=i_{\alpha,\delta}\circ f_i(\alpha)$
for $i=0,1$ we get that $\BB_\delta\restriction f_0$ cannot be compatible with  
 $\BB_\delta\restriction f_1$ in $\UU^{\SSP}$, contradicting our assumption.
\end{proof}

The above Lemma shows that in order to prove Theorem~\ref{thm:distrSSP*} we just need to exhibit
an iteration system in which the direct limit is $\SSP$ and all its elements freeze their predecessors.

In the next subsection we shall define for any given $\SSP$ poset
$P$ a $P$-name $\dot{Q}$ for an $\SSP$-poset which freezes $P$.
In the subsequent one we will combine this result with the equality $\SP=\SSP$ 
(which holds in models of $\MM^{++}$) to define he desired iteration systems whose direct limits
are totally rigid and whose first factor can be any $\SSP$ poset.

\subsubsection{Freezing $\SSP$ posets}

We shall now define for any given stationary set preserving 
poset $P$ a poset $\dot{R}_{P}\in V^{P}$
such that $\QQ_{P}=P*\dot{R}_{P}$ freezes 
$P$. 

\begin{definition}
For any regular cardinal $\kappa\geq\omega_2$ 
fix 
\[
\{S^i_\alpha:\alpha<\kappa,i<2\}
\] 
a partition of 
$E_\kappa^\omega$ (the set of points in $\kappa$ of countable cofinality) in 
pairwise disjoint stationary sets. Fix 
\[
\{A_\alpha:\alpha<\omega_1\}
\] partition of $\omega_1$ in $\omega_1$-many pairwise disjoint
stationary sets such that $\min(A_\alpha)>\alpha$
and such that there is a club subset of $\omega_1$ contained in
\[
\bigcup\{A_\alpha:\alpha<\omega_1\}.
\]
Given $P$ a stationary set preserving poset, we fix
in $V$ a surjection $f$ of the least regular $\kappa> |P|$ with $P$. 
Let $\dot{g}_{P}:\kappa\to 2$ be the $P$-name for a function which codes
$\dot{G}_{P}$ using $f$, i.e. for all $\alpha<|P|$
\[
p\Vdash_{P}\dot{g}_{P}(\alpha)=1\text{ iff }f(\alpha)\in\dot{G}_{P}.
\]
Now let $\QQ_{P}$ be the complete boolean algebra 
$\BBB(P*\dot{R}_{P})$ where $\dot{R}_{P}$ 
is defined as follows in 
$V^{P}$:

Let $G$ be $V$-generic for $P$.
Let $g=\val_G(\dot{g}_{P})$.
$R=\val_G(\dot{R}_P)$ in $V[G]$ is the poset
given by pairs $(c_p,f_p)$ such that for some countable ordinal
$\alpha_p$
\begin{itemize}
\item
$f_p:\alpha_p+1\to \kappa$,
\item
$c_p\subseteq\alpha_p+1$ is closed, 
\item 
for all $\xi\in c_p$
\[
\xi\in A_\beta\text{ and }g\circ f_p(\beta)=i\text{ if and only if }\sup(f_p[\xi])\in S^i_{f_p(\beta)}.
\]
\end{itemize}
The order on $R$ is given by $p\leq q$ if $f_p\supseteq f_q$ and $c_p$ end extends $c_q$.
Let
\begin{itemize}
\item
$\dot{f}_{\QQ_{P}}:\omega_1\to\kappa$ be the $P*\dot{R}_{P}$-name for the function given by
\[
\bigcup\{f_p:p\in\dot{G}_{P*\dot{R}_{P}}\},
\]
\item
$\dot{C}_{\QQ_{P}}\subset\omega_1$ be the $P*\dot{R}_{P}$-name for the club given by
\[
\bigcup\{c_p:p\in\dot{G}_{P*\dot{R}_{P}}\},
\]
\item
$\dot{g}_{\QQ_{P}}\subset\omega_1$ be the $P*\dot{R}_{P}$-name for the function 
$\dot{g}_{P}$ coding a 
$V$-generic filter for $P$ using $f$.
\end{itemize}
\end{definition}

We are ready to show that all stationary set preserving posets are freezeable.
\begin{theorem}[Freezing lemma]\label{thm:freezeden}
Assume $P$ is stationary set preserving. 
Then $P$ forces that $\dot{R}_{P}$ is stationary set preserving and
 $\QQ_{P}=\BBB(P*\dot{R}_{P})$ freezes
$P$ as witnessed by the map $k:\BBB(P)\to \QQ_P$ which maps 
$p\in P$ to $\langle p, 1_{\dot{R}_P}\rangle$.
\end{theorem}

\begin{proof}
It is rather standard to show that $\dot{R}_{P}$ is forced by $P$ to be
stationary set preserving.
We briefly give the argument for $R=\val_G(\dot{R}_{P})$
working in $V[G]$ where $G$ is $V$-generic for $P$.
First of all
we observe that $\{S^i_\alpha:\alpha<\kappa,i<2\}$ is still in $V[G]$ a partition of 
$(E^\omega_\kappa)^V$ in pairwise disjoint stationary sets, since $P$ is $<\kappa$-CC and that 
$\{A_\alpha:\alpha<\omega_1\}$ is still a maximal antichain
on $P(\omega_1)/\NS_{\omega_1}$ in $V[G]$ since $P\in \SSP$ and 
$\{A_\alpha:\alpha<\omega_1\}$ contains a club subset of $\omega_1$.

\begin{claim}
$R$ is stationary set preserving.
\end{claim}
\begin{claimproof}
Let $\dot{E}$ be an $R$-name for a club subset of $\omega_1$ and $S$ be a stationary subset of
$\omega_1$. Then we can find $\alpha$ such that $S\cap A_\alpha$ is stationary. 
Pick $p\in R$ such that $\alpha\in \dom(f_p)$, Let $\beta=f_p(\alpha)$ and $i=g(\beta)$ where 
$g:\kappa\to 2$ is the function coding $G$ by means of $f$. 
By standard arguments find $M\prec H_\theta^{V[G]}$ countable such that 
\begin{itemize}
\item
$p\in M$, 
\item
$M\cap\omega_1\in S\cap A_\alpha$,
\item
$\sup (M\cap\kappa)\in S^i_\beta$.
\end{itemize}
Working inside $M$ build a decreasing chain of conditions $p_n\in R\cap M$ which 
seals all dense sets of $R$ in $M$ and such that $p_0=p$.
By density we get that
\[
f_\omega=\bigcup_{n<\omega}f_{p_n}:\xi=M\cap\omega_1\to M\cap\kappa
\]
is surjective and that $\xi$ is a limit point of
\[
c_\omega=\bigcup_{n<\omega}c_{p_n}
\]
which is a club subset of $\xi$.
Set 
\[
q=(f_\omega\cup\{\langle \xi,0\rangle\},c_{\omega}\cup\{\xi\}).
\]
Now observe that $q\in R$ since 
$\xi\in A_\alpha$ and $\sup(f_q[\xi])\in S^{g(f_p(\alpha))}_{f_p(\alpha)}$ and 
$c_q$ is a closed subset of $\xi+1$.
Now by density $q$ forces that $\xi\in\dot{E}\cap S$ and we are done.
\end{claimproof}

We now argue that $\QQ_{P}$ freezes $P$.
We shall do this by means of Lemma~\ref{lem:eqfr}(\ref{lem:eqfr2}).

Assume that $\RR\leq \QQ_P$, let $H$ be $V$-generic for $\RR$ and pick $G_0,G_1\in V[H]$
distinct correct $V$-generic fiters for $\QQ_P$. It is enough to show that
\[
\bar{G}_0=\bar{G}_1
\]
where
\[
\bar{G}_j=\{p\in P: \exists \dot{q}\in V^P \text{ such that }\langle p,\dot{q}\rangle\in G_j\}
\]

Let $g_j:\kappa\to 2$ be the evaluation by $G_j$ of the function $\dot{g}_{P}$
which is used to code 
$\bar{G}_j$ as a subset of $\kappa$ by letting $g_j(\alpha)=1$ iff $f^{-1}(\alpha)\in \bar{G}_j$.
Let 
\[
h_j=\bigcup\{f_p:p\in G_j\},
\]
\[
C_j=\bigcup\{c_p:p\in G_j\},
\]

In particular we get that $C_0$ and $C_1$ are club subsets of
$\omega_1$ in $V[H]$, $h_0,h_1$ are bijections of $\omega_1$ with $\kappa$.

Now observe that $\kappa$ has size and cofinality $\omega_1$ in $V[G_j]$ and thus
(since $V[H]$ is a generic extension of $V[G_j]$ with the same $\omega_1$)
$\kappa$ has size and cofinality $\omega_1$ in $V[H]$.
Observe also that in $V[H]$ 
\begin{quote}
$S$ is a stationary subset of $\kappa$ iff 
$S\cap\{\sup h[\xi]:\xi<\omega_1\}$ is non empty for any bijection
$h:\omega_1\to\kappa$.
\end{quote}
 
Now the very definition of the $h_j$ gives that for all $\alpha\in C_j$:
\begin{quote}
$h_j(\alpha)=\eta$  and $g_j(\eta)=i$ if and only if 
$\sup h_j[\xi]\in S^i_{\eta}$
for all
$\xi\in A_\alpha\cap C_j$.
\end{quote}
Now the set
\[
E=\{\xi\in C_0\cap C_1: h_0[\xi]=h_1[\xi]\}
\]
is a club subset of $\omega_1$, and the above observations show that
\[
S^{g_j(\eta)}_{\eta}\supseteq \{\sup h_j[\xi]: \xi\in E\cap A_\alpha\}\neq \emptyset
\] 
for both $j$.
In particular $g_0(\eta)=g_1(\eta)$ for all $\eta<\kappa$,
else $S^0_\eta\cap S^1_\eta$ is non-empty for some $\eta$ contradicting the very definition of the family 
of sets $S^i_\eta$.
Thus $\bar{G}_0=\bar{G}_1$.

The proof of Theorem~\ref{thm:freezeden} is completed.
\end{proof}

\subsubsection{Freezeability implies total rigidity}

We are now in the position to
use the previous theorem and some elementary observations on iterations to conclude that 
the family of totally rigid partial orders is also dense in $\UU^{\SSP}$ and $\UU^{\SP,\SP}$
provided there are enough large cardinals.

\begin{theorem}\label{thm:totrigden1}
Assume $\delta$ is a limit of uncountable cofinality of 
cardinals $\alpha$ which are $\alpha+\omega+1$-supercompact and 
is such that $V_\delta$ models $\ZFC$. 
Then for every $\BB\in\SSP\cap V_\delta$ there is
$k_{\BB}:\BB\to \RR_{\BB}$ such that:
\begin{itemize}
\item
$\RR_{\BB}\in\SSP$ is a totally rigid complete boolean algebra of size $\delta$,
\item
$\llbracket \RR_{\BB}/_{k_{\BB}[\dot{G}_{\BB}]}\text{ is semiproper}\rrbracket_{\BB}=1_{\BB}$.
\end{itemize}
\end{theorem}

We refer the reader to Subsection~\ref{subsec:itfor} for the relevant definitions and results on iterations
and to~\cite{VIAAUDSTE13} for a detailed account.
\begin{proof}
First notice that $V_\delta$ models that there are class many $\alpha+\omega+1$-supercompact 
cardinals $\alpha$.

Now let for any $\BB\in\SSP\cap V_\delta$, 
$i_{\BB}:\BB\to \QQ_{\BB}$  be the regular embedding 
defined in the previous
subsection to freeze $\BB$ and $\dot{R}_{\BB}$ be the canonical name for the quotient forcing.
Observe that $\dot{R}_{\BB}$ 
can be chosen to be a $\BB$-name for a stationary set preserving poset in 
$V^{\BB}_{\rank(\BB)+\omega}$.
Moreover $\QQ_{\BB}$ collapses $\BB$ to have 
size $\omega_1$.

Next observe the following:
Assume $\BB$ has inaccessible size,
forces $\MM^{++}$ up to $\beth(\omega)$, and collapses its
size to become $\omega_2$. Then
\[
\BB\Vdash\dot{R}_{\BB*\dot{Q}}\text{ is semiproper}
\]
whenever $\dot{Q}$ is a $\BB$-name for a stationary set preserving poset
of size less than $\beth(\omega)$ in $V^{\BB}$.
Let $k_{\BB}:\BB\to \CC_{\BB}\in V_\delta$ be a regular embedding such that
\begin{itemize}
\item
$\llbracket \CC_{\BB}/_{k_{\BB}[\dot{G}_{\BB}]}\text{ is semiproper}\rrbracket_{\BB}=1_{\BB}$,
\item
$\CC_{\BB}$ forces $\MM^{++}$ up to $\beth(\omega)$ while collapsing $\alpha$ to become
$\omega_2$ for some $\alpha\in (|\BB|,\delta)$ which is $\alpha+\omega+1$-supercompact.
\end{itemize}
Such a $k_{\BB}$ can be found in $V_\delta$ 
applying Theorem~\ref{thm:FMSvar} in $V_\delta$.

Let
\[
\mathcal{F}=\{i_{\alpha,\beta}:\BB_\alpha\to\BB_\beta:\alpha<\beta\leq\delta\}
\]
be defined as follows for all $\alpha$ limit and $n\in\omega$:
\begin{itemize}
\item
$\BB_0=\BB$,
\item
$i_{\alpha+2n,\alpha+2n+1}=k_{\BB_{\alpha+2n}}$,
\item
$i_{\alpha+2n+1,\alpha+2n+2}=i_{\BB_{\alpha+2n+1}}$.
\end{itemize}

The claim below follows in a rather straightforward manner from the above observation:
\begin{claim}
For all $\alpha<\delta$:
\begin{itemize}
\item
$\BB_{\alpha+2}$ collapses $\BB_\alpha$ to have size $\omega_1$,
\item
$\llbracket\BB_{\alpha+2}/_{i_{\alpha,\alpha+2}[\dot{G}_{\BB_\alpha}]}\text{ is semiproper }
\rrbracket_{\BB_\alpha}=1_{\BB_\alpha}$,
\item
$\BB_{\alpha+2}$ has size less than $\delta$.
\end{itemize}
\end{claim}
Then (by the standard arguments which are used to prove the consistency of 
$\MM^{++}$ see~\cite[Theorem 37.9]{JEC03} or~\cite[Theorem 8.5]{VIAAUDSTE13}) 
this iteration will be such that:
\[
\llbracket C(\FFF)/i_{0}[\dot{G}_{\BB_0}]\text{ is semiproper}\rrbracket_{\BB_0}=1_{\BB_0}.
\]
Thus $C(\FFF)\in\SSP$ and $C(\FFF)\leq_{\SSP}\BB_0$.

We can now prove the following:
\begin{claim} 
$C(\FFF)$ is totally rigid.
\end{claim}
\begin{claimproof}
$i_{\alpha,\alpha+2}$ freezes $\BB_{\alpha}$ for any $\alpha<\delta$.
Thus we get 
that $C(\FFF)$ is a direct limit of freezed posets (since $\delta$ has uncountable cofinality).
In particular Lemma~\ref{lem:freezelem*} grants that $C(\FFF)$ is totally 
rigid.
\end{claimproof}
Theorem~\ref{thm:totrigden1} is proved.
\end{proof}

We can now prove also Theorem~\ref{thm:distrSSP*}:
\begin{proof}
Let $\delta$ be a supercompact cardinal.
By Theorem~\ref{thm:totrigden1} any $\BB\in \SSP\cap V_\delta$ is absorbed by a
totally rigid poset $\QQ_{\BB}$ whose size is at most $\alpha$ and is such that
the quotient forcing is semiproper,
where
$\alpha<\delta$ is the first $\alpha+\omega+2$-supercompact cardinal larger than $\BB$.
We fix $f:\delta\to V_\delta$ a Laver function and we define an iteration system 
\[
\mathcal{F}=\{i_{\alpha,\beta}:\BB_\alpha\to\BB_\beta:\alpha<\beta\leq\delta\}
\]
as follows for all $\alpha$ limit and $n\in\omega$:
\begin{itemize}
\item
$\BB_0=\BB$,
\item
$i_{\alpha,\alpha+1}:\BB_\alpha\to \QQ_{\BB*f(\alpha)}$ is an $\SP$-correct embedding, 
if $f(\alpha)$ is a $\BB_\alpha$-name for 
a semiproper poset,
\item
$i_{\alpha,\alpha+1}:\BB_\alpha\to \BBB(\BB*Coll(\omega_1,\BB))$ is any $\SP$-correct embedding
otherwise.
\end{itemize}
 By Theorem~\ref{thm:totrigden1} this definition makes sense for all stages.
 Now we can mimick the consistency proof of $\MM^{++}$ as in~\cite[Theorem 37.9]{JEC03}
 or~\cite[Theorem 8.5]{VIAAUDSTE13}
 to argue that $\BB_\delta$ forces $\MM^{++}$.
 By the same argument of the previous proof we can also grant that $\BB_\delta$ is totally rigid
 and that 
 \[
 \llbracket\BB_{\delta}/_{i_{0,\delta}[\dot{G}_{\BB}]}\text{ is semiproper }
\rrbracket_{\BB}=1_{\BB}.
 \]
 Theorem~\ref{thm:distrSSP*} is completely proved.
\end{proof}

%%%%%%%%%%%%%%%%%%%%%%%%%%%%%%%%%%%%%%%%%%%%%%%%%%%

\subsubsection{Forcing properties of $\UU^{\SSP}_\delta$}

In this subsection we assume $\delta$ is large, meaning that it is an inaccessible
limit of $<\delta$-supercompact cardinals.
With these assumptions at hand we have that
the set of totally rigid partial orders $Q$ in $\UU_\delta$ which force
$\MM^{++}$ up to $\delta$ is dense in $\UU_\delta$ by Theorem~\ref{thm:distrSSP*}.
We shall limit ourselves to analyze
$\UU_\delta$ restricted to this set which we denote by $\TR$.
\begin{fact}\label{fac:TR}
The following holds for $\TR$.
\begin{enumerate}
\item\label{fac:TR0}
For any $\BB\in\UU_\delta$ there is $\CC\leq_{\SP}\BB$ in $\TR$ by Theorem~\ref{thm:distrSSP*}.
\item\label{fac:TR1}
For $\BB,\QQ\in\TR$, $\BB\leq_{\SSP}\QQ$ ($\BB$ and $\QQ$ are $\leq_{\SSP}$-incompatible)
iff 
$\BB\leq_{\SP}\QQ$ ($\BB$ and $\QQ$ are $\leq_{\SP}$-incompatible).
\item\label{fac:TR2}
Let $\mathcal{G}=\{i_{\alpha,\beta}:\BB_\alpha\to\BB_\beta\}\subseteq\TR$
be an iteration system such that each $i_{\alpha,\beta}$ is $\SSP$-correct.
Then $\RCS(\mathcal{G})\in\SSP$ is a
lower bound for each $\BB_\alpha$ under $\leq^*_{\SSP}$.
\item\label{fac:TR3}
Assume $A\subset \TR$ is an antichain of size less than $\delta$. Then $\bigvee A\in \TR$.
\item\label{fac:TR4}
For any totally rigid $\CC\in\UU_\delta$ the map
$k_{\CC}:\CC\to\UU_\delta\restriction \CC$ which maps $c\in \CC$ to
$\CC\restriction c$ is a correct regular embedding.
\end{enumerate}
\end{fact}
\begin{proof}
Left to the reader.
\end{proof}

The next key observation is the following:
\begin{lemma}\label{lem:densityreflected}
Let $D\subset\TR$ be a dense open subset of $\TR\cap \UU_\delta$.
Then for every $\BB\in \UU_\delta$ there is $\CC\in \TR$, 
an \emph{injective $\SP$-correct} complete homomorphism $i:\BB\to\CC$, and 
$A\subset \CC$ maximal antichain of $\CC$ such that
$k_{\CC}[A]\subset D$.
\end{lemma}
\begin{proof}
Given $\BB\in \TR$ find 
$\CC_0\leq\BB$ in $D$. 
Let $i_0:\BB\to \CC_0$ be a complete and $\SP$-correct homomorphism of $\BB$ into $\CC_0$
given by Theorem~\ref{thm:distrSSP*}.
Let $b_0\in \BB$ be the complement of $\bigvee_{\BB}\ker(i_0)$ so that 
$i_0\restriction b_0:\BB\restriction b_0\to \CC_0$ is an injective $\SP$-correct homomorphism.
Proceed in this way to define $\CC_l$ and $b_l$ such that:
\begin{itemize}
\item
$i_l\restriction b_l:\BB\restriction b_l\to \CC_l$ is an injective $\SP$-correct homomorphism,
\item
$\CC_l\in D\subset \TR$,
\item
$b_l\wedge_{\BB}b_i=0_{\BB}$ for all $i<l$. 
\end{itemize}

This procedure must terminate in $\eta<|\BB|^+$ steps producing
a maximal antichain $\{b_l:l<\eta\}$ of $\BB$ and
injective $\SP$-correct homomorphisms
$i_l:\BB\restriction b_l\to \CC_l$ such that $\CC_l\in D\subset\TR$ refines $\BB$ in
the $\leq_{\SP}$ order.
Then we get that 
\begin{itemize}
\item
$\CC=\bigvee_{l<\eta}\CC_l\in \TR$ by Fact~\ref{fac:TR}.\ref{fac:TR3}.
\item $i$ is an injective $\SP$-correct homomorphism where
\begin{align*}
i=&\bigvee_{k<\eta} i_k:\BB\to\CC_k\\
& c\mapsto\langle i_k(c\wedge_{\BB}b_k):k<\eta\rangle
\end{align*}
is such that 
\[
\llbracket \BB/i[\dot{G}_{\BB}]\in\SP\rrbracket_{\BB}=
1_{\BB},
\]
since $i$ is the lottery sum of the injective $\SP$-correct homomorphisms $i_l$.
\item
$\CC\restriction i(b_k)\in D$ for all $k<\eta$.
\end{itemize}
In particular we get that 
$A=i[\{b_k:k<\eta]$ is a maximal antichain of $\CC\in\TR$ such that $\CC\restriction c\in D$ for 
all $c\in A$ as was to be shown.
\end{proof}

\begin{lemma}\label{lem:presaturationUdelta}
$\UU_\delta$ is $(<\delta,<\delta)$-presaturated.
\end{lemma}
\begin{proof}
Let $\dot{f}$ be a $\UU_\delta$-name for an increasing
function from $\eta$ into $\delta$ for some $\eta<\delta$.
Given $\BB\in \UU_\delta$
let $A_i\subset\TR$ be the dense set of totally rigid partial orders in 
$\UU_\delta\restriction \BB$ which decide that $\dot{f}(i)=\alpha$ for some $\alpha<\delta$.
Then using the previous lemma we can build inside $V_\delta$ an RCS iteration of
complete boolean algebras 
\[
\{i_{\alpha,\beta}:\CC_\alpha\to \CC_\beta:\alpha\leq\beta\leq\eta\}\in V_\delta
\]
such that
for all $i<\eta$ $\CC_{i+1}\in \TR$ and 
there is $B_i$ maximal antichain of $\CC_{i+1}\leq^*_{\SP}\CC_i$ such
that $k_{\CC_{i+1}}[B_i]\subset A_i$. 
Then $\CC_\eta\in \SSP$ forces that $\dot{f}$ has values bounded by
\[
\sup\{\alpha:\exists c\in \CC_i\text{ such that $\CC_i\restriction c$ forces that
$\dot{f}(i)=\alpha$}\}<\delta.
\]
\end{proof}

\begin{lemma}\label{lem:copyingUdelta}
Assume $\dot{f}\in V^{\UU_\delta}$ is a name for a function in $\Ord^\alpha$ for some 
$\alpha<\delta$. Then there is a dense set of $\CC\in \TR$ 
with an $\dot{f}_{\CC}\in V^{\CC}$ such 
that 
\[
\llbracket\hat{k}_{\CC}(\dot{f}_{\CC})=\dot{f}\rrbracket_{\BBB(\UU_\delta)}\geq\CC.
\]
\end{lemma}

\begin{proof}
Given $\dot{f}$ as above, let for all $\xi<\alpha$
\[
D_\xi=\{\CC\in\TR: \exists\beta\,\CC\Vdash_{\UU_\delta}\dot{f}(\xi)=\beta\}.
\]
Let $\BB\in \TR$ be arbitrary.
By the previous lemma we can find $\CC\in \TR$ below $\BB$ such that
for all $\xi<\alpha$ there is a maximal antichain $A_\xi\subset \CC$ such that
$k_{\CC}[A_\xi]\subset D_\xi$.
Now let $\dot{f}_{\CC}$ be the $\CC$-name
\[
\{\langle(\xi,\eta),c): c\in A_\xi \text{ and }\CC\restriction c\Vdash_{\UU_\delta}\dot{f}(\xi)=\eta\}.
\]
It is immediate to check that for all $\xi<\alpha$ and $c\in A_\xi$
\[
c\Vdash_{\CC}\dot{f}_{\CC}(\xi)=\eta
\text{ iff } \CC\restriction c\Vdash_{\UU_\delta}\dot{f}(\xi)=\eta.
\]
This gives that $\llbracket\hat{k}_{\CC}(\dot{f}_{\CC})=\dot{f}\rrbracket_{\BBB(\UU_\delta)}\geq\CC$.
The Lemma is proved.
\end{proof}

In particular we also get the following:
\begin{lemma}\label{lem:copyingUdelta1} 
Assume $\tau\in V^{\UU_\delta}$ is a $\UU_\delta$-name for an element
of $(H_\delta)^{V^{\UU_\delta}}$. Then there is a dense open set of $\CC\in\TR$ and
names $\sigma_{\CC}\in V^{\CC}\cap V_\delta$ such that
\[
\llbracket\tau=k_{\CC}(\sigma_{\CC})\rrbracket_{\UU_\delta}\geq_{\BBB(\UU_\delta)}\CC.
\]
\end{lemma}
\begin{proof}
Left to the reader: observe that any such $\UU_\delta$-name $\tau$ can be coded by a  
$\UU_\delta$-name for a function from some $\alpha<\delta$ into $\delta$.
\end{proof}

\begin{lemma}\label{lem:chargenericityUdelta}
Assume $G$ is $V$-generic for $\UU_\delta$.
Then for all $\BB\in\UU_\delta$,
$\BB\in G$ iff there is $H\in V[G]$ correct $V$-generic filter for $\BB$.
\end{lemma}
\begin{proof}
We have to prove:
\begin{enumerate}
\item\label{impl1}
Assume $\BB\in G$. Then there is is $H\in V[G]$ correct $V$-generic filter for $\BB$.
\item\label{impl2}
Assume  there is some $H\in V[G]$ correct $V$-generic filter for $\BB$. Then
$\BB\in G$.
\end{enumerate}
We proceed as follows:
\begin{description}
\item[Proof of \ref{impl1}]
Assume $\BB\in G$.
Then there is a totally rigid $\RR\leq_{\SSP}\BB$ freezing $\BB$ in $G$ as witnessed
by $i:\BB\to\RR$, since the set of such $\RR$ is dense in $\UU_\delta$ below $\BB$.
In particular $H=\{b\in\BB:\RR\restriction i(b)\in G\}$ is a
$V$-generic filter for $\BB$. We must show that it is also a correct $V$-generic filter, i.e.
that for all $\BB$-names $\dot{S}$ for stationary subsets of $\omega_1$
$\dot{S}_H$ is stationary in $V[G]$. So assume this is not the case.
Then there is $\dot{S}$, $\BB$-name for a stationary subset of $\omega_1$, and 
$\dot{C}$, $\UU_\delta$-name for a club subset of $\omega_1$, such that 
\[
V[G]\models\dot{C}_G\cap\dot{S}_H=\emptyset.
\]
We can thus find $\DD\leq_{\SSP}\RR$ in $G$ forcing the above statement, i.e. more precisely:
Let for all $\QQ\leq_{\SSP}\RR$
$i_{\QQ}:\BB\to\QQ$ be the unique correct homomorphism which factors through
$i:\BB\to\RR$, recall also that $k_{\QQ}:\QQ\to\UU_\delta\restriction\QQ$ is the complete homomorphism
defined by $d\mapsto\QQ\restriction d$ for all $\QQ\in\TR$. Then
\[
\llbracket\dot{C}\cap\hat{k}_{\RR}\circ\hat{i}(\dot{S})=
\emptyset\rrbracket_{\BBB(\UU_\delta)}\geq\DD\in G.
\]
We leave to the reader to check that 
\[
\llbracket\hat{k}_{\RR}\circ\hat{i}(\dot{S})=\hat{k}_{\QQ}\circ \hat{i}_{\QQ}(S)
\rrbracket_{\BBB(\UU_\delta)}\geq\QQ
\]
holds for all $\QQ\leq_{\SSP}\RR$ .

Applying Lemma~\ref{lem:copyingUdelta1} to the $\UU_\delta$-name $\dot{C}$
for a subset of $\omega_1$, we can find $\CC\leq_{\SSP}\DD$ in $\UU_\delta$
such that $\CC\in\TR$ and for some $\dot{E}$ $\CC$-name for a subset of $\omega_1$
\[
\llbracket \hat{k}_{\CC}(\dot{E})=\dot{C}\rrbracket_{\UU_\delta}\geq\CC.
\]
Since the formula $\phi(x,y,\omega_1)$ stating that 
$x$ is a club subset of $\omega_1$ disjoint from $y$ is a $\Sigma_0$-statement in the parameter
$\omega_1$ we get that
\begin{align*}
k_{\CC}(\llbracket\dot{E}\cap \hat{i}_{\CC}(\dot{S})=\emptyset\rrbracket_{\CC})=
\CC\wedge \llbracket\dot{C}\cap\hat{k}_{\CC}\circ \hat{i}_{\CC}(\dot{S})=
\emptyset\rrbracket_{\BBB(\UU_\delta)}\geq_{\SSP}&\\
\geq_{\SSP}\CC\wedge \llbracket\dot{C}\cap
\hat{k}_{\RR}\circ\hat{i}(\dot{S})=
\emptyset\rrbracket_{\BBB(\UU_\delta)}\geq_{\SSP}\CC\in G.
\end{align*}
In particular we conclude that:
\[
\llbracket\dot{E}\cap \hat{i}_{\CC}(\dot{S})=\emptyset\rrbracket_{\CC}=1_{\CC}.
\]
This contradicts the very definition of $i_{\CC}:\BB\to\CC$ being an $\SSP$-correct homomorphism,
concluding the proof of~\ref{impl1}.

\item[Proof of \ref{impl2}]
Assume there is $H\in V[G]$ correct $V$-generic filter for $\BB$.
Towards a contradiction assume that $\BB\not\in G$.
Then there is some $\RR\in G$ such that $\BB$ and $\RR$ are incompatible in
$\UU_\delta$ since
\[
\{\RR\in\UU_\delta: \RR\text{ is orthogonal to $\BB$ in $\UU_\delta$ or }\RR\leq_{\SSP}\BB\}
\]
is open dense in $\UU_\delta$ and $G$ must meet it in some $\RR$ satisfying the first clause of the
above disjunction.
Let $\dot{H}$ be a $\UU_\delta$-name for $H$ and $\CC\in G\cap\TR$ be an element refining $\RR$
and forcing 
in $\UU_\delta$ that
$\dot{H}$ is a correct $V$-generic filter for $\BB$.
In particular:
\begin{itemize}
\item 
for all $\dot{S}$ $\BB$-name for a subset of 
$\omega_1$ there is
$\dot{\underline{S}}$  in $V^{\UU_\delta}$ such that
\[
\llbracket\dot{\underline{S}}\text{ is the interpretation of the $\BB$-name }\dot{S}
\text{ by the $V$-generic filter }
\dot{H}\rrbracket_{\BBB(\UU_\delta)}\geq\CC
\]
and such that
\[
\llbracket\dot{S}\text{ is stationary }\rrbracket_{\BB}=1_{\BB}
\]
if and only if for all $\dot{C}$ $\UU_\delta$-name for a club subset of $\omega_1$
\[
\llbracket\dot{\underline{S}}\cap\dot{C}\neq\emptyset\rrbracket_{\BBB(\UU_\delta)}\geq\CC.
\]
\item
$\CC$ forces in $\UU_\delta$ that $\dot{H}$ is a ultrafilter on $\BB$,
\item
$\CC$ forces in $\UU_\delta$ that $\dot{H}\cap D\neq\emptyset$ for all $D\in V$
dense subset of $\BB$.
\end{itemize}

Now the family
\[
\mathcal{A}=\{\dot{\underline{S}}:
\llbracket\dot{S}\subset\omega_1\rrbracket_{\BB}=
1_{\BB}\}
\cup\{\dot{H}\}
\]
has size less than $\delta$, and is a family of names for sets of size less than $\delta$.
Thus, as in the proof of Lemma~\ref{lem:presaturationUdelta},
we can apply iteratively Lemma~\ref{lem:copyingUdelta1} and Lemma~\ref{lem:densityreflected}
to find that the set of 
\[
\{\DD\leq_{\SSP}\CC: \forall\, \tau\in\mathcal{A}\,\exists\, \sigma(\tau)\in V^{\DD}\,\,
\llbracket \hat{k}_{\DD}(\sigma(\tau))=\tau\rrbracket_{\BBB(\UU_\delta)}\geq\DD\}
\]
is dense below $\CC$ in $\UU_\delta$. In particular we can find such a $\DD\in G$.
Notice that $\DD$ is orthogonal to $\BB$ in $\UU_\delta$ since $\CC$ is.

Now let $\dot{K}\in V^{\DD}$ be such that
\[
\llbracket \hat{k}_{\DD}(\dot{K})=\dot{H}\rrbracket_{\BBB(\UU_\delta)}\geq\DD
\] 
and
let:
\begin{itemize}
\item 
$\phi_0(x,y,z)$ be the $\Sigma_0$-formula asserting $x$ is a club subset of 
$z$ and $x\cap y\neq\emptyset$,
\item
$\phi_1(x,y)$ be the $\Sigma_0$-formula asserting $x$ is a ultrafilter on the boolean
algebra $y$
\item
$\phi_2(x,y,z)$ be the $\Sigma_0$-formula asserting $x$ meets $y$, and $y$ is a dense
subset of the boolean algebra $z$.
\end{itemize}
Notice that
\[
k_{\DD}(\llbracket \phi_1(\dot{K},\BB)\rrbracket_{\dot{\DD}}\geq\DD,
\]
\[
k_{\DD}(\llbracket \phi_2(\dot{K},D,\BB)\rrbracket_{\dot{\DD}}\geq\DD
\]
\[
k_{\DD}(\llbracket \phi_0(\dot{C},\sigma(\dot{\underline{S}}),\omega_1)\rrbracket_{\dot{\DD}}\geq\DD
\]
for all $\DD$-names $\dot{C}$ for a club subset of $\omega_1$ 
and all $\BB$-names $\dot{S}$ for a
stationary subset of $\omega_1$.

In particular we get that 
\[
\llbracket\dot{K} \text{ is a correct $V$-generic filter for }\BB\rrbracket_{\DD}=1_{\DD},
\]
applying Fact~\ref{fac:charSSP}.
We reached a contradiction since the map
\[
l:b\mapsto\llbracket b\in\dot{K}\rrbracket_{\DD}
\]
defines an $\SSP$-correct homomorphism of $\BB$ into $\DD$ witnessing that
$\DD\leq_{\SSP}\BB$, contrary to the already established fact that
$\DD$ is orthogonal to $\BB$ in $\UU_\delta$.
\end{description}
\end{proof}

\subsubsection{Proof of Theorem~\ref{thm:univ1}}
In this section we prove the following:

\begin{theorem}
Assume $\delta$ is inaccessible and $V_\delta$ models that there 
are class many supercompact cardinals $\gamma$.
Then $\UU^{\SSP}_\delta\in\SSP$ and collapses $\delta$ to become $\aleph_2$.
If $\delta$ is supercompact, then $\UU^{\SSP}_\delta$ forces
$\MM^{++}$.
\end{theorem}

\begin{proof}

We prove all items as follows:
\begin{description}

\item[ \textbf 
{$\UU_\delta$ preserves the regularity of $\delta$}]

This follows immediately from the $(<\delta,<\delta)$-presaturation of $\UU_\delta$.

\item[ \textbf{$\UU_\delta$ is stationary set preserving}]

Fix $S\in V$ stationary subset of $\omega_1$ and $\BB\in \UU_\delta$.
Let $\dot{C}$ be a $\UU_\delta$-name for a club subset of $\omega_1$.
We must find $\CC\leq\BB$ which forces in $\UU_\delta$ that $S\cap\dot{C}$ is non empty.

Let $\dot{f}$ be a $\UU_\delta$-name for the increasing enumeration of $\dot{C}$.

By Lemma~\ref{lem:copyingUdelta} we can find $\CC\leq\BB$ in $\TR$ and $\dot{f}_{\CC}$ such that
$\CC$ forces in $\UU_\delta$  that $\hat{k}_{\CC}(\dot{f}_{\CC})=\dot{f}$.

Now observe that the statement 

\emph{``$\dot{f}$ is a continuos strictly increasing map from
$\omega_1$ into $\omega_1$''} 

is a $\Sigma_0$ statement in the parameters
 $\omega_1$, $\dot{f}$ and is forced by $\CC$ in $\UU_\delta$. 
 Since $\hat{k}_{\CC}$ is $\Delta_1$-elementary and $\hat{k}_{\CC}(\dot{f}_{\CC})=\dot{f}$, we 
 get that 
 \[
 \llbracket\text{$\dot{f}_{\CC}$ is a continuos strictly increasing map from
$\omega_1$ into $\omega_1$}\rrbracket_{\CC}=1_{\CC}.
\]
This gives that 
\[
\llbracket\text{$\rng(\dot{f}_{\CC})$ is a club}\rrbracket_{\CC}=1_{\CC}.
\]
Since $\CC$ is $\SSP$ we get that 
 \[
\llbracket S\cap\rng(\dot{f}_{\CC})\neq\emptyset\rrbracket_{\CC}=1_{\CC}.
\]
Now once again we observe that the above statement is $\Sigma_0$ in the parameters
$S,\dot{f}_{\CC}$ and thus, since $\hat{k}_{\CC}$ is $\Delta_1$-elementary, we can conclude
 that
  \[
\llbracket S\cap\rng(\dot{f})\neq\emptyset\rrbracket_{\BBB(\UU_\delta\restriction\CC)}=
1_{\BBB(\UU_\delta\restriction\CC)},
\]
which gives the thesis.

\item[\textbf{$\UU_\delta$ forces $\delta$ to become $\omega_2$}]

This is immediate since
there are densely many posets $\BB\in\TR\cap\UU_\delta$ which collapse
their size to become
at most $\omega_2$ and embed completely into $\UU_\delta\restriction\BB$.
Thus for unboundedly many $\xi<\delta$ we get that $\UU_\delta$ forces $\xi$ to be an
ordinal less or equal $\omega_2$. All in all we get that $\UU_\delta$ forces $\delta$
to be less or equal than $\omega_2$.

\item[\textbf{$\UU_\delta$ forces $\MM^{++}$ if $\delta$ is supercompact}]
Let $\dot{\RR}\in V^{\UU_\delta}$ be a name for a stationary set preserving poset.
Given $\BB$ in $\UU_\delta$ find 
$j:V_\gamma\to V_\lambda$ in $V$ such that $\crit(j)=\alpha$, $\BB\in V_\alpha$,
$j(\crit(j))=\delta$, and $\dot{\RR}\in j[V_\gamma]$.

Let $\dot{\QQ}\in V^{\UU_\alpha}$ be such that $j(\dot{\QQ})=\dot{\RR}$. 

By elementarity of $j$ we get that $\UU_\alpha\in \SSP$.
Then $\QQ=(\UU_\alpha\restriction\BB)*\dot{\QQ}\leq_{\SSP}\UU_\alpha,\BB$ 
forces in $\UU_\delta$ that $j$ lifts
to
\[
\bar{j}:V_\gamma[\dot{G}_{\UU_\alpha}]\to V_\lambda[\dot{G}_{\UU_\delta}].
\]
Moreover let $G$ be $V$-generic for $\UU_\delta$ 
with $\QQ\in G$ and $G_0=G\cap V_\alpha$.
Then in  $V_\lambda[G]$ there is a correct $V[G_0]$-generic filter $K$ for 
$\QQ/G_0=\val_{G_0}(\dot{\QQ})$.

Finally we get that in $V[G]$,
$\bar{j}[K]$ is a correct $\bar{j}[V_\gamma[G_0]]$-generic filter for 
$\RR=\val_G(\dot{\RR})=\bar{j}\circ \val_{G_0}(\dot{\QQ})$ 
showing that $T_{\RR}$ is stationary in $V[G]$.
Since this holds for all $V$ generic filter $G$ to which $\QQ\leq\BB$ 
belongs, we have shown that for any $\dot{\RR}$
$\UU_\delta$-name for a stationary set preserving poset and
below any condition $\BB$ there is a dense set of posets $\QQ$
in $\UU_\delta$ which forces that $T_{\dot{\RR}}$ is stationary in $V[\dot{G}_{\UU_\delta}]$.

The thesis follows.
\end{description}
The proof of Theorem~\ref{thm:univ1} is completed.

\end{proof}

We remark that all items except the last one required only that $\delta$ is an inaccessible 
limit of $<\delta$-supercompact cardinals.

\subsection{Proof of Theorem~\ref{lem:univ4}}
\begin{proof}
For any $\BB\in \UU^{\SSP}$ we can find
$\QQ\leq_{\SSP} \BB$ in $\UU^\SSP$ which is totally rigid and freezes
$\BB$ as witnessed by some $r:\BB\to\QQ$ by the results of section~\ref{sec:denprptrpo}.

We just prove the theorem for such totally rigids $\QQ$, this simplifies the exposition since it allows to 
dispense with several bits of heavy notation, the general case for arbitrary $\BB$ is left to 
the reader\footnote{The general case requires to 
repeat the proof that follows replacing in the relevant places 
$i_{\RR}$ with $i_{\RR}\circ r$.
The proof will go through also in this case using the fact that 
$r$ is a freezing homomorphism for $\BB$.}.

Since $\QQ$ is totally rigid
there is only one correct homomorphism
\[
i_{\RR}:\QQ\to\RR
\]
for any $\RR\leq\QQ$.
For each totally rigid $\RR\in \UU^\SSP\restriction\QQ$ let
\[
k_{\RR}:\RR\to\UU^\SSP\restriction\RR
\]
be given by $r\mapsto\RR\restriction r$.
Then $k_{\RR}$ is an order and incompatibility preserving embedding of $\RR$ in the class forcing
$\UU^{\SSP}\restriction\RR$ which maps maximal antichains to maximal antichains.
Let us denote by $k$ the map $k_{\QQ}$.

Let $G$ be $V$-generic for $\QQ$ and $J$ denote its dual prime ideal.
We will show that $(\UU^{\SSP})^{V[G]}$ can be identified in $V[G]$ to the quotient
forcing given by $(\UU^{\SSP}\restriction\QQ)^{V}/_{k[G]}$ as defined in 
subsection~\ref{subsec:boolalgforc}.
To this aim first observe that in $V[G]$
\[
\downarrow k[J]=\{\RR\in (\UU^{\SSP}\restriction\BB)^V: 
\exists q\in J\, \RR\leq^V_\SSP \QQ\restriction q\}.
\]
In $V[G]$ consider the map
\[
i^*:(\UU^\SSP)^{V[G]}\to (\UU^\SSP\restriction\QQ)^V/_{k[G]}
\]
defined by $\RR/_{i_{\RR}[G]}\mapsto \RR$.
We must show that
$i^*$ is a total and surjective relation which preserve the order and incompatibility relation.
If this is the case $i^*$ is an embedding with a dense image which witnesses that
the two forcing notions are equivalent.
\begin{description}
\item[\textbf{$i^*$ is a total and surjective relation}]
Assume $\RR$ is a non-trivial element in the quotient forcing
$(\UU^\SSP\restriction\QQ)^V/_{k[G]}$.
Then $\RR$ is positive with respect to the filter (generated by) $k[G]$, thus
 $0_{\RR}\not\in i_{\RR}[G]$: else $1_{\RR}\in i_{\RR}[J]$ which gives that
$\RR\leq_{\SSP}\QQ\restriction q$ for some $q\in J$. In particular
$\RR$ would be in $\downarrow k[J]$, contradicting our assumptions.
Since 
 $0_{\RR}\not\in i_{\RR}[G]$, we can conclude that $\RR/_{i_{\RR}[J]}$ 
 is a non trivial complete boolean algebra
in $(\UU^{\SSP})^{V[G]}$ and that 
the pair $(\RR/_{i_{\RR}[G]},\RR)\in i^*$.
\item[\textbf{$i^*$ is order and compatibility preserving}]
Observe that if $j:\RR_0/i_{\RR_0}[G]\to \RR_1/i_{\RR_1}[G]$ is a correct complete
homomorphism, we have (by Proposition~\ref{lem:firstfctlem-2}(2) applied to 
$\QQ,\RR_0,\RR_1$ in the place of $\BB,\QQ_0,\RR$) 
that $j=l/_G$ for some
 complete homomorphism $l:\RR_0\to \RR_1$
 with $l$ in $V$
 and $0_{\RR_1}\not \in l[G]$.
 We can also check that $j$ is correct in $V[G]$ iff $l$ is correct in $V$.
 In particular $l$ witnesses that $\RR_0\geq^V_\SSP \RR_1$ and
 the fact that $0_{\RR_1}\not \in l[G]$ grants that 
 $\RR_1$ is a non trivial condition in 
 $(\UU^\SSP\restriction\QQ)^V/_{k[G]}$ refining $\RR_0$.
 This shows that $i^*$ is order preserving and maps non trivial conditions
 to non trivial conditions. In particular we can also conclude that $i^*$
 maps compatible conditions to compatible conditions.
 
 \item[\textbf{$i^*$ preserves the incompatibility relation}]
 We prove it by contraposition.
 Assume $\RR_0$ is compatible with $\RR_1$ in 
 $(\UU^\SSP\restriction\QQ)^V/_{k[G]}$.
 By definition of this quotient forcing 
 there is some $\CC\in \SSP$ in $V$ such that $\RR_l\geq_{\SSP}\CC$ 
 for $l=0,1$ and $1_{\CC}\not\in i_{\CC}[J]$.
We let $h_l:\RR_l\to\CC$ be the $\SSP$-correct embeddings witnessing that
$\RR_l\geq_{\SSP}\CC$ for $l=0,1$.
 Since $\QQ$ is totally rigid we get that
 $h_0\circ i_{\RR_0}=h_1\circ i_{\RR_1}=i_{\CC}$.
 Now $\CC/_{i_{\CC}[G]}$ is a non trivial $\SSP$ and complete boolean algebra 
 in $V[G]$, since $1_{\CC}\not\in i_{\CC}[J]$.
 Then, by Proposition~\ref{lem:firstfctlem-2}(\ref{lem:firstfctlem-2-1}), 
 $\CC/_{i_{\CC}[G]}\leq^{V[G]}_{\SSP} \RR_1/_{i_{\RR_1}[G]},\RR_0/_{i_{\RR_0}[G]}$ as witnessed by the 
 correct homomorphisms $h_0/_G,h_1/_G$.
 \end{description}

\end{proof}

\section{Background material on normal tower forcings}\label{subsec:nortow}

In this section we present definitions and results on normal tower forcings
which are relevant for us. 
Since we depart
in some cases from the standard terminology, we decided to be quite detailed 
in this presentation\footnote{We soon hope to have
detailed notes 
available on the author's webpage concerning the material presented in this section.}.
We assume the reader has some familiarity with tower forcings as presented in
Foreman's handbook chapter on tower forcing~\cite{FORHST} or in 
Larson's book on stationary tower forcing~\cite{LAR04}
which are reference texts for our treatment of this topic.
Another source of our presentation is Burke's paper~\cite{BUR97}.

Recall that $P_\lambda$ is the class of sets $M$ such that $M\cap\lambda\in\lambda>|M|$
and $P_\lambda(X)=P_\lambda\cap P(X)$. 
The following is well known:
\begin{fact}
Assume $I$ is a normal ideal, then $I$ is countably complete
(i.e. $<\omega_1$-complete).
A normal ideal $I$ that concentrates on $P_\lambda(X)$ for some $X\supseteq\lambda$
has completeness $\lambda$.

$\NS_X$ is a normal ideal and is the intersection of all normal ideals on $X$.
\end{fact}
 
\begin{definition}
Let $X\subset Y$ and $I,J$ be ideals on $Y$ and $X$ respectively.
$I$ canonically projects to $J$ if 
$S\in J$ if and only if $\{Z\in P(Y): Z\cap X\in S\}\in I$.
\end{definition}

\begin{definition}
$\{I_X:X\in V_\delta\}$ is a tower of ideals
if 
$I_Y$ canonically projects to $I_X$
for all $X\subset Y$ in $V_\delta$.
\end{definition} 

The following results are due to Burke~\cite[Theorems 3.1-3.3]{BUR97}:
\begin{lemma}\label{lem:burrepthm}
Assume $I$ is a normal ideal on $X$.
Then
$I$ is the canonical projection of $\NS\restriction S(I)$ where
\[
S(I)=\{M\prec H_{\theta}: M\cap X\not\in A \text{ for all }A\in I\cap M\}
\]
and $\theta$ is any large enough cardinal such that $I\in H_\theta$.

Assume $\delta$ is inaccessible and 
$\mathcal{I}=\{I_X:X\in V_\delta\}$ is a tower of normal ideals.
Then each $I_X$ is the projection of 
$\NS\restriction S(\II)$ where
\[
S(\II)=\{M\prec H_{\theta}: M\cap X\not\in A 
\text{ for all }A\in I_X\cap M \text{ and }X\in M\cap V_\delta\}
\]
and $\theta>|V_\delta|$ is a cardinal.
\end{lemma}

As usual if $I$ is an ideal on $X$, there is a natural order relation on $P(P(X))$ 
given by $S\leq_I T$ if $S\setminus T\in I$.

There is also a natural order on towers of normal ideals.
Assume $\II=\{I_X:X\in V_\delta\}$ is a tower of normal ideals.
For $S,T$ in $V_\delta$, $S\leq_\II  T$ 
if letting $X=\cup S\cup\cup T$, $S^{X}\leq_{I_X} T^X$.
It is possible to check that 
$S\leq_{\II} T$ if and only if $S\wedge S(\II)\leq T\wedge S(\II)$ as stationary sets.

\begin{definition}
Let $\delta$ be inaccessible
and assume $\II=\{I_X:X\in V_\delta\}$ is a tower of large ideals.

$\TT^{\II}_\delta$ is the tower forcing whose conditions
are the stationary sets $T\in V_\delta$ such that
$T\in I_{\cup T}^+$.

We let $S\leq_{\II} T$ if $S^X\leq_{I_X} T^X$ for
$X=\cup S\cup\cup T$.

For any $A$ stationary  such that $\cup A\supseteq V_\delta$,
we denote by $\TT^A_\delta$ the normal tower forcing
$\TT^{\II(A)}_\delta$, where 
\[
\II(A)=\{I^A_X:X\in V_\delta\}
\]
is the tower of normal ideals given by the projection of $\NS\restriction A$ to $X$ as $X$
ranges in $V_\delta$.
\end{definition}

Notice that in view of Burke's representation Theorem~\cite[Theorem 3.3]{BUR97}
(see Lemma~\ref{lem:burrepthm} above), 
any normal tower forcing
induced by a tower of normal ideal 
\[
\mathcal{I}=\{I_X:X\in V_\delta\}
\]
is of the form $\TT^{S(\mathcal{I})}_\delta$.

Observe also that modulo the identification of an element $S\in \TT^{\II}_\delta$
with its equivalence class $[S]_{\II}$ induced
by $\leq_{\II}$, and the adjunction of the class
$[\emptyset]_{\II}$, $\TT^{\II}_\delta$ can be viewed as a $<\delta$-complete boolean algebra 
$\BB_{\II}$ with the boolean operations and constants given by:
\begin{itemize}
\item
$\neg_{\BB_{\II}}[S]_{\II}=[P(\cup S)\setminus S]_{\BB_{\II}}$,
\item
$\bigvee_{\BB_{\II}}\{[S_i]_{\II}:i<\gamma\}=[\bigvee\{S_i:i<\gamma\}]_{\BB_{\II}}$ for all 
$\gamma<\delta$,
\item
$0_{\BB_{\II}}=[\emptyset]_{\II}$,
\item
$1_{\BB_{\II}}=[P(X)]_{\II}$ for any non-empty $X\in V_\delta$.
\end{itemize}

There is a tight connection between normal towers and generic elementary embeddings.
To formulate it we need to introduce the notion of $V$-normal towers of ultrafilters.
\begin{definition}
Assume $V\subset W$ are transitive models of $\ZFC$
and $X\in V$.

$G_X\in W$ is a $V$-normal ultrafilter on $X$ if it is a filter contained in $P(P(X))$
such that 
for all
regressive $f:P(X)\to X$ in $V$ there is $x\in X$ such that
$f^{-1}[\{x\}]\in G$.

$G\subset V_\delta$ in $W$ is a $V$-normal tower of ultrafilters on $\delta$ if
for all $X\subset Y$ in $V_\delta$, $G_X=\{S\in G:S\subset P(X)\}$ is  
a $V$-normal ultrafilter and 
$G_Y$ projects to $G_X$.
\end{definition}

\begin{proposition}
Let $V\subset W$ be transitive models of $\ZFC$. 
\begin{itemize}
\item
Assume $j:V\to N$ is an elementary embedding which is a definable class in $W$
and $\alpha=\sup\{\xi: j[V_\xi]\in N\}$.
Let $S\in G$ if $S\in V_{\alpha}$ and $j[\cup S]\in j(S)$.
Then $G$ is a $V$-normal tower of ultrafilters on $\alpha$.
\item
Assume $G$ is a $V$-normal tower of ultrafilters on $\delta$.
Then 
$G$ induces in a natural way a direct limit ultrapower embedding 
$j_G:V\to Ult(V,G)$, where $[f]_G\in Ult(V,G)$ if $f:P(X_f)\to V$ in $V$,
$X_f\in V_\delta$ and for any binary relation $R$,
$[f]_G\mathrel{R}_G [h]_G$ if and only if for some $\alpha<\delta$ such that 
$X_f,X_h\in V_\alpha$
we have that
\[
\{M\prec V_\alpha: f(M\cap X_f)\mathrel{R} h(M\cap X_h)\}\in G.
\]
\end{itemize}
\end{proposition}

Recall that for a set $M$ we let $\pi_M:M\to V$ denote the transitive collapse of 
the structure $(M,\in)$ onto a transitive set $\pi_M[M]$ and we let $j_M=\pi_M^{-1}$.
We state the following general results about generic elementary embeddings induced by 
$V$-normal tower of ultrafilters:
\begin{theorem}
Assume $V\subset W$ are transitive models of $\ZFC$,
$\delta$ is a limit ordinal, $\lambda<\delta$ is regular in $V$,
$G\in W$ is a $V$-normal tower of ultrafilters on $\delta$ which concentrates on $P_\lambda$,
 and $V$ is a definable class in $W$
by the parameter $p$. 
Then
\begin{enumerate}
\item \label{thm:wstf-1}
$j_G:V\to Ult(V,G)$ is a definable class in $W$ in the parameters $G$, $V_\delta$, $p$.
\item\label{thm:wstf-3}
$Ult(V,G)\models\phi([f_1]_G,\dots,[f_n]_G)$ if and only if for some $\alpha<\delta$ such that
$f_i:P(X_i)\to V$ with $X_i\in V_\alpha$ for all $i\leq n$:
\[
\{M\prec V_\alpha: V\models \phi(f_1(M\cap X_1),\dots,f_n(M\cap X_n))\}\in G.
\]
\item\label{thm:wstf-5}
For all $x\in V_\delta$, $x\in Ult(V,G)$ is represented by $[\rho_x]_G$ where
$\rho_x: P(V_\alpha)\to V_\lambda$ is such that $x\in V_\alpha$ for some $\alpha<\delta$
and maps any $M\prec V_\alpha$ with $x\in M$ to $\pi_M(x)$.
\item\label{thm:wstf-4}
$(V_\delta,\in)$ is isomorphic to the substructure of $(Ult(V,G),\in_G)$ given
by the equivalence classes of the functions $\rho_x:P(V_\alpha)\to V_\alpha$ described above.
\item \label{thm:wstf-2}
$\crit(j_G)=\lambda$ and $\delta$ is isomorphic to an initial segment
of the linear order 

$(j_G(\lambda),\in_G)$.

\item\label{thm:wstf-6}
For all $x\in V$, $j_G(x)\in Ult(V,G)$ is represented by the equivalence class of the constant
function $c_x: P(X)\to V$ which maps any $Y\subset X$ to $x$ for an arbitrary choice of 
$X\in V_\delta$.
\item\label{thm:wstf-7}
For any $x\in V_\delta$, 
$j_G[x]\in Ult(V,G)$ is represented in $Ult(V,G)$ 
by the equivalence class of the identity function on $P(x)$.
\item \label{thm:wstf-8}
Modulo the identification of $V_\delta$ with the subset of $Ult(V,G)$ defined in~\ref{thm:wstf-4},
$G$ is the subset of $Ult(V,G)$ defined by $S\in G$ if and only if $j_G[\cup(S)]\in_G j_G(S)$.
Moreover $Ult(V,G)\models j_G(\lambda)\geq\delta$.
\item \label{thm:wstf-9}
For any $\alpha<\delta$, let $G_\alpha=G\cap V_\alpha$.
Let $k_{\alpha}:Ult(V,G_\alpha)\to Ult(V,G)$ by defined by $[f]_{G_\alpha}\mapsto [f]_G$.
Then $k_\alpha$ is elementary, $j_G=k_\alpha\circ j_{G_\alpha}$ and 
$\crit(k_\alpha)=j_{G_\alpha}(\crit(j_G))\geq\alpha$.
\end{enumerate}
\end{theorem}

A key observation is that $V$-generic filters for $\TT^{\II}_\delta$ are also $V$-normal
filters, but $V$-normal tower of ultrafilters $G$ on $\delta$ which are contained in 
$\TT^{\II}_\delta$ may not 
be fully $V$-generic for $\TT^{\II}_\delta$. However such
$V$-normal tower of ultrafilters, while they may not be strong enough to decide the theory of
$V^{\TT^{\II}_\delta}$, are sufficiently informative to decide the behaviour of the generic ultrapower
$Ult(V,\dot{G}_{\TT^{\II}_\delta})$ which can be defined inside $V^{\TT^{\II}_\delta}$.

\begin{lemma}
Assume $V$ is a transitive model of $\ZFC$
and $\II\in V$ is a normal tower of height $\delta$.
The following holds:
\begin{enumerate}
\item
For any $V$-normal tower of ultrafilters $G$ on $\delta$ contained in $\TT^{\II}_\delta$ the following are equivalent:
\begin{itemize}
\item
$Ult(V,G)$ models $\phi([f_1]_G,\dots,[f_n]_G)$.
\item
There is $S\in G$
such that for all $M\in S$ 
\[
V\models \phi(f_1(M\cap X_{f_1}),\dots,f_n(M\cap X_{f_n})).
\]
\end{itemize}

\item
For any $S\in \TT^{\II}_\delta$, $f_1:X_{f_1}\to V,\dots,f_n:X_{f_n}\to V$ in $V$, and
any formula $\phi(x_1,\dots,x_n)$ the following are equivalent:
\begin{itemize}
\item
For any  $\alpha$ such that $X_{f_i}\in V_\alpha$ for all $i\leq n$ 
\[
\{M\prec V_\alpha: V\models\phi(f_1(M\cap X_{f_1}),\dots,f_n(M\cap X_{f_n}))\}\geq_{\II}S.
\] 
\item
For all $V$-normal tower of ultrafilters $G$ on $\delta$ contained in $\TT^{\II}_\delta$ to which
$S$ belongs we have that
\[
Ult(V,G)\models \phi([f_1]_G,\dots,[f_n]_G).
\]
\end{itemize}

\end{enumerate}
\end{lemma}

In the situation in which $W$ is a $V$-generic extension of
$V$ for some poset $P\in V$ and $G\in W$ is a $V$-normal tower of ultrafilters we can use 
$G$ to define a tower of normal ideals in $V$ as follows:

\begin{lemma}\label{lem:embntcba}
Assume $V$ is a transitive model of $\ZFC$.
Let $\delta$ be an inaccessible cardinal in $V$
and $\QQ\in V$ be a complete boolean algebra.
Assume $H$ is $V$-generic for $\QQ$
and $G\in V[H]$ is a $V$-normal tower of ultrafilters on $\delta$.
Let $\dot{G}\in V^{\QQ}$ be such that $\val_H(\dot{G})=G$
and
\[
\llbracket \dot{G}\text{ is a $V$-normal tower of ultrafilters on }\delta\rrbracket_{\QQ}=1_{\QQ}.
\]
Define
\[
\II(\dot{G})=\{I_X:X\in V_\delta\}\in V
\]
by setting $A\in I_X$ if and only if 
$A\subset P(X)$ and 
\[
\llbracket A\in\dot{G}\rrbracket_{\QQ}=0_{\QQ}.
\]
Then $\II(\dot{G})\in V$ is a tower of normal ideals.
\end{lemma}

In particular the above Lemma can be used to define 
$<\delta$-complete injective homomorphisms $i_{\dot{G}}:\TT^{\II(\dot{G})}_\delta\to \QQ$
whenever $\QQ$ is a complete boolean algebra and
$\dot{G}\in V^{\QQ}$ is forced by $\QQ$
to be a $V$-normal tower of ultrafilters on some $\delta$ which is inaccessible in $V$.
These homomorphisms are defined by $S\mapsto \llbracket S\in\dot{G}\rrbracket_{\QQ}$.
It is easy to check that these maps are indeed $<\delta$-complete injective homomorphisms,
however without further assumptions on $\dot{G}$ 
these maps cannot in general be extended to complete homomorphisms of the 
respective boolean completions because we cannot control 
if $i_{\dot{G}}[\mathcal{A}]$ is a maximal antichain of $\QQ$ whenever $\mathcal{A}$ is a
maximal antichain of $\TT^{\II(\dot{G})}_\delta$.
There are however two nice features of these mappings:
\begin{proposition}\label{prop:embntcba}
Assume $\QQ\in V$ is a complete
boolean algebra and $\dot{G}\in V^{\QQ}$ is forced by $\QQ$
to be a $V$-normal tower of ultrafilters on some $\delta$. Then the following holds:
\begin{itemize}
\item
The preimage of a $V$-generic filter for $\QQ$ under $i_{\dot{G}}$
is a $V$-normal tower of ultrafilters contained in $\TT^{\II(\dot{G})}_\delta$,
\item
$i_{\dot{G}}:\TT^{\II(\dot{G})}_\delta\to \QQ$ extends to an isomorphism
of $\BBB(\TT^{\II(\dot{G})}_\delta)$ with $\QQ$ if its image is dense in
$\QQ$.
\end{itemize}
\end{proposition}

\begin{definition}
Let $\lambda=\gamma^+$,
$\delta$ be inaccessible and 
$\II=\{I_X:X\in V_\delta\}$ be  a tower of normal ideals which concentrate on 
$P_\lambda(X)$.
The tower $\TT^{\II}_\delta$ is presaturated if
$\delta$ is regular in $V[G]$ for all $G$ $V$-generic for $\TT^{\II}_\delta$.

\end{definition}

The following is well known:
\begin{theorem}
Let $\lambda=\gamma^+$,
$\delta$ be inaccessible and 
$\II=\{I_X:X\in V_\delta\}$ be  a tower of normal
ideals which concentrate on 
$P_\lambda(X)$.

$\TT^{\II}_\delta$ is a presaturated normal tower if and only whenever
$G$ is $V$-generic for $\TT^{\II}_\delta$
$V[G]$ models that $Ult(V,G)$ is closed under $<\delta$-sequences in
$V[G]$.
\end{theorem}

\begin{theorem}[Woodin]~\cite[Theorem 2.7.7]{LAR04}
Let $\TT^{\omega_2}_\delta$ be the tower given by stationary sets in $V_\delta$ 
which concentrate on $P_{\omega_2}$.
Assume $\delta$ is a Woodin cardinal. Then $\TT^{\omega_2}_\delta$ is a presaturated
tower.
\end{theorem}

\begin{lemma}
Assume $\TT^{\II}_\delta$ is a presaturated tower such that $\II$ concentrates on 
$P_{\omega_2}(H_{\delta^+})$.
Then $\TT^{\II}_\delta$ is stationary set preserving.
\end{lemma}
\begin{proof}
This is a  consequence of the fact that if a tower is presaturated and concentrates on
$P_{\omega_2}$,
the stationary subsets of $\omega_1$ of the generic extension belong to the generic
ultrapower. For more details see for example~\cite[Section 2.4]{VIAMMREV}.
\end{proof}

We also have the following duality property:
\begin{lemma}\label{lem:presatequiv}
An $\SSP$ complete boolean algebra 
$\BB$ which collapses $\delta$ to become the second uncountable cardinal
is forcing equivalent to a presaturated normal tower of height $\delta$
if and only if 
there is a $\BB$-name $\dot{G}$ for a $V$-normal tower of ultrafilters on $\delta$ such that
\begin{itemize}
\item
$j_{\dot{G}}:V\to Ult(V,\dot{G})$ is forced by $\BB$ to be an elementary embedding such that
$j_{\dot{G}}(\omega_2)=\delta$ and $Ult(V,\dot{G})^{<\delta}\subset Ult(V,\dot{G})$,
\item
The map $S\mapsto \llbracket S\in \dot{G}\rrbracket_{\BB}$ which embeds 
$\TT^{\II(\dot{G})}_\delta$ into $\BB$ has a dense image.
\end{itemize} 
\end{lemma}

Finally we will need the following lemma to produce by forcing presaturated towers
in generic extensions:
\begin{lemma}\label{lem:forcpresat}
Assume $j:V\to M$ is an almost huge embedding with $\delta=\crit(j)$.
Let $P\subset V_\delta$ be a forcing notion such that
$\BBB(P)$ is $(<\delta,<\delta)$-presaturated.

Assume also that 
\begin{itemize}
\item
$j\restriction P$
can be extended to a complete homomorphism between $\BBB(P)$ and
$\BBB(j(P))\restriction q$ for some $q\in \BBB(j(P))$,
\item
$\BBB(j(P))\restriction q$ is $(<j(\delta),<j(\delta))$-presaturated.
\end{itemize}
Let $H$ be $V$-generic for $\BBB(j(P))$ with $q\in H$ and $G=j^{-1}[H]$.
Then in $V[H]$ the map
\[
\bar{j}: V[G]\to M[H]
\]
given by $\val_G(\tau)\mapsto \val_H(j(\tau))$ is elementary, $\crit{\bar(j)}=\delta$ and
$M[H]^{<j(\delta)}\subset M[H]$.
\end{lemma}
\begin{proof}
Let $\QQ=\BBB(j(P))\restriction q$ and $\BB=\BBB(P)$.
The hypotheses grants that $j[G]\subset H$ and thus that
\[
\llbracket\phi(\tau_1,\dots,\tau_n)\rrbracket_{\BB}\in G
\text{ iff }
\llbracket\phi(j(\tau_1),\dots,j(\tau_n))\rrbracket_{\QQ}\in H
\]
for all formulas $\phi$ and $\tau_1,\dots,\tau_n\in V^{\BB}$.
This immediately gives that $\bar{j}$ is elementary and well
defined.

Now we check that $M[H]^{<j(\delta)}\subset M[H]$ in $V[H]$.

Let $\tau:\gamma\to \Ord$  be in $V^{\QQ}$ 
a name for a sequence of ordinals of length $\gamma<j(\delta)$.
Then there is a family $\{A_i:i<\gamma)\}\in V$ of maximal antichains of $j(P)\restriction q\subseteq\QQ$ such that
for all $a\in A_i$ there is $\beta$ such that 
\[
\llbracket \beta=\tau(i)\rrbracket_{\QQ}\geq a.
\]
By the $(<j(\delta,<j(\delta))$-presaturation of $\QQ$,
there is $r\leq q$ in $H\cap j(P)$ such that 
\[
B_i=\{p\in A_i: p\wedge r>0_{\QQ}\}\subseteq j(P)\restriction q\subseteq M
\]
has size less than $j(\delta)$ for all $i<\gamma$.
This gives that $\{B_i:i<\gamma\}\in M$
and that if we set
\[
\sigma=\{\langle (i,\beta),a\rangle:\,i<\gamma,\, a\in B_i,\, 
\llbracket \tau(i)=\beta\rrbracket_{\QQ}\geq a\},
\]
we have at the same time that $\sigma\in M$ and that 
\[
\llbracket \sigma=\tau\rrbracket_{\QQ}\geq r.
\]
In particular we get that $\val_H(\tau)=\val_H(\sigma)\in M[H]$ as was
to be shown.
\end{proof}

%%%%%%%%%%%%%%%%%%%%%%%%%%%%%%%%%%%%%%%%%%%%%

\section{$\MM^{+++}$}\label{sec:MM+++}

In this section we introduce the forcing axiom $\MM^{+++}$ as a density property of the
class forcing $\UU^{\SSP}$ and we show how to derive the generic absoluteness
of the Chang model $L(\Ord^{\omega_1})$ with respect to models of 
$\MM^{+++}+$\emph{large cardinals}.
We first show that ---in the presence of class many Woodin cardinals--- 
$\MM^{++}$ can be characterized as the statement that the class of presaturated
towers of normal filters is dense in $\UU^{\SSP}$. 
Next we analyze the interactions betwen the category forcing $\UU^{\SSP}$ and the class forcing given 
by stationary sets $S$ contained in $P_{\omega_2}(\cup S)$, an interaction
which shows up only assuming $\MM^{++}$.
We then introduce $\MM^{+++}$ as (a slight strenghtening of) the assertion that the class of totally rigid,
presaturated towers of normal filters is dense in $\UU^{\SSP}$. Next we prove that
$\MM^{+++}$ is equivalent to the assertion that $\UU_\delta$ is a presaturated tower of normal filters
for all $\delta$ which are $\Sigma_2$-reflecting cardinals.
Then we show how to use the presaturation of $\UU_\delta$ to derive the absoluteness of the 
Chang model $L(\Ord^{\omega_1})$.
Finally we prove the consistency of $\MM^{+++}$ relative to large cardinal axioms.
The reader who is just interested in the results concerning $\MM^{+++}$ can skip the next subsection.

\subsection{$\MM^{++}$ as a density property of $\UU^{\SSP}$}\label{subsec:MM++}
We shall denote by $\TT^{\omega_2}_\delta$ the stationary tower whose elements
are stationary sets $S$ in $V_\delta$ which concentrate on
$P_{\omega_2}(\cup S)$. 
A key property of this partial order 
is that it is at the same stationary set preserving and a presaturated tower whenever $\delta$ is 
a Woodin cardinal. 
In~\cite[Theorem 2.16]{VIAMMREV} we showed:

\begin{theorem}
Assume $\delta$ is a Woodin cardinal and
let $P\in V_\delta$ be a partial order.
Then the following are equivalent:
\begin{enumerate}
\item
$T_{\BBB(P)}$ is stationary,
\item
$\BBB(P)\geq_\SSP \TT^{\omega_2}_\delta\restriction S$
for some stationary set $S\in \TT^{\omega_2}_\delta$.
\end{enumerate}
\end{theorem}

In particular we have the following immediate corollary:
\begin{corollary}\label{cor:MM++den}
Assume there are class many Woodin cardinals. Then the following are equivalent:
\begin{enumerate}
\item
$T_{\BB}$ is stationary for all $\BB\in\SSP$.
\item
The class of presaturated normal towers is dense in $\UU^{\SSP}$.
\end{enumerate}
\end{corollary}

Notice that in the presence of $\MM^{++}$ and class many Woodin 
cardinals we have a further characterization of total rigidity:

\begin{proposition}\label{lem:eqtrbis}
Assume $\MM^{++}$ and there are class many Woodin cardinals .
Then the following are equivalent for a $\BB\in\SSP$:
\begin{enumerate}
\item
$\BB$ is totally rigid.
\item\label{lem:eqtr2bis}
$G(M,\BB)=\{b\in M: M\in T_{\BB\restriction b}\}$
is the unique correct $M$-generic filter for $\BB$ for a club of $M\in T_{\BB}$.
\end{enumerate}
\end{proposition}
\begin{proof}
We first show that $G(M,\BB)$ is a correct 
$M$-generic filter for $\BB$ iff there is a unique such
correct $M$-generic filter.

So assume there are two distinct correct $M$-generic filter for $\BB$ $H_0,H_1$.
Let $b\in H_0\setminus H_1$. Then $M\in T_{\BB\restriction b}\cap T_{\BB\restriction\neg b}$
as witnessed by $H_0,H_1$, thus $b,\neg b\in G(M,\BB)$ and $G(M,\BB)$ is not a filter.

Conversely assume $H$ is the unique correct $M$-generic filter for $\BB$.
Then $b\in H$ gives that $M\in T_{\BB\restriction b}$. Thus $H\subseteq G(M,\BB)$.
Now if $c\in G(M,\BB)\setminus H$ there is a correct $M$-generic filter $H^*$ for $M$
with $c\in H^*\setminus H$. This contradicts the uniqueness assumption on $H$.
Thus $H= G(M,\BB)$ as was to be shown.

Now we prove the equivalence of total rigidity with~\ref{lem:eqtr2bis}.

Assume first that~\ref{lem:eqtr2bis} fails. Let $S\subset T_{\BB}$ be a stationary set
such that for all $M\in S$ there are at least two distinct correct 
$M$-generic filters $H^M_0$, $H^M_1$.
For each such $M$ we can find $b_M\in M\cap (H^M_0\setminus H^M_1)$.
By pressing down on $S$ and refining $S$ if necessary, 
we can assume that $b_M=b^*$ for all $M\in S$.
Let $\delta>|\BB|$ be  a Woodin cardinal. For $j=0,1$ define
$i_j:\BB\to \TT^{\omega_2}_\delta\restriction S$ by 
\[
b\mapsto\{M\in S: b\in H^M_j\}.
\]
Then $i_0,i_1$ are complete homomorphisms such that
\[
\BB\Vdash \TT^{\omega_2}_\delta\restriction S/_{i_j[\dot{G}_{\BB}]}\text{ is stationary set preserving.}
\]
and 
$i_0(b^*)=S=i_1(\neg b^*)$.
In particular we get that $i_0$ witnesses that 
$\BB\restriction b^*\geq \TT^{\omega_2}_\delta\restriction S$
and $i_1$ witnesses that $\BB\restriction \neg b^*\geq \TT^{\omega_2}_\delta\restriction S$.
All in all we have that $\BB\restriction b^*$ and $\BB\restriction \neg b^*$ are compatible 
conditions in $\UU^{\SSP}$, i.e. $\BB$ is not totally rigid.

Now assume that $\BB$ is not totally rigid. Let
$i_0:\BB\restriction b\to \CC$ and $i_1:\BB\restriction \neg_{\BB}b\to\CC$ be distinct
complete homomorphisms of $\BB$ into $\CC$ such that for $j=0,1$
\[
\BB\Vdash \CC/_{i_j[\dot{G}_{\BB}]}\text{ is stationary set preserving.}
\]
Then for all $M\in T_{\CC}$ such that $i_0,i_1\in M$ we can pick $H_M$ a
correct $M$-generic filter for $\CC$.
Thus $G_j=i_j^{-1}[H_M]$ for $j=0,1$ are both correct and $M$-generic and such that $b\in G_0$ and 
$\neg_{\BB} b\in G_1$. In particular we get that for a club of 
$M\in T_{\CC}\restriction H_{|\BB|^+}\subseteq T_{\BB}$ there are at least two $M$-generic 
filter for $\BB$,
i.e.~\ref{lem:eqtr2bis} fails.

\end{proof}

\subsubsection*{Duality between $\SSP$-forcings and stationary sets concentrating 
on $P_{\omega_2}$}\label{subsec:comrel}

We outline some basic properties which tie the arrows of $\UU^{\SSP}$ with 
the order relation on stationary sets concentrating on $P_{\omega_2}$
assuming that $\MM^{++}$ holds.

\begin{lemma}\label{lem:cmpUUdelta}
Assume $\MM^{++}$ and there are class many Woodin cardinals. 
Then $\BB_0,\BB_1$ are compatible conditions in $\UU^{\SSP}$
(i.e. there are arrows $i_j:\BB_j\to\CC$ for $j=0,1$ and $\CC$ fixed in $\UU^{\SSP}$) 
if and only if
$T_{\BB_0}\wedge T_{\BB_1}$ is stationary.
\end{lemma}
\begin{proof}
First assume that $\CC\leq \BB_0,\BB_1$.
We show that $T_{\BB_0}\wedge T_{\BB_1}$ is stationary.
Let $i_0:\BB_0\to \CC$, $i_1:\BB_1\to\CC$ be atomless complete homomorphisms
such that
\[
\BB_j\Vdash \CC/_{i_j[\dot{G}_{\BB}]}\text{ is stationary set preserving.}
\]
For all $M\in T_{\CC}$ such that $i_j\in M$ for $j=0,1$ 
pick $G_M$ correct $M$-generic filter for $\CC$.
Let $H_M^j=i_j^{-1}[G_M]$, then $H_M^j$ is a correct $M$-generic for $\BB_j$,
since each $i_j$ is such that 
\[
\BB_j\Vdash \CC/_{i_j[\dot{G}_{\BB}]}\text{ is stationary set preserving.}
\]
In particular $M\cap  H_{|\BB_j|^+}\in T_{\BB_j}$ for $j=0,1$.
Thus $T_{\CC}\restriction H_{|\BB_j|^+}\subset T_{\BB_j}$ for $j=0,1$,
i.e. $T_{\BB_0}\wedge T_{\BB_1}$ is stationary.

Conversely assume that $T_{\BB_0}\wedge T_{\BB_1}$ is stationary.
For each $M\in S=T_{\BB_0}\wedge T_{\BB_1}$ pick
$H_M^j$ correct $M$-generic filter for $\BB_j$ for $j=0,1$.
Fix a Woodin cardinal $\delta>|S|$.
Let $\TT^{\omega_2}_\delta$ denote the stationary tower with critical point $\omega_2$ and height
$\delta$.
Let $i_j:\BB_j\to \TT^{\omega_2}_\delta\restriction S$ map 
\[
b\to\{M\in S: b\in H^j_M\}.
\] 
Then each 
$i_j$ is a complete embedding such that
\[
\BB_j\Vdash \TT^{\omega_2}_\delta\restriction S/_{i_j[\dot{G}_{\BB}]}\text{ is stationary set preserving.}
\]
 i.e.
$\BB_j\geq \TT^{\omega_2}_\delta\restriction S$ for $j=0,1$ showing that 
$\BB_0$ and $\BB_1$ are compatible conditions in $\UU^{\SSP}$.
\end{proof}

We can also give a simple representation of totally rigid boolean algebras and a
characterization of the complete embeddings
between totally rigid complete boolean
algebras:

\begin{fact}
Assume $\BB$ is totally rigid and $T_{\BB\restriction b}$ is stationary
for all $b\in \BB^+$.
Then $\BB$ is isomorphic to the complete boolean subalgebra
$\{T_{\BB\restriction b}:b\in\BB\}$ of the boolean algebra $P(T_{\BB})/\NS$.
\end{fact}
Notice that in the above setting, $P(T_{\BB})/\NS$ may not be a complete boolean algebra and may
not be stationary set preserving,
while $\{T_{\BB\restriction b}:b\in\BB\}$ is an $\SSP$ and complete boolean subalgebra.

\begin{fact}
Assume $\MM^{++}$ holds.
Assume $\BB\geq_{\SSP}\QQ$ are totally rigid and complete boolean algebras.
Let $i:\BB\to\QQ$ be the unique $\SSP$-correct homomorphism between $\BB$ and $\QQ$.
Then for all $b\in\BB$ and $q\in\QQ$, 
$T_{\BB\restriction b}\wedge T_{\QQ\restriction q}$ is stationary
if and only if $i(b)\wedge_{\QQ} q>0_{\QQ}$.
\end{fact}
\begin{proof}
Left to the reader.
\end{proof}

It is also immediate to check the following:
\begin{fact}\label{fac:funcrelvee}
For any family $\mathcal{A}\subset \UU^{\SSP}$ 
\[
T_{\bigvee\mathcal{A}}=\bigvee\{T_{\BB}:\BB\in\mathcal{A}\}.
\]
\end{fact}

All in all assuming $\MM^{++}$ and class many Woodin cardinals we are in the following
situation:

\begin{enumerate}
\item
$\MM^{++}$ can be defined as the statement that the class 
$\mathsf{PT}$ of presaturated 
towers of normal filters is dense in the category $\UU^{\SSP}$.
\item
We also have in the presence of $\MM^{++}$ a functorial map 
\[
F:\UU^{\SSP}\to \{S\in V: S\text{ is stationary and concentrate on }P_{\omega_2}\}
\]
defined by $\BB\mapsto T_{\BB}$ which: 
\begin{itemize}
\item
is order and incompatibility preserving,
\item
maps 
set sized suprema to set sized suprema in the respective class partial orders,
\item
gives a neat representation of the separative quotients of totally rigid partial orders 
and of the complete embeddings between them.
\end{itemize}
\end{enumerate}
It is now 
tempting to conjecture that it is possible to reflect this to some $V_\delta$ and obtain that 
the map
$F\restriction V_\delta$ defines a complete embedding of $\UU_\delta$ into 
$\TT^{\omega_2}_\delta$. However we just have that $F$ preserves suprema of set sized
subsets of $\UU^{\SSP}$ which would reflect to the fact
$F\restriction V_\delta$ defines a $<\delta$-complete embedding of $\UU_\delta$ into 
$\TT^{\omega_2}_\delta$. However we have no reason to expect that
$F\restriction V_\delta$ extends to a complete homomorphism of the respective boolean completion
because neither of the above posets is $<\delta$-CC.

On the other hand we have shown in the previous sections 
that in the presence of class many supercompact
cardinals we have that the class $\TR$ of totally rigid posets 
is dense in the category $\UU^{\SSP}$.
What about the intersection of the classes $\TR$ and $\mathsf{PT}$?
Can there be densely many presaturated towers which are also totally rigid
in $\UU^{\SSP}$?

A positive answer to this question leads to the definition of $\MM^{+++}$.

\subsection{Strongly presaturated towers and $\MM^{+++}$}

\begin{definition}
Let $\II=\{I_X:X\in V_\delta\}$ 
be a tower of normal filters of height $\delta$ which concentrates on $P_{\omega_2}$.
$S$ is a fixed point for $\II$ if $S\geq T_{\TT^{\II}_\delta\restriction S}$.

A normal tower $\TT=\TT^{\II}_\delta$ is strongly 
presaturated\footnote{There is a slight ambiguity between $\TT$ meant as 
a partial order and $\TT$ meant as a boolean algebra. We remark that 
any $S\in V_\delta$ such that $[S]_{\II}=0_{\TT}$ is a fixed point for $\TT$ since
$T_{\TT\restriction S}$ and $T_{\TT}\wedge S$ are always non-stationary in this case
(see Fact~\ref{fac:statfixpoint} below). 
Thus this definition
applies just to the $S\in V_\delta$ such that $S\not \in I_{\cup S}$.} 
if it is presaturated,
has a dense set of fixed point, and  $T_{\TT\restriction S}$ is stationary for all
fixed points $S\in\TT$ such that $S\not\in I_{\cup S}$.

$\SPT$ denotes the class of strongly presaturated towers of normal filters.
\end{definition}

We shall see below that for any $\BB$, $T_{\BB}$ is a fixed point of 
$\TT$ for any presaturated
tower $\TT$ to which $T_{\BB}$ belongs.
However it is well possible (and we conjecture that) for all $\delta$
$\TT^{\omega_2}_\delta$ is never strongly presaturated, 
since it is likely that the family $\{T_{\BB}:\BB\in \UU_\delta\}$ is not dense in 
$\TT^{\omega_2}_\delta$. On the other hand if $\UU_\delta$ is 
forcing equivalent to a presaturated normal tower $\TT$ we shall see
that $\{T_{\BB}:\BB\in \UU_\delta\}$ is dense in 
$\TT$. In particular this would give that $\UU_\delta\in \SPT$.
Moreover we shall see that (in almost all cases) a tower $\TT$ of height $\delta$
is strongly presaturated iff 
there is a family $D\subset \UU_\delta$ such that
$\{T_{\BB}:\BB\in D\}$ is dense in 
$\TT$.
These comments are a motivation for the following definition:

\begin{definition}\label{def:MM+++}
$\MM^{+++}$ holds if $\SPT$ is a dense class in $\UU^{\SSP}$.
\end{definition}

Notice that $\MM^{+++}$ strenghtens $\MM^{++}$ in view of Corollary~\ref{cor:MM++den},
we shall see in Lemma~\ref{lem:totrigSPTMM++} below 
that $\MM^{++}$ entails that any $\TT\in\SPT$ is totally rigid.
In particular assuming $\MM^{+++}$ the class of $\SPT$ posets is a subset of the intersection of 
the class of presaturated towers and  the class of totally rigid posets,
however we haven't been able to establish whether it is a proper subset or a characterization of the 
intersection.

The plan of the remainder of this section is the following:
\begin{enumerate}
\item
We shall introduce some basic properties of the elements $\TT$ of $\SPT$.
\item
We will show that if $\UU_\delta\in \SSP$ and $\SPT\cap V_\delta$ is dense in $\UU_\delta$,
then actually $\UU_\delta$ is forcing equivalent 
to a strongly presaturated normal tower.
\item We will use the previous item in combination with Theorem~\ref{lem:univ4} 
 to prove that the theory $\ZFC+\MM^{+++}+$\emph{there 
are class many $\Sigma_2$-reflecting cardinals $\delta$ which are a limit of $<\delta$-supercompact
cardinals} makes the the theory of the Chang model
$L(\Ord^{\leq\aleph_1})$ invariant under stationary set preserving forcings which preserve
$\MM^{+++}$.
\item
Finally we will show that essentially any ``iteration'' of length $\delta$ which produces 
a model of $\MM^{++}$ 
actually produce a model of $\MM^{+++}$, provided $\delta$ is super almost huge.
\end{enumerate}

\subsubsection*{Basic properties of $\SPT$}

\begin{lemma}\label{lem:basspt1}
Assume $\TT=\TT^{\II}_\delta$ is a strongly presaturated tower of normal filters.
Then for all $M\in T_{\TT}$
\[
G(M,\TT)=\{S\in M\cap V_\delta: M\cap\cup S\in S\}
\]
is the unique correct $M$-generic filter for $\TT$.
\end{lemma}
\begin{proof}
Observe that if $M\in T_{\TT}$, $G$ is a correct 
$M$-generic for $\TT$, and $S\in G$ is a fixed point, we get that
$M\in T_{\TT\restriction S}$ and thus that $M\cap\cup S\in S$ since
$S\geq T_{\TT^{\II}_\delta\restriction S}$. 
In particular (since the set of fixed points is dense in $\TT$) 
we get that $G(M,\TT)\supseteq G$. 
Since $G$ is an ultrafilter on the boolean algebra
$\TT\cap M$, this gives that 
$G=G(M,\TT)$ is a correct $M$-generic ultrafilter for $\TT$.
The thesis follows.
\end{proof}

The following fact is almost self-evident:
\begin{fact}\label{fac:statfixpoint}
Assume $\II=\{I_X:X\in V_\delta\}$ is a tower of normal filters such that
$\TT=\TT^{\II}_\delta$ is strongly presaturated. Let $S\in V_\delta$ be a fixed point for $\TT$.
Then $S\wedge T_{\TT}=T_{\TT\restriction S}$.
Moreover the latter sets are stationary iff $S\in I_{\cup S}^+$.
\end{fact}
\begin{proof}
First assume $S\not\in I_{\cup S}^+$. We show that $S\wedge T_{\TT}$ and
$T_{\TT\restriction S}$ are both non stationary.
Since
$S\in I_{\cup S}$ we get that 
$[P(\cup S)\setminus S]_{\II}=1_{\TT}$ and $[S]_{\II}=0_{\TT}$.
Thus 
for all $M\in T_{\TT}$ we get that 
$P(\cup S)\setminus S\in G(M,\TT)$, i.e. that $M\cap\cup S\not\in S$.
This gives at the same time that $T_{\TT\restriction S}$ and 
$S\wedge T_{\TT}$ are both non-stationary.

Next assume $S\in I_{\cup S}^+$, i.e. $[S]_{\II}>0_{\TT}$.
Then $M\in T_{\TT}\wedge S$ iff $S\in G(M,\TT)$ iff $M\in T_{\TT\restriction S}$. 
This gives that $S\wedge T_{\TT}=T_{\TT\restriction S}$ also in this case.
Finally by definition of strong presaturation, $T_{\TT\restriction S}$ is always stationary
if $S\in I_{\cup S}^+$.
\end{proof}
\begin{fact}\label{fac:sptotp}
Assume 
$\TT=\TT^{\II}_\delta$ is a strongly presaturated tower of normal filters of height $\delta$
which makes $\delta$ the second uncountable cardinal.
Then for all $M\in T_{\TT}$ and all $h:P(X)\to\omega_2$ in $M\cap V_\delta$, we have that
$h(M\cap X)<\otp(M\cap\delta)$.
\end{fact}
\begin{proof}
Pick $M\in T_{\TT}$ (which is stationary by our assumption that $\TT$ is strongly presaturated).
Observe that $G(M,\TT)$ is the unique correct $M$-generic filter for 
$\TT$. Let $H=\pi_M[G(M,\TT)]$ and $N=\pi_M[M]$.
Then $H$ is a correct $N$-generic filter for $\bar{\TT}=\pi_M(\TT)$ and $N$ models that
$\bar{\TT}$ is a presaturated tower of height $\bar{\delta}=\pi_M(\delta)$.
In particular if $h:P(X)\to\omega_2$ is in $M$ we get that
\[
Ult(N,H)\models[\pi_M(h)]_H<\pi_M(\delta)=\omega_2^{N[H]}.
\]
Thus there is $\gamma\in M\cap\delta$ such that 
\[
[\pi_M(h)]_H=\pi_M(\gamma)=[\rho_{\pi_M(\gamma)}]_H.
\]
This occurs if and only if 
\[
S=\{M'\prec V_{\gamma+1}: \otp(M'\cap\gamma)=
\rho_\gamma(M')=h(M'\cap X)\}\in G(M,\TT)=\{S\in \TT: M\cap\cup S\in S\}
\]
The conclusion follows.
\end{proof}
\begin{fact}
For any $\BB\in \SSP$,
$T_{\BB}$ is a fixed point of $\TT$ for any
presaturated normal tower  $\TT=\TT^{\II}_\delta$ of height $\delta$ such that $\BB\in V_\delta$. 
\end{fact}
\begin{proof}
In case $T_{\BB}\equiv_{\II}\emptyset$, 
$T_{\TT\restriction T_{\BB}}$ is non stationary
and thus $T_{\BB}\geq T_{\TT\restriction T_{\BB}}$ is a fixed point of $\TT$.

Now observe that 
if $G$ is a $V$-generic filter for $\TT$,
then 
\[
T_{\BB}\in G\text{ iff }M=j_G[H_{|\BB|^+}]\in j_G(T_{\BB})=T_{j_G(\BB)}.
\]
This gives that $H$ in $Ult(V,G)$ is a correct $M$-generic filter for $j_G(\BB)$
iff $\bar{H}=\pi_M(H)$ is a $V$-generic filter for $\BB$ such that $V[\bar{H}]$ is correct about 
the stationarity of its subsets of $\omega_1$.
In particular we get that if
$M\in T_{\TT\restriction T_{\BB}}$, then $N=\pi_M[M]$ has an $N$-generic filter
for $\pi_M(\BB)$ which computes correctly the stationarity of its subsets of $\omega_1$.
Thus $M\cap\cup T_{\BB}\in T_{\BB}$.
All in all we have shown that $T_{\BB}\geq T_{\TT\restriction T_{\BB}}$, i.e.
$T_{\BB}$ is a fixed point of $\TT$. 
\end{proof}

\begin{lemma}\label{lem:presat-strpresat}
Assume $\UU_\delta\in\SSP$ forces that there is a $\UU_\delta$-name $\dot{G}$ for a 
$V$-normal tower of ultrafilters on $\delta$ such that
$j_{\dot{G}}:V\to Ult(V,\dot{G})$ is an elementary embedding with
$j_{\dot{G}}(\omega_2)=\delta$. 
Then $\UU_\delta$ is strongly presaturated.
\end{lemma}
\begin{proof}
Let for $X\in V_\delta$
\[
I^{\dot{G}}_X=\{S\subset P(X):
\llbracket j_{\dot{G}}[X]\in j_{\dot{G}}(S)\rrbracket_{\BBB(\UU_\delta)}=0_{\BBB(\UU_\delta)}\}
\]
and 
\[
\II(\dot{G})=\{I^{\dot{G}}_X:X\in V_\delta\}.
\]

By Lemma~\ref{lem:presatequiv} it is enough to show the following:
\begin{claim}
For all $\QQ\in\UU_\delta$
\[
\QQ=\llbracket T_{\QQ}\in\dot{G}\rrbracket_{\BBB(\UU_\delta)}.
\]
\end{claim}
For suppose the claim holds, then we have that the map 
\[
i:\TT^{\II(\dot{G})}_\delta\to \BBB(\UU_\delta)
\]
which maps 
\[
S\mapsto\llbracket j_{\dot{G}}[\cup S]\in j_{\dot{G}}(S)\rrbracket_{\BBB(\UU_\delta)}
\]
extends to an injective homomorphism of complete boolean algebras with a dense image, i.e.
to an isomorphism.
In particular this gives at the same time that $\TT^{\II(\dot{G})}_\delta$ is a presaturated
tower forcing (since it preserve the regularity of $\delta$) which is
equivalent to $\UU_\delta$ and that 
\[
\{T_{\BB}:\BB\in\UU_\delta\}
\]
is a dense subset of $\TT^{\II(\dot{G})}_\delta$ since its image is $\UU_\delta$ which is, by definition,
a dense subset of $\BBB(\UU_\delta)$.
However 
\[
\{T_{\BB}:\BB\in\UU_\delta\}
\]
is a family of fixed points of $\TT^{\II(\dot{G})}_\delta$ which gives that
$\TT^{\II(\dot{G})}_\delta$ is strongly presaturated.
So we prove the claim:
\begin{claimproof}
Let $H$ be $V$-generic for $\UU_\delta$
and $G=\val_H(\dot{G})$. 
We show the following:
\begin{subclaim}
$\BB\in H$ if and only if $T_{\BB}\in G$.
\end{subclaim} 
\begin{subclaimproof}
First assume that $\BB\in H$. 
Then we have that there is in $V_\delta[H]$ a $V$-generic filter $H_0$ for
$\BB$ such that $H_{|\BB|^+}[H_0]$ computes correctly the 
stationarity of its subsets of $\omega_1$ by Lemma~\ref{lem:chargenericityUdelta}.

Since $j_G(\omega_2)=\delta$ and $Ult(V,G)^{<\delta}\subset Ult(V,G)$, we have that 
\[
H_{\omega_2}^{V[H]}=V_\delta[H]=H_{\omega_2}^{Ult(V,G)}.
\]
Since $\BB\in V_\delta$, $H_0\subset\BB\in  V_\delta[H]=H_{\omega_2}^{Ult(V,G)}$.
Thus $H_0\in Ult(V,G)$ and
$Ult(V,G)$ models that $j_G[H_{|\BB|^+}]\in T_{j(\BB)}=j(T_{\BB})$ which occurs if and only if
$T_{\BB}\in G$.

Conversely assume that $T_{\BB}\in G$. We get that $j_G[H_{|\BB|^+}]\in T_{j(\BB)}$ and thus that
in $Ult(V,G)\subset V[H]$ there is a $V$-generic filter $H_0$ for
$\BB$ such that $H_{|\BB|^+}[H_0]$ computes correctly the 
stationarity of its subsets of $\omega_1$, since 
\[
H_{\omega_2}^{V[H]}=V_\delta[H]=H_{\omega_2}^{Ult(V,G)},
\]
we get that
$H_0\in V[H]$ is a correct $V$-generic filter for $\BB$. We conclude that $\BB\in H$ by 
Lemma~\ref{lem:chargenericityUdelta}.
\end{subclaimproof}
The claim is proved.
\end{claimproof}

The Lemma is proved in all its parts.
\end{proof}

Finally we cannot infer rightawat that a $\TT\in \SPT$ is totally rigid
since we are not able to exclude the case
that there could be two distinct correct $i_0:\TT\to\RR$ and $i_1:\TT\to\RR$ for some 
$\RR\in\SSP$ such that $T_{\RR}$ is non-stationary. However we can show that
$\MM^{++}$ entails that any $\SPT$ tower is totally rigid
by the following:
\begin{lemma}\label{lem:totrigSPTMM++}
Assume $\MM^{++}$ holds and there are class many Woodin cardinals.
Then any $\TT\in\SPT$ is totally rigid.
\end{lemma}
\begin{proof}
Immediate by Lemmas~\ref{lem:eqtrbis},~\ref{lem:basspt1}.
\end{proof}

\subsection{$\MM^{+++}$ and the presaturation of $\UU_\delta$}

\begin{theorem}\label{thm:udelpst}
Assume that any element of 
$\SPT$ is totally rigid,
$\UU_\delta\in\SSP$ preserves the regularity of $\delta$
and that
$\SPT\cap V_\delta$ is dense in $\UU_\delta$.
Then
$\UU_\delta\in\SPT$ as well.
\end{theorem}

\begin{proof}

Let for each $\QQ\in \UU_\delta$, 
$\RR(\QQ,\gamma)\leq\QQ$ be a strongly presaturated normal tower of height
$\gamma$ in $\UU_\delta$.

Let $H$ be $V$-generic for $\UU_\delta$ and 
$G$ be the set of $T_{\BB}$ for $\BB\in H$.
We aim to show that the upward closure of
$G$ generates a $V$-normal tower of ultrafilters $\bar{G}$ on $\delta$ and that
$Ult(V,\bar{G})$ is $<\delta$-closed in $V[H]$.
This suffices to prove the theorem by Lemma~\ref{lem:presat-strpresat}.

First we observe that 
for any $\BB\in H$ the set of $\RR(\QQ,\gamma)\in D=\SPT\cap V_\delta$ such that 
$\BB\geq \RR(\QQ,\gamma)$ is dense below $\BB$
and thus $H$ meets this dense set. In particular we get that
\[
A=\{\gamma: \exists\RR(\QQ,\gamma)\in D\cap H\}
\]
is unbounded in $\delta$. 
Notice that:
\begin{claim} 
For all $\gamma\in A$,
$G\cap V_\gamma$ generates the unique correct 
$V$-generic filter for $\RR(\QQ,\gamma)$ in $V[H]$.
\end{claim}
\begin{claimproof}
Notice that for all $\BB\in V_\gamma$, $T_{\BB}, T_{\RR(\QQ,\gamma)}\in G$ gives that
$T_{\BB}\wedge  T_{\RR(\QQ,\gamma)}$ is stationary.
Now by fact~\ref{fac:statfixpoint} 
we get that
\[
T_{\RR(\QQ,\gamma)\restriction T_{\BB}}= T_{\BB}\wedge  T_{\RR(\QQ,\gamma)}
\]
and that these sets are stationary iff $T_{\BB}$ is a positive element of
the forcing $\RR(\QQ,\gamma)$.
This gives that
$T_{\RR(\QQ,\gamma)\restriction T_{\BB}}\in G$ for all $\BB\in H\cap V_\gamma$.

Now observe that by assumption $\RR(\QQ,\gamma)$ is totally rigid, thus
the unique correct 
embedding $i$ of $\RR(\QQ,\gamma)$ into $\UU_\delta\restriction \RR(\QQ,\gamma)$ is given by
$S\mapsto \RR(\QQ,\gamma)\restriction S$.
Thus for all $\BB\in V_\gamma$,
$\BB\in H$ iff $T_{\BB}\in G\cap V_\gamma$, thus
$i^{-1}[H]= G\cap V_\gamma$ generates a correct $V$-generic filter for $\RR(\QQ,\gamma)$. 
Finally the strong presaturation of $\RR(\QQ,\gamma)$ grants that 
$G\cap V_\gamma$ generates the unique such.
\end{claimproof}

Since $G=\bigcup_{\gamma\in A}G\cap V_\gamma$ and 
$G\cap V_\gamma$ generates a $V$-normal tower of 
ultrafilters on $\gamma$ for all $\gamma\in A$, 
we get that 
$G\in V[H]$ and its upward closure $\bar{G}$ is a $V$-normal tower of ultrafilters for $\delta$. 

Now we apply Fact~\ref{fac:sptotp} to each $\RR(\QQ,\gamma)\in H$
and we get that for all
$h:P(X)\to\omega_2$ in $V$, and $\gamma\in A$ such that
$X\in V_\gamma$, we have that 
for all $M\in T_{\RR(\QQ,\gamma)}$, $h(M\cap X)<\otp(M\cap\gamma)$.

This gives in particular that for all
$h:P(X)\to\omega_2$ in $V$, and 
$X\in V_\delta$, we have that 
$[h]_{\bar{G}}<\delta$, since $T_{\RR(\QQ,\gamma)}\in G$ witnesses that
$[h]_{\bar{G}}<\gamma$ 
for all $h:P(X)\to \omega_2$ in $V_\gamma$.

All in all we get that
$j_{\bar{G}}:V\to Ult(V,\bar{G})$ is such that
$j_{\bar{G}}(\omega_2)=\delta$ and $Ult(V,\bar{G})^{<\delta}\subset Ult(V,\bar{G})$.

Since $G\in V[H]$ gives a $V$-normal tower of ultrafilters on $\delta$
with $j_G:V\to Ult(V,G)$ such that $j_G(\omega_2)=\delta=\omega_2^{V[H]}$,
we can conclude that $\UU_\delta$ is a presaturated normal tower by Lemma~\ref{lem:presat-strpresat}.
\end{proof}

\begin{corollary}\label{thm:eqMM+++}
Assume there are class many supercompact cardinals $\delta$ which are limit of $<\delta$-supercompact
cardinals.
The following are equivalent:
\begin{enumerate}
\item\label{thm:eqMM+++-1}
$\SPT$ is a dense subclass of $\UU^{\SSP}$ consisting of totally rigid posets.
\item\label{thm:eqMM+++-2}
$\UU_\delta$ is strongly presaturated for all 
$\Sigma_2$-reflecting cardinals $\delta$ which are limit of $<\delta$-supercompact
cardinals.
\end{enumerate}
\end{corollary}

\begin{proof}
We prove just the non-trivial direction.
Assume $\delta$ is a $\Sigma_2$-reflecting cardinal.
The statement $\phi(i,\QQ,\II,T_{\TT},\delta)$:
\begin{quote}
\emph{$i:\QQ\to\TT=\TT^{\II}_\delta$ is an $\SSP$-correct homomorphism and $\TT$
is a strongly presaturated tower of height $\delta$}
\end{quote} 
is a 
statement formalizable in $H_{\theta}$ for any $\theta>2^\delta$ using the
parameters $i,\QQ,\II,T_{\TT^{\II}_\delta}$, since the statement 
\emph{$\TT$
is a strongly presaturated tower} 
can be phrased as
\begin{quote}
$\delta$ is inaccessible and
$\II=\{I_X:X\in V_\delta\}$ is a tower of normal filters of height $\delta$
and $T_{\TT^{\II}_\delta}$ is stationary and 
$S\wedge T_{\TT^{\II}_\delta}=T_{\TT^{\II}_\delta\restriction S}$ for densely many
$S\in\TT^{\II}_\delta$.
\end{quote}
Moreover the statement \emph{$i:\QQ\to\TT$ is $\SSP$-correct} is also
definable in the parameters $i,\QQ,\TT$ in the structure
$H_{\theta}$
using the relevant parameters.
Now assume $\QQ\in \UU_\delta$. Then the statement
\[
\exists\,H_{\theta}\,[H_{\theta}\models\exists \TT,\delta,i,\II\,\phi(i,\QQ,\II,T_{\TT},\delta)]
\]
is a $\Sigma_2$-property in the parameter $\QQ$ and holds in $V$, thus it
holds in $V_\delta$, since $\delta$ is $\Sigma_2$-reflecting.
In particular this gives that $\SPT\cap V_\delta$ is dense in $\UU_\delta$.
Now by Theorem~\ref{thm:univ1} we also have that $\UU_\delta\in\SSP$ if
$\delta$ is an inaccessible limit of $<\delta$-supercompact cardinals.
We can now apply Theorem~\ref{thm:udelpst} to get that
$\UU_\delta\in\SPT$, completing the proof.
\end{proof}

\subsection{$\MM^{+++}$ and generic absoluteness of the theory of $L([\Ord]^{\leq\aleph_1})$}
Once we are able to infer that $\UU_\delta$ is a presaturated tower for any 
$\Sigma_2$-reflecting cardinal $\delta$ which is a limit of $<\delta$-supercompact cardinals,
the generic aboluteness results is an easy consequence of Theorem~\ref{lem:univ4}.

\begin{theorem}\label{cor:absres}
Assume $\MM^{+++}$ holds and there are class many $\Sigma_2$-reflecting cardinals $\delta$ 
which are limits of $<\delta$-supercompact 
cardinals.
Then 
\[
\langle L([\Ord]^{\leq\aleph_1})^V, P(\omega_1)^V,\in\rangle\equiv
\langle L([\Ord]^{\leq\aleph_1})^{V^P}, P(\omega_1)^V,\in\rangle
\]
for all $P$ which are stationary set preserving and preserve $\MM^{+++}$.
\end{theorem}

We leave to the reader to convert a proof of this result in a proof of Theorem~\ref{thm:mainthm}.
\begin{proof}

Let $\delta$ be a $\Sigma_2$-reflecting cardinal which is a limit of $<\delta$-supercompact cardinals. 
Let $H$ be $V$-generic for $\UU_\delta$.
By Theorem~\ref{thm:eqMM+++} 
$\UU_\delta$ is forcing equivalent to a strongly presaturated normal tower. Thus in
$V[H]$ we can define an elementary $j:V\to M$ such that $\crit(j)=\omega_2$ and 
$M^{<\delta}\subset M$.
In particular we get that
\[
L([\Ord]^{\leq\aleph_1})^{M}=L([\Ord]^{\leq\aleph_1})^{V[H]}.
\]
Thus
\[
\langle L([\Ord]^{\leq\aleph_1})^V, P(\omega_1)^V,\in\rangle\equiv
\langle L([\Ord]^{\leq\aleph_1})^{V[H]}, P(\omega_1)^V,\in\rangle.
\]

Now observe that if $P\in \UU_\delta$ forces $\MM^{+++}$ and $G$ is $V$-generic for $P$
we have that $\delta$ is still a $\Sigma_2$-reflecting cardinal which is a limit of
$<\delta$-supercompact cardinals in $V[G]$. We can first appeal to Lemma~\ref{lem:univ4}
to find:
\begin{itemize}
\item 
$\QQ\in\UU_\delta\restriction P$,
\item
$i_0:\BBB(P)\to\QQ$ a regular embedding,
\item
$i_1:\BBB(P)\to\UU_\delta\restriction\QQ$ also a regular embedding,
\end{itemize}
such that whenever $G$ is $V$-generic for $P$,
then
$V[G]$ models that the forcing
\[
(\UU^{\SSP}_\delta\restriction\QQ)^V/i_1[G]
\]
is identified with the forcing 
\[
(\UU^{\SSP}_\delta)^{V[G]}\restriction(\QQ/i_0[G])
\]
as computed in $V[G]$.

Now let $H$ be $V$-generic for $\UU_\delta\restriction \QQ$ and
 $G=i_1^{-1}[H]$ be $V$-generic for $P$.
 Then we have that
 $H/i_1[G]$ is $V[G]$-generic for 
 \[
 (\UU_\delta)^{V[G]}\restriction (\QQ/i_0[G])
 \] 
 as computed in $V[G]$ and
that $\delta$ is a $\Sigma_2$-reflecting cardinal which is a limit of
$<\delta$-supercompact cardinals in $V[G]$ as well. 
In particular by applying the above fact in $V$ and in $V[G]$ which are both models of 
$\MM^{+++}$
where $\delta$ is a $\Sigma_2$-reflecting cardinal which is a limit of
$<\delta$-supercompact cardinals we get that
 \[
\langle L([\Ord]^{\leq\aleph_1})^V, P(\omega_1)^V,\in\rangle\equiv
\langle L([\Ord]^{\leq\aleph_1})^{V[H]}, P(\omega_1)^V,\in\rangle.
\]
and
 \[
\langle L([\Ord]^{\leq\aleph_1})^{V[G]}, P(\omega_1)^{V[G]},\in\rangle\equiv
\langle L([\Ord]^{\leq\aleph_1})^{V[H]}, P(\omega_1)^{V[G]},\in\rangle.
\]
The conclusion follows.
\end{proof}

\subsection{Consistency of $\MM^{+++}$}
We now turn to the proof of the consistency of $\MM^{+++}$.
The guiding idea is to turn all the known proofs of the consistency of $\MM^{++}$ 
by means of iterations of length $\delta$ which collapse $\delta$ to $\omega_2$
in methods
to prove the consistency of $\MM^{+++}$ by just assuming stronger large cardinal properties of $\delta$,
i.e. super (almost) hugeness of $\delta$.
We shall give a fast and detailed proof and a more meditated and general (but sketchy) one.

\subsubsection{Fast proof of the consistency of $\MM^{+++}$}

\begin{lemma}
Assume $j:V\to M\subset V$ is an almost huge embedding with $\crit(j)=\delta$.
Assume that $G$ is $V$-generic
for $\UU_\delta$.

Then in $V[G]$ 
$\UU_{j(\delta)}$ is a strongly presaturated tower of normal filters.

\end{lemma}
\begin{proof}
The almost hugeness of $j$ grants that $j(\delta)$ is an inaccessible limit of $<j(\delta)$-supercompact
cardinals, since $M$ models that $j(\delta)$ is such, $V_{j(\delta)}\subset M$ and the closure of 
$M$ grant that $j(\delta)^{<j(\delta)}\subset M$, so any sequence in $V$ cofinal to $j(\delta)$
cannot have order type less than $j(\delta)$ otherwise it would be in $M$ contradicting the
inaccessibility of $j(\delta)$ in $M$.

Let $G$ be $V$-generic for $\UU_\delta$.
Then in $V[G]$ $j(\delta)$ is still an inaccessible limit of $<j(\delta)$-supercompact
cardinals. By Theorem~\ref{thm:univ1}, $\UU_{j(\delta)}^{V[G]}$ is in $V[G]$ an $\SSP$ poset.
By Theorem~\ref{lem:univ4}, $\UU_{j(\delta)}^{V[G]}$ is isomorphic with
$(\UU_{j(\delta)}\restriction \UU_\delta)^V/G$ in $V[G]$.
Moreover 
\[
\UU_{\delta}\geq_{\SSP}\UU_{j(\delta)}\restriction \UU_\delta
\] 
as witnessed by the map
$\BB\mapsto \UU_\delta\restriction\BB$.
Now let $K$ be $V[G]$-generic for $\UU_{j(\delta)}^{V[G]}$,
$H_0=\{\BB:[\BB]_G\in K\}$ and
\[
H=\uparrow H_0=\{\QQ\in V_{j(\delta)}:\exists \BB\in H_0, \QQ\geq_{\SSP}\BB\}.
\]
Then $H$ is $V$-generic for $(\UU_{j(\delta)}\restriction \UU_\delta)^V$ and
$j[G]=G\subset H$.
Thus we can apply Lemma~\ref{lem:forcpresat}
and extend $j$ to an elementary
$\bar{j}:V[G]\to M[H]$ letting $\bar{j}(\val_G(\tau))=\val_H(j(\tau))$.
Lemma~\ref{lem:forcpresat} grants that
in $V[H]$ $\bar{j}$ is definable
and $M[H]^{<\bar{j}(\delta)}\subset M[H]$. 

Thus in $V[G]$:
\begin{itemize}
\item 
By Lemma~\ref{lem:presatequiv}
$\UU_{j(\delta)}^{V[G]}$ is forcing equivalent to 
a presaturated normal tower,
as witnessed by the tower of $V[G]$-normal ultrafilters 
\[
\{S\in V_{j(\delta)}[G]:\bar{j}[\cup S]\in \bar{j}(S)\}
\] 
living in $V[H]$.
\item
By Lemma~\ref{lem:presat-strpresat}, 
we get that $\UU_{j(\delta)}^{V[G]}$ is
in $V[G]$ a strongly presaturated tower.
\end{itemize}
\end{proof}

\begin{corollary}\label{cor:MMM-easy}
Assume $\delta$ is super almost huge. Then 
$\UU_\delta$ forces $\MM^{+++}$.
\end{corollary}
\begin{proof}
Let $G$ be $V$-generic for $\UU_\delta$.
The family
\[
\{(\UU_{j(\delta)}\restriction\BB)^{V[G]}:\BB\in \UU_{j(\delta)}^{V[G]},\, j:V\to M\subset V\text{ is 
super almost huge with $\crit(j)=\delta$}\}
\]
witnesses that the class of strongly presaturated towers is dense in $V[G]$. 
\end{proof}

\begin{remark}
Notice that (in view of Theorem~\ref{thm:udelpst}) in $V[G]$  
the class of strongly presaturated towers includes $\UU_\gamma^{V[G]}$ 
also for all $\gamma$ which are
not equal to $j(\delta)$ for some almost huge $j$ but which are still $\Sigma_2$-reflecting 
cardinals which are 
limits of $<\gamma$-supercompact cardinals.
It is the largeness of the family of $\gamma$ for which we can predicate in models
of $\MM^{+++}$ that $\UU_\gamma$ is strongly presaturated that allows to run the proof of
Corollary~\ref{cor:absres}: Assume $V,V[G]$ are both models of $\MM^{+++}$ with $V[G]$ 
a generic extension of $V$ by an $\SSP$ partial order.
Assume on the other hand that the class of
$\gamma$ such that $\UU_\gamma$ is strongly presaturated in $V$ and the class of
$\eta$ such that $\UU_\eta^{V[G]}$ is strongly presaturated in $V[G]$ have bounded intersection.
Then the proof of Corollary~\ref{cor:absres} would break down.
The fact that these two classes have an unbounded overlap is essential and is the content of
Theorem~\ref{thm:udelpst}.
\end{remark}

\subsubsection{A template for proving the consistency of $\MM^{+++}$}
We want to show that the consistency proof of $\MM^{+++}$ does not depend 
on the particular choice we made among the poset that can be used to obtain a model of $\MM^{++}$
(in the previous proof we chose category forcing). 
So we shall devise a template to establish the consistency of $\MM^{+++}$
which can be applied to potentially all forcings which can prove the consistency of $\MM^{++}$.
To this aim, we need to introduce a slight strengthening of total rigidity. Since the arguments of this section
are not essential to establish our main results but are just used to provide a deeper insight on the 
nature of the forcings that can produce a model of $\MM^{+++}$ we feel free to omit all proofs.

\subsubsection*{Rigidly layered partial orders}

\begin{definition}\label{def:riglayprop}
Given an atomless complete boolean algebra $\BB\in\SSP$, let
$\delta_{\BB}$ be the least size of a dense subset of $\BB^+=\BB\setminus\{0_{\BB}\}$.

$\BB$ is rigidly layered if there is a family
$\{\QQ_p:p\in D\}$
indexed by a dense subset of $\BB$ such that each $\QQ_p$ is a complete subalgebra of 
$\BB\restriction p$ satisfying the following properties
for all $p,q$ in $D$:
\begin{enumerate}
\item\label{def:riglayprop1}
$1_{\QQ_p}=p$, $\QQ_p$ has size less than $\delta_{\BB}$ and is totally rigid.
\item\label{def:riglayprop2}
 $\BB\restriction p\leq_{\SSP}\QQ_q$ iff $p\leq_{\BB} q$
iff the map defined
by $r\mapsto r\wedge_{\BB}p$ is the unique correct homomorphism of $\QQ_q$
into $\QQ_p$.
\item\label{def:riglayprop3}
 $\QQ_p$ is orthogonal to $\QQ_q$ in $\UU^{\SSP}$ iff $p\wedge_{\BB} q=0_{\BB}$.
\end{enumerate}
\end{definition}

\begin{definition}
Let $\BB\in\SSP$ be a complete boolean algebra.
A family
$\{\BB_\alpha:\alpha<\delta_{\BB}\}$ of complete
subalgebras of $\BB$ is a linear rigid layering of $\BB$ if
for all $\alpha<\delta_{\BB}$:
\begin{enumerate}
\item
$\BB_\alpha$ has size less than $\delta$ and is totally rigid,
\item
$\BB_\alpha\subset\BB_\beta$ and
the inclusion map of $\BB_\alpha$ in $\BB_\beta$ is a regular embedding
witnessing that $\BB_\alpha\geq_{\SSP}\BB_\beta$.
\item
$\bigcup_{\alpha<\delta}\BB_\alpha$ is dense in $\BB$. 
\item
Letting $\alpha_p$ be the least $\alpha$ such that $p\in \BB_\alpha$,
we have that the map $p\mapsto\alpha_p$ is order reversing.
\end{enumerate}
\end{definition}

\begin{lemma}
Assume $\{\QQ_\alpha:\alpha<\delta_{\BB}\}$ is a linear rigid layering of $\BB$
with union $D$. For each $p\in D$ let $\QQ_p=\QQ_{\alpha_p}\restriction p$.
Then $\{\QQ_p:p\in D\}$ is a rigid layering of $\BB$.
\end{lemma}

The usual forcings which force $\MM^{++}$ are rigidly layered:
\begin{enumerate}
\item
Assume $\MM^{++}$ holds,
$\UU_\delta\in \SSP$, $\delta$ is inaccessible,  
and the set of totally rigid partial orders is dense
in $\UU_\delta$. We show that $\UU_\delta$ is rigidly layered.

$\UU_\delta$ can be identified with the partial order 
$Rep(\UU_\delta)=\{T_{\BB}\wedge T_{\UU_\delta}:\BB\in\UU_\delta\}$ 
with the ordering given by the usual ordering on stationary sets.
Notice that this partial order is isomorphic to the separative
quotient of $\UU_\delta$.
The family
\[
\{\{T_{\BB\restriction b}\wedge T_{\UU_\delta}:b \in\BB\}:\BB\in\UU_\delta\text{ is totally rigid}\}
\] 
defines a rigid layering of $\BBB(Rep(\UU_\delta))$.

\item
The standard RCS iteration\footnote{We refer the reader to 
Subsection~\ref{subsec:itfor} for the relevant definitions and results on iterations
and to~\cite{VIAAUDSTE13} for a detailed account.}
\[
\FFF=\{i_{\alpha,\beta}:\BB_\alpha\to\BB_\beta:\alpha\leq\delta\}
\] 
of length $\delta$ which uses a Laver function $f:\delta\to V_\delta$ to 
decide $\BB_{\alpha+1}$ according to the value of $f(\alpha)$
shows that
$C(\FFF)$ forces that $\MM^{++}$ is rigidly layered.
To see this observe that $\BB_\alpha\subset V_\alpha$ is totally rigid
for stationary many $\alpha<\delta$, and for all $\alpha<\beta$ 
$i_{\alpha,\beta}$ witnesses that $\BB_\alpha\geq_{\SSP} \BB_\beta$.
Let $\QQ_\alpha$ be the boolean completion of
\[
\{f\in C(\FFF):\text{ the support of $f$ is at most $\alpha$}\}.
\]
We get that the inclusion map 
is the unique correct regular embedding of 
$\QQ_\alpha$ into $\QQ_\beta$ for all $\alpha$ such that $\QQ_\alpha$ is totally rigid.
We leave to the reader to check that the family
\[
\{\QQ_\alpha:\alpha<\delta,\text{ $\QQ_\alpha$ is totally rigid}\}
\]
is a linear rigid layering of $C(\FFF)$.
\item
In a similar way we can define a linear rigid layering of the lottery preparation forcing to force
$\MM^{++}$.
\end{enumerate}

\subsubsection*{Rigidly layered presaturated towers are strongly presaturated}

\begin{lemma}~\label{lem:rlf->SPT}
Assume that 
$\BB$ is a rigidly layered complete boolean algebra which
is forcing equivalent to a presaturated tower of normal filters of height $\delta$.
Then $\BB$ is forcing equivalent to a strongly presaturated tower of normal filters.
\end{lemma}

\subsubsection*{Linear rigid layerings are inherited by generic quotients}

\begin{lemma}\label{lem:riglayquo}
Assume $\BB$ is a linearly rigidly layered complete boolean algebra
as witnessed by
$\{\QQ_\alpha:\alpha<\delta_{\BB}\}$.
Let
$D=\bigcup\{\QQ_\alpha:\alpha<\delta_{\BB}\}$ and
Assume $\QQ\geq_{\SSP} \QQ_\alpha\restriction p$ for some
$p\in D$ as witnessed by a correct homomorphism $k:\QQ\to \QQ_\alpha\restriction p$.
We can extend $k:\QQ\to \BB\restriction p$ ``composing'' it with the inclusion of
$\QQ_\alpha\restriction p$ into $\BB\restriction p$.

Let $G$ be $V$-generic for $\QQ$ and such that $0_{\BB}\not\in k[G]$.
Then $\BB\restriction p/k[G]$ 
is rigidly layered in $V[G]$.
\end{lemma}

\begin{remark}
Also rigid layerings of $\BB$ are inherited by generic quotients of $\BB$.
The proof of this fact is slightly more elaborate and we won't need it. 
\end{remark}

\subsubsection*{Other strongly presaturated towers}

The hypotheses on $\FFF$ of the following propositions 
are satisfied by the standard forcings 
of size $\delta$ which 
produce a model of $\MM^{++}$ collapsing a super almost huge cardinal $\delta$
to become $\omega_2$.

\begin{lemma}
Assume $j:V\to M\subset V$ is elementary with $\crit(j)=\delta$ and $M^{<j(\delta)}\subset M$.
Let $\FFF=\{i_{\alpha,\beta}:\alpha<\beta<\delta\}$ be a semiproper iteration system
contained in $V_\delta$
such that for stationarily many $\alpha$ 
\[
\RCS(\FF\restriction\alpha)=C(\FFF\restriction\alpha)
\] 
is totally rigid.

Then $C(\FFF)$ and $C(j(\FFF))$ are both in $\SSP$ 
and $j\restriction C(\FFF):C(\FFF)\to C(j(\FFF))$ is a
regular embedding.

Assume finally that $G$ is $V$-generic
for $C(\FFF)$.
Then in $V[G]$ we get that
$C(j(\FFF))/j[G]$ is a strongly presaturated tower of normal filters.

\end{lemma}

\begin{theorem}\label{thm:CONMM+++}
Assume $\delta$ is super almost huge.
Let 
\[
\FFF=\{i_{\alpha,\beta}:\alpha<\beta<\delta\}
\] 
be a semiproper iteration system
contained in $V_\delta$
such that: 
\begin{itemize}
\item
$C(\FFF)$ is linearly rigidly layered.
\item
For all $\QQ\leq_{\SSP} C(\FFF)$ in $\SSP$ there is 
$j:V\to M$ such that:
\begin{enumerate}
\item\label{thm:conMM+++closM}
$M^{<j(\delta)}\subseteq M$, 
\item
$\crit(j)=\delta$,
\item
$\QQ\geq_{\SSP}C(j(\FFF))\restriction q$ for some $q\in C(j(\FFF))$,
\item 
$\QQ\in V_{j(\delta)}$.
\end{enumerate}
\end{itemize}
Then,
whenever $G$ is $V$-generic for $C(\FFF)$, we have that
$V[G]$ models $\MM^{+++}$ as witnessed by the family of forcings
$C(j(\FFF))\restriction q/j[G]\in \SPT^{V[G]}$ as
$j:V\to M$ ranges among the witnesses of the super almost hugeness of $\delta$
and $q\in C(j(\FFF))$.
\end{theorem}

\begin{remark}
Theorem~\ref{thm:CONMM+++} establishes that the consistency of 
$\MM^{+++}$ can be obtained using 
Hamkins lottery preparation forcing or the standard iteration for producing a model of $\MM^{++}$
guided by a Laver function for a super almost huge cardinal. 
By~\cite[Theorem 12, Fact 13]{COX12}, this is 
a large cardinal assumption consistent relative to the existence of a $2$-huge cardinal.
\end{remark}

\section{Some comments}\label{sec:com}
We can sum up the results of this paper (as well as some other facts about forcing axioms) 
as density properties of the category $\UU^{\SSP}$ as follows:
\begin{itemize}
\item
$\ZFC+$\emph{ there are class many supercompact cardinals}
implies that the class of totally rigid partial orders which force
$\MM^{++}$ is dense in $\UU^{\SSP}$.
\item
$\ZFC+$\emph{ $\delta$ is an inaccessible limit of $<\delta$supercompact cardinals}
implies that 
\[
\UU^{\SSP}\cap V_\delta=\UU_\delta\in\SSP.
\]
\item
$\MM^{++}$ is equivalent (over the theory $\ZFC+$\emph{ there are class many Woodin cardinals})
to the statement that the class of presaturated towers is dense in the category $\UU^{\SSP}$.
\item
$\MM^{+++}$ asserts 
that the class of strongly presaturated towers is dense in the category $\UU^{\SSP}$.
\item
$\ZFC+$\emph{ there is an almost superhuge cardinal}
implies that the class of rigidly layered partial orders which force
$\MM^{+++}$ is dense in $\UU^{\SSP}$.
\item
$\ZFC+\MM^{+++}+$\emph{ $\delta$ is a $\Sigma_2$-reflecting cardinal which is a 
limit of $<\delta$ supercompact cardinals}
implies that $\UU_\delta$ is forcing equivalent to a presaturated tower.
\end{itemize}
In particular it appears that the categorial framework we introduced is particularly well suited to
express strong forcing axioms as density properties of the category $\UU^{\SSP}$.
This approach is being pursued further in~\cite{VIAAUD14} where we 
merge this categorial approach 
with the researches of Hamkins and Johnstone\cite{HAMJOH13} and of 
Tsaprounis~\cite{TSA13} on resurrection axioms and their unbounded versions.

Regarding the consistency strength of our results,
it is likely that supercompactness is not sufficient to get the consistency of $\MM^{+++}$.
The problem is the following:
Assume $P_\delta\subset V_\delta$ forces $\MM^{++}$ and collapses $\delta$ to become 
$\omega_2$. 
Assume $\delta$ is supercompact but not 
almost superhuge, then for any $j:V\to M$ such that $M^{<j(\delta)}\not\subseteq M$ we have 
no reason to expect that $j(P_\delta)$ is stationary set preserving in $V$.
We can just prove that it is stationary set preserving in $M$. 
On the other hand if $P_\delta\subseteq V_\delta$ is stationary set preserving and
$j:V\to M$ is an almost huge embedding (i.e. $M^{<j(\delta)}\subset M$), then $j(P_\delta)$ is
stationary set preserving in $V$ and we can argue that if $G$ is $V$-generic for 
$P_\delta$, than $j(P_\delta)/G\in V[G]$ is a strongly presaturated normal tower.  
This crucial difference
suggests why $\MM^{+++}$ is likely to have a stronger consistency strength than $\MM^{++}$.
On the other hand it can be checked that
all the forcings considered in this paper to obtain the consistency of 
$\MM^{+++}$ collapsing an inaccessible $\delta$ to become $\omega_2$ satisfy the
$\delta$-covering and $\delta$-approximation property (see~\cite[Def. 4.5]{VIAWEI11}). In particular 
by the results of~\cite{VIAWEI11}, we can infer that such a $\delta$ must at 
least a strongly compact cardinal. However this conclusion can already be inferred for the models of 
$\MM^{++}$ obtained by such forcings and in the present stage we are not able to extract any further 
indication regarding the consistency strength of $\MM^{+++}$ with respect to that of $\MM^{++}$.
It seems that we lack a combinatorial characterization of super almost hugeness in the same fashion 
of the one provided by the work of Jech, Magidor, and Weiss for supercompactness 
and strong compactness.

We want also to remark that 
the work of
Larson~\cite{larson-dwo} and Asper\'o~\cite{ASPPFA++} shows that our result are 
close to optimal and that we cannot hope to prove 
Theorem~\ref{thm:mainthm} for forcing axioms which are strictly
weaker than $\MM^{++}$:
\begin{itemize}
\item
Larson showed 
that there is a 
$\Sigma_3$ formula $\phi(x)$ such that over any model $V$ of
$\ZFC$ with large cardinals and for any $a\in H_{\omega_2}^V$
there is a \emph{semiproper} 
forcing extension $V[G]$ of $V$ which models $\MM^{+\omega}$ 
 and such that
$a$ is the unique set which satisfies $\phi(a)^{H_{\omega_2}^{V[G]}}$.
\item
Asper\'o showed 
that there is a 
$\Sigma_5$ formula $\psi(x)$ such that over any model $V$ of
$\ZFC$ with large cardinals and for any $a\in H_{\omega_2}^V$
there is a \emph{semiproper} forcing extension $V[G]$ of $V$ which models $\PFA^{++}$ 
 and such that
$a$ is the unique set which satisfies $\phi(a)^{H_{\omega_2}^{V[G]}}$.
\end{itemize}
These results show that the theory of $H_{\omega_2}$ in models of $\MM^{+\omega}$
(respectively $\PFA^{++}$) cannot be generically  
invariant with respect to $\SSP$-forcings which preserve 
these axioms, since we would get otherwise that all elements of 
$H_{\aleph_2}$ could be defined as the unique
objects satisfying $\phi$ (respectively $\psi$).
It remains nonetheless open whether the results we got on the forcing $\UU^{\SSP}$
can be declined for other category forcings given by suitable classes of forcings $\Gamma$. 
We conjecture that this is the case for many such $\Gamma$ among which the proper posets.
This requires to investigate for which class of forcings we can predicate the freezeability property,
since ultimately all the properties we were able to infer for $\UU^{\SSP}$ were obtained appealing to:
\begin{itemize}
\item closure of $\leq_{\SSP}^*$ under set sized descending sequences (obtained by identifying
$\SSP$ with $\SP$),
\item closure of $\SSP$ under two step iterations,
\item closure of $\SSP$ under set sized lottery sums,
\item a simple definition in terms of first order logic of the class $\SSP$,
\item the freezeability property.
\end{itemize}
Of the above list the unique property as yet not known to hold
for many other interesting categories of forcing notions
is the freezeability property. If we are able to infer such a property for other classes of forcings we are 
confident that the appropriate generic absoluteness result for the appropriate version of $\MM^{+++}$
 declined for these categories is at reach.

Finally the theory of $L(\mathbb{R})$ in the context of large cardinals is generically invariant and 
among other things this has led to the development of the rich theory of universally Baire subsets
of $\mathbb{R}$, sets whose properties are generically invariant and which played an 
important role to understand the theory of $L(\mathbb{R})$ under determinacy axioms.
The direction we want to investigate is that of isolating the correct notion of 
universally Baire subset of $2^{\omega_1}$ and to understand the property of these sets in 
models of $\MM^{+++}$, since for this theory we also have a notion of generic invariance.
Most likely a theory of universally Baire subsets of $2^{\omega_1}$ should complement 
the rich understanding we already have of the theory of $L(P(\omega_1))$ in the presence 
of strong forcing axioms.

\bibliographystyle{amsplain}
\bibliography{matteo-viale-JAMS-biblio}
\end{document}